\pdfoutput=1 
\documentclass[10pt, reqno]{amsart}

\usepackage[T1]{fontenc}
\usepackage[utf8]{inputenc}
\usepackage{lmodern}
\usepackage[USenglish]{babel}
\usepackage{amsmath,amssymb,amsthm,upgreek}
\usepackage{mathrsfs}
\usepackage{aligned-overset}
\usepackage{mathtools}
\usepackage{nicematrix}
\usepackage{colortbl,color} %
\usepackage{subcaption,float,longtable}
\usepackage[bookmarks=true,
            bookmarksnumbered,
            plainpages=false,
            linktocpage,
            colorlinks=true,
            citecolor=green!80!black,
            linkcolor=red!70!black,
            filecolor=magenta,
            urlcolor=magenta,
            hidelinks,
            breaklinks,
            pdfauthor={Paul Geuchen, Palina Salanevich, Olov Schavemaker, Felix Voigtlaender},
            pdftitle={On best approximation by multivariate ridge functions},
            unicode=true]{hyperref}
\usepackage[nameinlink,capitalize,noabbrev]{cleveref}
\usepackage{adforn} %
\usepackage{bbm}
\usepackage{csquotes}
\MakeOuterQuote{"}
\usepackage{comment}
\usepackage{soul}
\usepackage[margin=3cm]{geometry}
\usepackage[style = numeric, url = true, giveninits=true,  sortcites=true,backend = biber, maxbibnames=99]{biblatex}
\AtBeginBibliography{\small}
\AtEveryBibitem{\clearlist{language}}
\AtEveryBibitem{\clearfield{issn}}
\AtEveryCitekey{\clearfield{issn}}

\DeclareBibliographyCategory{needsurl}
\renewbibmacro*{url+urldate}{%
  \ifcategory{needsurl}
    {\printfield{url}%
     \iffieldundef{urlyear}
       {}
       {\setunit*{\addspace}%
        \printurldate}}
    {}}

\usepackage{ifthen}

\newcommand{\ii}{\mathrm{i}}
\newcommand{\ee}{\mathrm{e}}
\newcommand{\dd}{\mathrm{d}}
\newcommand{\NN}{\mathbb{N}}                                     %
\newcommand{\RR}{\mathbb{R}}                                     %
\newcommand{\ZZ}{\mathbb{Z}}                                     %
\newcommand{\QQ}{\mathbb{Q}}                                     %
\newcommand{\CC}{\mathbb{C}}                                   
\newcommand{\PP}{\mathbb{P}}
\newcommand{\EE}{\mathbb{E}}
\newcommand{\KK}{\mathbb{F}}

\renewcommand{\SS}{\mathbb{S}}

\newcommand{\fall}{\:\forall\:}                                  %

\newcommand{\abs}[1]{\left\lvert#1\right\rvert}                  %
\newcommand{\mnorm}[1]{\left\lVert#1\right\rVert}                %
\newcommand{\defeq}{\mathrel{\mathop:}=}                               %

\newcommand{\norel}{\mathrel{\phantom{=}}}                                           %

\newcommand{\iid}{\overset{\mathrm{i.i.d.}}{\sim}}

\newcommand{\elll}{\boldsymbol{\ell}}

\renewcommand{\tilde}{\widetilde}
\newcommand{\kk}{\mathbf{k}}
\newcommand{\prs}{\mathrm{Pr}_s}
\newcommand{\rdl}{\paul{\mathcal{R}_{d,\ell}}}
\newcommand{\rrdl}{\paul{\mathcal{R}^\ast_{d,\ell}}}
\newcommand{\rndl}{\paul{\mathcal{R}^\ast_{n,d,\ell}}}
\newcommand{\rrndl}{\paul{\mathcal{R}_{n,d,\ell}}}
\newcommand{\rrndlq}{\paul{\mathcal{R}_{n,d,\ell}}}
\newcommand{\unif}{\mathrm{Unif}}
\newcommand{\mb}{\mathcal{M}_b}
\newcommand{\sob}{\mathcal{B}(W^{r,p}_d)}
\newcommand{\sobb}{\mathcal{B}(W^{r,\infty}_d)}
\newcommand{\sobbb}{\mathcal{B}(W^{r,p}_d(\CC))}
\newcommand{\profilefunction}{\paul{\varrho}}
\newcommand{\pnorm}[1]{\paull{\mnorm{#1}_{\ell^p}}}
\newcommand{\onenorm}[1]{\paull{\mnorm{#1}_{\ell^1}}}
\newcommand{\twonorm}[1]{\paull{\mnorm{#1}_{\ell^2}}}
\newcommand{\infnorm}[1]{\paull{\mnorm{#1}_{\ell^\infty}}}
\newcommand{\polyy}{\paull{\mathcal{P}}}

\newcommand{\paul}[1]{{\color{black}#1}}
\newcommand{\paull}[1]{{\color{black}#1}}

\newcommand{\metalambda}{%
  \mathop{%
    \rlap{$\lambda$}%
    \mkern2mu
    \raisebox{.275ex}{$\lambda$}%
  }%
}

\DeclareMathOperator{\spann}{span}
\DeclareMathOperator{\supp}{supp}

\DeclareMathOperator{\rk}{rank}
\DeclareMathOperator{\sgn}{sgn}

\DeclareMathOperator{\diag}{diag}
\DeclareMathOperator{\IM}{Im}
\DeclareMathOperator{\bin}{bin}
\DeclareMathOperator{\RE}{Re}
\DeclareMathOperator{\esssup}{ess \ sup}
\let \eps \varepsilon
\let \epsilon \varepsilon

\theoremstyle{plain} 
\newtheorem{theorem}{Theorem}[section]
\newtheorem{corollary}[theorem]{Corollary}

\newtheorem{lemma}[theorem]{Lemma}
\newtheorem{proposition}[theorem]{Proposition}

\theoremstyle{definition} %

\theoremstyle{remark} %
\newtheorem{remark}[theorem]{Remark}
\AtEndEnvironment{remark}{\hfill$\diamond$}

\crefname{theorem}{Theorem}{Theorems}
\crefname{Prop}{Proposition}{Propositions}
\crefname{Lem}{Lemma}{Lemmas}
\crefname{Kor}{Corollary}{Corollaries}
\crefname{Bem}{Remark}{Remarks}
\crefname{Bsp}{Example}{Examples}
\crefname{Def}{Definition}{Definitions}
\crefname{Alg}{Algorithm}{Algorithms}

\numberwithin{equation}{section}

\makeatletter
\renewcommand*{\eqref}[1]{%
  \hyperref[{#1}]{\textup{\tagform@{\ref*{#1}}}}%
}
\makeatother

\makeatletter
\@ifundefined{textcommabelow}{%
  \DeclareTextCommandDefault\textcommabelow[1]
    {\hmode@bgroup\ooalign{\null#1\crcr\hidewidth\raise-.31ex
     \hbox{\check@mathfonts\fontsize\ssf@size\z@
     \math@fontsfalse\selectfont,}\hidewidth}\egroup}%
}{}
\makeatother

\newcommand*{\fres}[2]{ {\left.\kern-\nulldelimiterspace #1 \vphantom{\big|} \right|_{\kern-1pt #2} }}

\usepackage{url}

\begin{document}
\allowdisplaybreaks

\hyphenation{Lip-schitz}

\title[On best approximation by multivariate ridge functions]{On best approximation by multivariate ridge functions \\ with applications to generalized translation networks}

\author{Paul Geuchen}
\address[P. Geuchen]{Mathematical Institute for Machine Learning and Data Science (MIDS),
Catholic University of Eichstätt--Ingolstadt (KU),
Auf der Schanz 49, 85049 Ingolstadt, Germany}
\email{paul.geuchen@ku.de}
\thanks{}

\author{Palina Salanevich}
\address[P. Salanevich]{Mathematical Institute,
Utrecht University,
Budapestlaan 6,
3584 CD Utrecht,
Netherlands}
\email{p.salanevich@uu.nl}
\thanks{}

\author{Olov Schavemaker}
\address[O. Schavemaker]{Mathematical Institute,
Utrecht University,
Budapestlaan 6,
3584 CD Utrecht,
Netherlands}
\email{o.p.schavemaker@uu.nl}
\thanks{}

\author{Felix Voigtlaender}\thanks{}
\address[F. Voigtlaender]{Mathematical Institute for Machine Learning and Data Science (MIDS),
Catholic University of Eichstätt--Ingolstadt (KU), Auf der Schanz 49, 85049 Ingolstadt, Germany}
\email{felix.voigtlaender@ku.de}



\date{\today}

\begin{abstract}    
\paul{In this paper, }we prove sharp upper and lower bounds for the approximation of Sobolev functions by sums of \emph{multivariate ridge functions}, i.e., \paul{for approximation by} functions of the form 
$\RR^d \ni x \mapsto \sum_{k=1}^n \profilefunction_k(A_k x) \in \RR$ with $\profilefunction_k : \RR^\ell \to \RR$ and $A_k \in \RR^{\ell \times d}$. 
We show that the order of approximation asymptotically behaves as $n^{-r/(d-\ell)}$, where $r$ is the regularity \paul{(order of differentiability)} of the Sobolev functions to be approximated. 
Our lower bound even holds when approximating $L^\infty$-Sobolev functions of regularity $r$ with error measured in $L^1$,
while our upper bound applies to the approximation of $L^p$-Sobolev functions in $L^p$ for any $1 \leq p \leq \infty$.
These bounds generalize well-known results \paul{regarding} the approximation properties of \emph{univariate} ridge functions to the multivariate case. 
We use \paul{our results} to obtain sharp asymptotic bounds for the approximation of Sobolev functions using generalized translation networks and
complex-valued neural networks. 
\end{abstract}
\maketitle

\newcommand{\lploc}{L^p_{\mathrm{loc}}}
\section{Introduction}
Ridge functions --- both univariate and multivariate --- and their linear combinations have received 
significant attention in the mathematical literature \cite{pinkus_ridge_2016,ismailov2021ridge,candes2003ridgelets,lin1993fundamentality},
in particular due to their applications to neural networks and related areas \cite{10.1093/imaiai/iaaa036,maiorov1999lower}.
 A (multivariate) ridge function on $\RR^d$ is defined as a composition of the form $\profilefunction \circ A$ with $\profilefunction: \RR^\ell \to \RR$ and $A \in \RR^{\ell \times d}$, where $\ell < d$. 
 In the case when $\ell = 1$, we refer to such a ridge function as univariate. 
 A topic of particular interest is the approximation capacity of linear combinations of ridge functions. 
 Although this question has been extensively
studied in the case of \emph{univariate} ridge functions \cite{pinkus1997approximating,ismailov2013review,maiorov_best_1999,maiorov2010best,gordon2001best}, 
the case of \emph{multivariate} ridge functions has largely remained unexplored until now.

In this paper, we establish asymptotically tight bounds for the approximation of Sobolev functions by sums of multivariate ridge functions,
more precisely by sums of $\ell$-variate ridge functions for any fixed $1 \leq \ell < d$. 
Specifically, we show that for every $L^p$-Sobolev function with smoothness \paul{(order of differentiability)} $r$ defined on the unit ball in $\RR^d$, 
there exists a sum of $n$ multivariate ridge functions such that the error of approximation in the $L^p$-norm is asymptotically 
$n^{-r/(d-\ell)}$, where $p \in [1,\infty]$. 
We furthermore prove that this rate is \emph{optimal}. 

Moreover, we show that our results for multivariate ridge functions
can be used to derive tight asymptotic bounds for the approximation of Sobolev functions by \emph{generalized translation networks} (GTNs) 
and \emph{complex-valued} neural networks (CVNNs).
A (shallow) generalized translation network with activation dimension $\ell$ is a function of the form
\begin{equation}
  \RR^d \ni x \mapsto \sum_{i=1}^n c_i\,\paul{\tau}(A_i\,x + v_i) \in \RR,
  \label{eq:IntroGTN}
\end{equation}
with matrices $A_i \in \RR^{\ell \times d}$, bias vectors $v_i \in \RR^\ell$,
coefficients $c_i \in \RR$ and a fixed activation function $\paul{\tau}: \RR^\ell \to \RR$. 
Note that one obtains classical shallow neural networks in the case $\ell = 1$.
Both generalized translation networks and complex-valued neural networks 
recently have received increased attention from both theoretical and applied perspectives; 
see \cite{bassey2021survey,lee2022complex,mhaskar1995degree,mhaskar1995degree,mhaskar1997choice} and the references therein.

Our results demonstrate that shallow GTNs with higher activation dimensions achieve a strictly better optimal approximation rate than GTNs with a smaller activation dimension. 
Moreover, we show that there exists a (nice) complex activation function $\paul{\phi}:\CC \to \CC$ 
such that shallow CVNNs can approximate Sobolev functions of smoothness $r$ on the unit ball of $\CC^d$ with rate $n^{-r/(2d-2)}$, 
whereas the best rate that can be achieved by shallow \emph{real-valued} neural networks is $n^{-r/(2d-1)}$.
We also show that the rate of $n^{-r/(2d-2)}$ is optimal in the class of all shallow CVNNs. 

\subsection{Related work}
\label{subsec:related}

In this subsection, we discuss how our results relate to the existing literature,
both regarding ridge functions and regarding various classes of neural networks.

\subsubsection{Approximation by sums of univariate ridge functions}

The approximation of Sobolev functions by sums of \emph{univariate} ridge functions was extensively studied by
V.~Maiorov and his co-authors, see \cite{maiorov_best_1999,maiorov2010best,gordon2001best}. 
As a first result in this direction, the sharp bound of $n^{-r/(d-1)}$ was shown
for the approximation of $L^2$-Sobolev functions of regularity $r$ on 
the $d$-dimensional unit ball, with error measured in the $L^2$-norm \cite{maiorov_best_1999}.
Two years later, this result was extended to the approximation of $L^p$-Sobolev functions
of regularity $r$, with error measured in the $L^q$-norm for arbitrary $p \geq q \geq 2$;
see \cite{gordon2001best}.
Finally, in 2010 the paper \cite{maiorov2010best} showed the same sharp bound of $n^{-r/(d-1)}$
for every $p \geq q \geq 1$.

The techniques used in the present paper rely to a large extent on the techniques used in \cite{maiorov2010best},
combined with a generalization of a result from \cite{maiorov_best_1999}.
In the process of writing the present paper, we discovered that
the proof presented in \cite{maiorov2010best} contains a gap,
which we discuss in detail in \Cref{sec:disc}.
The present work resolves the issue in the proof in \cite{maiorov2010best}
by replacing an orthogonal projection onto a certain space of polynomials by a \emph{quasi-projection} defined in \Cref{sec:ortho}. 

\subsubsection{Shallow neural networks with general activation function}

In \cite{maiorov1999lower}, it is shown that there exists \paul{an analytic}, sigmoidal
(but quite bespoke and somewhat pathological) activation function
$\paul{\tau}: \RR \to \RR$ with the property that every sum of $n$ univariate ridge functions
can be approximated up to arbitrary precision using shallow 
neural networks with activation function $\paul{\tau}$ and $3n$ neurons. 
Combining this with the results presented in \cite{gordon2001best,maiorov2010best} then readily implies
that shallow neural networks with $n$ neurons and the bespoke activation function $\paul{\tau}$ 
can approximate Sobolev functions of regularity $r$ at a rate of $n^{-r/(d-1)}$.
In this work, we extend this statement to the case of generalized translation networks
(as in \Cref{eq:IntroGTN}); see \Cref{thm:main_3}.

\subsubsection{Shallow neural networks and continuous weight selection}

In \cite{mhaskar1996neural}, Mhaskar proves an upper bound of $n^{-r/d}$
for approximating Sobolev functions of regularity $r$ using shallow neural networks
with \emph{any} smooth but non-polynomial activation function. 
In fact, \cite{mhaskar1996neural} does not only consider classical shallow neural networks
but also generalized translation networks (see \Cref{eq:IntroGTN}).
It is then shown that, for a fixed activation dimension $\ell$, the rate of $n^{-r/d}$
is \emph{optimal} under the assumption of \emph{continuous weight selection}, i.e., 
assuming that the map that assigns to a Sobolev function $f$ the coefficients
of the approximating network is continuous. 
Specifically, this shows that under the restriction of continuous weight selection,
there is \emph{no} benefit (in terms of the approximation rate) in choosing
a generalized translation network instead of a classical shallow network.
Our work shows that, if one drops this continuity assumption,
one can indeed \emph{improve} the approximation rate from $n^{-r/(d-1)}$ \paull{(for classical shallow neural networks)}
to $n^{-r/(d-\ell)}$ using a generalized translation network with activation dimension $\ell$, 
at least for a specific choice of the activation function; see \Cref{thm:main_3}.

\subsubsection{Deep neural networks with general activation function}

With respect to neural networks, the results derived in the present paper
can be used to determine the optimal rates of approximating Sobolev functions of regularity $r$,
where optimality is considered with respect to the class of all \emph{shallow} (real- or complex-valued)
neural networks \emph{with general (continuous) activation function}.
This is a novel result for the case of complex-valued neural networks,
and recovers known results for the case of real-valued networks.

In the setting of arbitrary (continuous) activation functions considered here,
the question of optimal approximation rates is only meaningful for \emph{shallow} networks,
since as a consequence of the so-called \emph{Kolmogorov-Arnold representation theorem}
(see \cite[Theorem 5]{maiorov1999lower}) one can show that using the
analytic, sigmoidal (but bespoke and pathological) activation function $\paul{\tau}$
mentioned above, \emph{every} continuous function can be approximated \emph{arbitrarily well}
by neural networks \emph{of constant size} with two hidden layers.
It is easy to see that an analogous result then also holds for networks with more than two hidden layers.
Recently, it was shown that similar results (for networks with $11$ hidden layers)
also hold when using a slightly less "contrived" (but still impractical)
activation function \cite{zhang2022deep}.
These results thus show that no meaningful lower bounds can be derived
in the class of \emph{deep} neural networks with \emph{general} activation function.

\subsubsection{Deep neural networks with ReLU activation function}

In \paull{modern deep learning}, the most commonly used activation function is the ReLU (rectified linear unit),
which has thus been the focus of intensive research.
For this special activation function, one can bound the VC dimension of the class of all ReLU networks
of a given size \cite{bartlett2019nearly}, and this can be used to derive lower bounds for the
approximation of Sobolev functions in terms of the number of weights or neurons,
even for arbitrarily deep networks and without assuming continuous weight selection \cite{yarotsky2020phase}.
Thus, the above phenomenon of "infinite approximation power" of networks with two hidden layers
does \emph{not} occur for the ReLU.
Moreover, for the case where $r / d \gg 1$,
it has been shown in \cite{safran2017depth,petersen2018optimal} that \emph{shallow} ReLU networks
cannot achieve the rate of $n^{-r/d}$ for approximating Sobolev functions of regularity $r$.
This is an instance of the general phenomenon that deeper ReLU networks can achieve
better approximation rates than shallow networks \cite{telgarsky2016benefits},
at least when the functions to be approximated are very smooth \cite{safran2017depth,petersen2018optimal}.

Optimal upper bounds that match the lower bounds (in some cases up to log factors)
have been derived in \cite{yarotsky2020phase,SHEN2022101,CiCP-28-5}.
In particular, it has been shown in \cite{yarotsky2020phase} that for arbitrarily deep ReLU networks
with unconstrained weights one can obtain a strictly better approximation rate
than using networks with constrained weights (i.e., assuming continuous weight selection,
or assuming that the magnitude of the weights is polynomially bounded
in terms of the number of neurons of the network).

Overall, these results show that in order to fully understand the "landscape of approximation results"
for special, elementary activation functions such as the ReLU,
an in-depth analysis is needed for each activation function.
Our bound in the class of shallow networks with \emph{general} activation function
then serves as a benchmark that can be realized using some activation functions
(such as those in \cite{maiorov1999lower,zhang2022deep}),
but not generically with functions such as the ReLU.
Similarly, not every complex activation function can
achieve the corresponding approximation bound of $n^{-r/(2d-2)}$\paul{; see, e.g., \cite[Theorem~4.3]{geuchen2024optimal}}.
This helps explain why the activation function that we construct for attaining the optimal rate
in the case of shallow complex-valued networks (see \Cref{thm:cvnnup}) is relatively impractical.

\subsubsection{Complex-valued neural networks}

While the previously discussed results deal with real-valued neural networks,
in recent years there has been a growing interest in establishing approximation
theoretical results for complex-valued neural networks (CVNNs).
Classical statements about the universality of real-valued neural networks were generalized
to the complex-valued setting in \cite{voigtlaender2023universal,geuchen2023universal}.
Furthermore, the quantitative statement \paull{for real-valued neural networks} from \cite{mhaskar1996neural} was generalized
to the complex-valued setting in \cite{geuchen2024optimal}.
The same paper also proves the existence of a complex-valued activation function $\phi: \CC \to \CC$
with the property that shallow complex-valued neural networks with that activation function
achieve an approximation rate of $n^{-r/(2d-1)}$ for Sobolev functions of smoothness $r$ on the unit ball 
in $\CC^d$.
In the present work, we prove that the same activation function in fact achieves
the optimal approximation rate of $n^{-r/(2d-2)}$, see \Cref{thm:main_4}.
Note that the optimal rate achievable by real-valued networks is $n^{-r / (2d - 1)}$,
as follows by identifying $\CC^d \cong \RR^{2d}$.
Thus, while the quantitative approximation bounds established for CVNNs so far
showed \emph{similar} approximation theoretical properties for complex-valued and real-valued neural networks,
the result derived in this paper indicates a \emph{superior} expressivity of CVNNs.

\subsection{Notation}

By $\ZZ$ we denote the set of integers. 
We use $\NN$ to denote the positive integers excluding $0$ and let $\NN_0 \defeq \NN \cup \{0\}$. 
For $x \in \RR^d$ and $p \in [1,\infty)$, we denote by 
$\pnorm{x} \defeq \left(\sum_{i=1}^d \abs{x_i}^p\right)^{1/p}$ the $\ell^p$-norm of the vector $x$
and by $\infnorm{x} \defeq \underset{j = 1,\dots,d}{\max} \abs{x_j}$ the $\ell^\infty$-norm of $x$.
We write $B^d \defeq \{x \in \RR^d :\ \twonorm{x} \leq 1\}$ for the closed $d$-dimensional Euclidean unit ball. 
Moreover, we let $\SS^{d-1} \defeq \{x \in \RR^d :\ \twonorm{x} = 1\}$ denote the unit sphere in $\RR^d$. 
\paul{We write $\metalambda^d$ for the $d$-dimensional Lebesgue measure.}
\paull{The symbol} $I_{d \times d}$ denotes the $d$-dimensional identity matrix. 
\paull{We denote the set of all (real) orthogonal $d \times d$ matrices by $O(d)$.}
For measurable functions $f,g: \ B^d \to \RR$, we let
\begin{equation} \label{eq:inprod}
\langle f,g \rangle \defeq \int_{B^d} f(x) \cdot g(x) \ \dd x,
\end{equation}
if the expression on the right-hand side is well-defined, i.e., if $f\cdot g \in L^1(B^d)$.
By $\sob$ we denote the set of all \emph{Sobolev functions} $f: \ B^d \to \RR$ for which $\mnorm{f}_{W^{r,p}_d} \leq 1$,
where $\Vert \cdot \Vert_{W_d^{r,p}}$ is defined as in \eqref{eq:sob_def}.
We refer to \Cref{sec:sobolev} for more details on the class of Sobolev functions.  

Let $\KK \in \{\RR, \CC\}$. 
For a measurable subset $K \subseteq \KK^d$, a measurable function $f: K \to \KK$ and a number $p \in [1,\infty)$, we define 
\[
\mnorm{f}_{L^p(K)} \defeq \left( \int_{K} \abs{f(x)}^p \ \dd x \right)^{1/p}
\quad \text{and} \quad \mnorm{f}_{L^\infty(K)} \defeq \underset{x \in K}{\esssup} \abs{f(x)},
\]
where $\esssup$ denotes the \emph{essential supremum}. 
For $p \in [1,\infty]$ and a measurable subset $K \subseteq \KK^d$, we let
\[
L^p(K; \KK) \defeq \left\{ f:\ K \to \KK: \ f \text{ measurable with} \quad \mnorm{f}_{L^p(K)} < \infty\right\}.
\]
We define 
\[
\lploc(\KK^d;\KK) \defeq \left\{ f:\ \KK^d \to \KK: \ f \text{ measurable and } \fres{f}{K} \in L^p(K;\KK) \text{ for every compact } K \subseteq \KK^d\right\}.
\]
Moreover, we set
\[
\mb(\KK^d;\KK) \defeq \left\{ f: \ \KK^d \to \KK: \ f \text{ measurable and bounded on every compact } K \subseteq \KK^d\right\}.
\]
Note here that the definition of $\mb(\KK^d;\KK)$ does not involve \emph{essential} boundedness but genuine boundedness on every compact set. 
In the case $\KK = \RR$, we often omit the second argument in $L^p(K; \KK)$, $\lploc(\KK^d;\KK)$ and $\mb(\KK^d;\KK)$.
For a subset $A \subseteq \KK^d$, we denote by $A^\circ$ its interior with respect to the usual Euclidean topology. 

For a multiindex $\kk \in \NN_0^d$ and a vector $x \in \KK^d$, we define 
\[
\abs{\kk} \defeq \sum_{j=1}^d \kk_j, \quad \kk ! \defeq \prod_{j=1}^d \kk_j ! \quad \text{and} \quad x^{\kk} \defeq \prod_{j=1}^d x_j^{\kk_j}.
\]
For a natural number $s \in \NN_0$ and a set $K \subseteq \RR^d$, we let 
\[
\polyy_s(K) \defeq \left\{ K \ni x \mapsto  \underset{\abs{\kk} \leq s}{\sum_{\kk \in \NN_0^d}} a_{\kk}x^{\kk} :\ a_{\kk} \in \RR\right\}
\]
denote the set of (real) polynomials on $K$ of degree at most $s$ and write $\polyy(K) \defeq \bigcup_{s \in \NN_0} \polyy_s(K)$ for the set of polynomials on $K$. 
\subsection{Main results}
In this section, we present the main results derived in the present paper. 
We start by stating the lower bound for the approximation of Sobolev functions by sums of multivariate ridge functions. 
\begin{theorem}\label{thm:main_1}
Let $d,\ell,r \in \NN$ with $d > \ell$ and $p,q \in [1,\infty]$.
Then there exists a positive constant ${c= c(d,\ell,p,q,r)>0}$ with the following property:
For every $n \in \NN$ there exists a Sobolev function ${f \in \sob}$ such that for every choice of functions $\profilefunction_1,\dots,\profilefunction_n : \RR^\ell \to \RR$
and matrices ${A_1,\dots,A_n \!\in\! \RR^{\ell \times d}}$ with the property that $\profilefunction_k(A_k \cdot \bullet) \in L^{1}(B^d)$ for every $k \in \{1,\dots,n\}$,
we have 
\[
\mnorm{f(x) - \sum_{k=1}^n \profilefunction_k(A_k x)}_{L^q(B^d)} \geq c \cdot n^{-r/(d-\ell)}.
\]
\end{theorem}
We refer to \Cref{sec:low_bound} and specifically to \Cref{corr:low_final} for the proof. 
Remarkably, this lower bound holds under the weak assumption that $\profilefunction_k(A_k \cdot \bullet) \in L^{1}(B^d)$ for every $k \in \{1,\dots,n\}$.
This is, for example, satisfied whenever $A_k$ is full-rank and $\profilefunction_k \in L^{1}_{\mathrm{loc}}(\RR^\ell)$; see \Cref{prop:rank}.

The complementary upper bound can already be achieved by only considering polynomial functions and a \emph{fixed} choice of matrices 
$A_k \in \RR^{\ell \times d}$.

\begin{theorem}\label{thm:main_2}
Let $d, \ell,r \in \NN$ with $d > \ell$ and $1 \leq q \leq p \leq \infty$. 
Then there exists a positive constant $C = C(d,\ell, p,q,r) > 0$ with the following property:
For any $n \in \NN$ there exist matrices $A_1, \dots, A_n \in \RR^{\ell \times d}$ such that for any $f \in \sob$ there exist polynomials $P_1,\dots,P_n \in \polyy(\RR^\ell)$
 with 
\begin{equation*}
\mnorm{ f(x) - \sum_{k=1}^n P_k(A_kx)}_{L^q(B^d)} \leq C \cdot n^{-r/(d-\ell)}.
\end{equation*}
\end{theorem}
The proof of this statement can be found in \Cref{sec:up_bound} (\Cref{thm:pqstatement}).

We use these results to obtain sharp asymptotic bounds for the approximation error using generalized translation networks and complex-valued neural networks. 
We start by stating a lower and upper bound for the approximation error using generalized translation networks.
\begin{theorem}\label{thm:main_3}
Let $\ell \in \NN$. 
Then the following two statements hold:
\begin{enumerate}
\item \label{item:1}For any $d,r \in \NN$ with $d> \ell$ and any $p,q \in [1,\infty]$, there exists a 
constant $c = c(d,\ell,p,q,r)>0$ with the following  
        property: For any $n \in \NN$ there exists a function $f \in \sob$ such that for any choice of $\paul{\tau} \in L^{1}_{\mathrm{loc}}(\RR^\ell)$, 
        matrices $A_1 ,\dots, A_n  \in \RR^{\ell \times d}$ with $\rk(A_k) \in \{0, \ell\}$, 
        biases $b_1 ,\dots, b_n \in \RR^\ell$ and coefficients $c_1 ,\dots, c_n \in \RR$ we have 
        \[
        \mnorm{f(x) - \sum_{k=1}^n c_k\paul{\tau}(A_k x + b_k)}_{L^q(B^d)} \geq c \cdot n^{-r/(d- \ell)}.
        \]
        If $\paul{\tau} \in \mb(\RR^\ell)$, the assumption $\rk(A_k) \in \{0,\ell\}$ is not needed. 

\item\label{item:2} There exists a smooth activation function $\paul{\tau}: \RR^\ell \to \RR$ with the property that for any choice of 
	$d,r \in \NN$ with $d> \ell$ and any $1 \leq q \leq p \leq \infty$ there exists a constant $C = C(d,\ell,p,q,r)>0$ satisfying the following:
	For any $n \in \NN$ there exist
        matrices $A_1 ,\dots, A_n  \in \RR^{\ell \times d}$ such that for any function $f \in \sob$ there exist
        biases $b_1 ,\dots, b_n \in \RR^\ell$ and coefficients $c_1 ,\dots, c_n \in \RR$ with
        \[
        \mnorm{f(x) - \sum_{k=1}^n c_k\paul{\tau}(A_k x + b_k)}_{L^q(B^d)} \leq C \cdot n^{-r/(d- \ell)}.
        \]
 
\end{enumerate}
\end{theorem}

The proofs of these statements are easy consequences of \Cref{thm:main_1,thm:main_2}; see \Cref{subsec:nng}. 

We furthermore establish similar bounds for the approximation error using complex-valued neural networks (CVNNs). 
While the lower bound for CVNNs is an immediate consequence of \Cref{thm:main_1} and might be viewed as a 
special case of \Cref{thm:main_3}\eqref{item:1} (when taking $\ell = 2$),
 the upper bound is \emph{not} immediately obtained as a special case of \Cref{thm:main_3}\eqref{item:2};
see the discussion at the beginning of \Cref{sec:nn}. 
To express the result in a convenient way, we use the identification $\CC^d \cong \RR^{2d}$ \paul{so that we can consider $B^{2d} \subseteq \RR^{2d}$ as a subset of $\CC^d$.}
\begin{theorem}\label{thm:main_4}
The following two statements hold:
\begin{enumerate}
\item \label{item:1_cvnn}For any $d,r \in \NN$ with $d \geq 2$ and any $p,q \in [1,\infty]$, 
	there exists a constant $c = c(d,p,q,r)>0$ with the following  
        property: For any $n \in \NN$ there exists a function $f : \CC^d \to \CC$ that satisfies $\RE(f), \IM(f) \in \mathcal{B}(W_{2d}^{r,p})$
        such that for any choice of a  function $\phi \in L^{1}_{\mathrm{loc}}(\CC;\CC)$, 
        complex vectors $\alpha_1 ,\dots, \alpha_n  \in \CC^d$, biases $\beta_1 ,\dots, \beta_n \in \CC$ and coefficients $\gamma_1 ,\dots, \gamma_n \in \CC$ 
        we have 
        \[
        \mnorm{f(z) - \sum_{k=1}^n \gamma_k\phi(\alpha_k^T z + \beta_k)}_{L^q(B^{2d})} \geq c \cdot n^{-r/(2d- 2)}.
        \]

\item\label{item:2_cvnn} There exists a smooth\footnote{
Smoothness is understood in the sense of real variables here, i.e., identifying $\CC \cong \RR^2$.}
activation function $\phi: \CC \to \CC$ with the property that for any choice of 
	$d,r \in \NN$ with $d \geq 2$ and any $1 \leq q \leq p \leq \infty$ there exists a constant $C = C(d,p,q,r)>0$ satisfying the following:
	For any $n \in \NN$ there exist
        complex vectors $\alpha_1 ,\dots, \alpha_n  \in \CC^d$ such that for any function $f : \CC^d \to \CC$ with $\RE(f), \IM(f) \in \mathcal{B}(W_{2d}^{r,p})$
        there exist biases $\beta_1 ,\dots, \beta_n \in \CC$ and coefficients $\gamma_1 ,\dots, \gamma_n \in \CC$ with
        \[
        \mnorm{f(z) - \sum_{k=1}^n \gamma_k\phi(\alpha_k^T z + \beta_k)}_{L^q(B^{2d})} \leq C \cdot n^{-r/(2d- 2)}.
        \]
 
\end{enumerate}
\end{theorem}
We refer to \Cref{subsec:cvnn} for the proof of this theorem \paull{(see in particular \Cref{corr:lowbound,thm:cvnnup})}.

\subsection{Organization of the paper}
In \Cref{sec:prelim}, we introduce the central objects of the present work and discuss preliminary statements which are important ingredients for our proofs. 
\Cref{sec:low_bound} is devoted to proving the lower bound stated in \Cref{thm:main_1}.
The complementary upper bound from \Cref{thm:main_2} is proven in \Cref{sec:up_bound}.
In \Cref{sec:nn}, these bounds are used to obtain the results for generalized translation networks and complex-valued neural networks stated in \Cref{thm:main_3,thm:main_4}.
Lastly, the appendices contain postponed proofs and a discussion of the gap in the proof of the lower bound in \cite{maiorov2010best}. 

\section{Preliminaries}
\label{sec:prelim}

\subsection{Multivariate ridge functions}
In this work, we study the approximation using sums of $n$ arbitrary multivariate (more precisely, $\ell$-variate) ridge functions. 
Here, for a given $d \in \NN$ and a natural number $\ell \in \{1,\dots,d-1\}$, an $\ell$-variate ridge function is a function $\profilefunction^\ast: \RR^d \to \RR$ of the form $\profilefunction^\ast(x) \defeq \profilefunction(Ax)$,
where $\profilefunction: \RR^\ell \to \RR$ and $A \in \RR^{\ell \times d}$.
Note that we get a classical ridge function in the case $\ell = 1$.
In order to derive approximation bounds, one needs to impose \paul{certain} (mild) assumptions on the functions $\profilefunction$ or $\profilefunction^\ast$ which we discuss here.

Recall that $B^d$ denotes the closed unit ball in $\RR^d$.
We let 
\begin{equation*}
\rdl \defeq \left\{ \profilefunction^\ast: \ B^d \to \RR, \ x \mapsto \profilefunction(Ax): \ A \in \RR^{\ell \times d} ,  \  \profilefunction: \ \RR^\ell \to \RR \text{ measurable with }  \profilefunction^\ast \in L^1(B^d) \right\}
\end{equation*}
denote the set of all $\ell$-variate ridge functions $B^d \to \RR$ that belong to $L^1(B^d)$.
Moreover, we define
\begin{equation*}
\rrdl \defeq \left\{ \profilefunction^\ast: \ B^d \to \RR, \ x \mapsto \profilefunction(Ax): \ A \in \RR^{\ell \times d} ,  \  \profilefunction \in \mb(\RR^\ell) \right\},
\end{equation*}
where we recall that $\mb(\RR^\ell)$ denotes the set of all locally bounded measurable functions. 
Here, it is important to note the subtle difference that in the definition of $\rdl$ we require the composed map $\profilefunction^\ast$ to be in $L^1(B^d)$,
whereas in the definition of $\rrdl$ we require the low-dimensional map $\profilefunction$ to be in $\mb(\RR^\ell)$.

The following proposition provides a sufficient condition for a function to belong to $\rdl$.
\begin{proposition}\label{prop:rank}
Let $d\in \NN$ and $\ell \in \{1,\dots,d-1\}$.
Let $\profilefunction \in L^1_{\mathrm{loc}}(\RR^\ell)$ and $A \in \RR^{\ell \times d}$ with $\rk(A) = \ell$. 
Then we have $\profilefunction^\ast \in \rdl$ where $\profilefunction^\ast(x) \defeq \profilefunction(Ax)$ for $x \in B^d$.
\end{proposition}
\begin{proof}
According to the definition of the set $\rdl$, it suffices to show $\profilefunction^\ast \in L^1(B^d)$.
Let $A = U\Sigma V^T$ be the singular value decomposition of $A$ with orthogonal matrices $U \in \RR^{\ell \times \ell}$, $V \in \RR^{d \times d}$ and a ``diagonal'' matrix
$\Sigma \in \RR^{\ell \times d}$ storing the singular values $\sigma_1>  \dots> \sigma_\ell > 0$ of $A$ (note that $\rk(A) =\ell$). 
Using the rotation invariance of the Lebesgue measure, we obtain 
\[
\int_{B^d} \abs{\profilefunction^\ast(x)}\ \dd x = \int_{B^d} \abs{\profilefunction(U\Sigma V^Tx)} \ \dd x \overset{y = V^Tx}{=} \int_{B^d} \abs{\profilefunction(U\Sigma y)} \ \dd y.
\]
We decompose a vector $y \in \RR^d$ as $y = (y',y'')$ with $y' \in \RR^\ell$ and $y'' \in \RR^{d-\ell}$.
By letting $\tilde{\Sigma} \in \RR^{\ell \times \ell}$ be the matrix that arises from $\Sigma$ by deleting the last $d-\ell$ columns (i.e., $\tilde{\Sigma}$ is the diagonal matrix 
with the singular values $\sigma_1, \dots, \sigma_\ell$ on the diagonal), we obtain 
\[
\profilefunction(U\Sigma y) = \tilde{\profilefunction}(y') \quad \text{for every } \paull{y =(y',y'') \in \RR^\ell \times \RR^{d - \ell} \cong \RR^d},
\]
with $\tilde{\profilefunction}(y') \defeq \profilefunction(U\tilde{\Sigma}y')$.
Thus, we get 
\[
\int_{B^d} \abs{\profilefunction(U\Sigma y)} \ \dd y \leq \int_{[-1,1]^d} \abs{\profilefunction(U\Sigma y)} \ \dd y = \int_{[-1,1]^{d-\ell}}\int_{[-1,1]^\ell} \abs{\tilde{\profilefunction}(y')} \ \dd y' \ \dd y'',
\]
where we used Tonelli's theorem. 
We compute
\begin{align*}
\int_{[-1,1]^\ell} \abs{\tilde{\profilefunction}(y')} \ \dd y' &= \int_{[-1,1]^\ell} \abs{\profilefunction(U\tilde{\Sigma}y')} \ \dd y' 
\overset{z = U\tilde{\Sigma}y'}{=} \frac{1}{\det(\tilde{\Sigma})} \cdot \int_{U\tilde{\Sigma} [-1,1]^\ell} \abs{\profilefunction(z)} \ \dd z =: \alpha < \infty .
\end{align*}
Here, we used the behavior of the Lebesgue measure under a linear change of variables, the fact that $U \tilde{\Sigma}  [-1,1]^\ell \paul{\subseteq \RR^\ell}$ is compact and that $\profilefunction \in L^1_{\mathrm{loc}}(\RR^\ell)$.
Overall, we get 
\[
\int_{B^d} \abs{\profilefunction^\ast(x)} \ \dd x \leq \alpha \cdot \int_{[-1,1]^{d-\ell}}\ \dd y'' = \alpha \cdot 2^{d-\ell} < \infty,
\]
which yields the claim. 
\end{proof}
\begin{remark}
We remark that the assumption that $A$ is full-rank in \Cref{prop:rank} \paull{cannot simply be omitted}. 
To see this, take $d \in \NN$ with $d \geq 3$ and any measurable function $h: \RR \to \RR$ that satisfies 
\[
\int_{\left[-\frac{1}{\sqrt{d}}, \frac{1}{\sqrt{d}}\right]} \abs{h(x)} \ \dd x = \infty.
\]
For arbitrary $\ell \in \NN$ with $d >\ell \geq 2$, we let 
\[
\profilefunction: \quad \RR^\ell \to \RR, \quad \profilefunction(x_1, \dots, x_\ell) = \begin{cases}h(x_1)&\text{if } x_2 = \dots = x_\ell = 0, \\ 0&\text{else.}\end{cases}
\]
Since $\profilefunction \equiv 0$ almost everywhere, we infer $\profilefunction \in L^1_{\mathrm{loc}}(\RR^\ell)$. 
However, we can construct the matrix 
\[
A \defeq \left(\begin{matrix} 1 & 0 & \cdots& 0 \\ 0 & 0 & \cdots & 0 \\ \vdots & \vdots&\ddots&\vdots \\ 0 & 0 & \cdots & 0 \end{matrix}\right)\in \RR^{\ell \times d}.
\]
For $x \in B^d$, we then get $\profilefunction(Ax) = h(x_1)$.
This yields 
\begin{align*}
\int_{B^d} \abs{\profilefunction(Ax)} \ \dd x &\geq \int_{\left[ - \frac{1}{\sqrt{d}}, \frac{1}{\sqrt{d}}\right]^d}  \abs{\profilefunction(Ax)} \ \dd x 
= \int_{\left[ - \frac{1}{\sqrt{d}}, \frac{1}{\sqrt{d}}\right]^d}  \abs{h(x_1)} \ \dd x \\
&=  \left(\frac{2}{\sqrt{d}}\right)^{d-1} \cdot \int_{\left[ - \frac{1}{\sqrt{d}}, \frac{1}{\sqrt{d}}\right]} \abs{h(x_1)} \ \dd x_1 = \infty.
\end{align*}
Thus, $\profilefunction(A \cdot \bullet) \notin L^1(B^d)$.
\end{remark}
As we aim to study the approximation properties of sums of multivariate ridge functions, we define
\begin{equation}\label{eq:rrndl}
\rrndl \defeq \left\{ \sum_{j=1}^n \profilefunction_j : \ \profilefunction_1,\dots, \profilefunction_n \in \rdl\right\}
\end{equation}
and
\begin{equation*}
\rndl \defeq \left\{ \sum_{j=1}^n \profilefunction_j : \ \profilefunction_1,\dots, \profilefunction_n \in \rrdl\right\}.
\end{equation*}

Next, we note that, given a fixed measurable function $f: B^d \to \RR$, the approximation error when approximating $f$
using elements of $\rrndl$ in the $\mnorm{\cdot}_{L^1(B^d)}$-norm remains the same when replacing $\rrndl$ by $\rndl$.
This will allow us to focus on functions in $\rndl$ for deriving our lower bounds. 
\begin{proposition}\label{prop:infequal}
Let $n,d\in \NN$ \paull{and} $\ell \in \{1,\dots,d-1\}$.
Moreover, let $f  : B^d \to \RR$ be a measurable function. 
Then we have 
\[
\underset{R \in \rrndl}{\inf} \mnorm{f - R}_{L^1(B^d)} = \underset{R \in \rndl}{\inf} \mnorm{f - R}_{L^1(B^d)}.
\]
\end{proposition}
\begin{proof}
In order to show 
\[
\underset{R \in \rrndlq}{\inf} \mnorm{f - R}_{L^1(B^d)} \geq \underset{R \in \rndl}{\inf} \mnorm{f - R}_{L^1(B^d)},
\]
we may without loss of generality assume that
\[
\underset{R \in \rrndlq}{\inf} \mnorm{f - R}_{L^1(B^d)} < \infty.
\]
In this case, for a given $\eps > 0$, there exists a function $\overline{R} \in \rrndlq$ satisfying 
\[
\mnorm{f - \overline{R}}_{L^1(B^d)} \leq \underset{R \in \rrndlq}{\inf} \mnorm{f - R}_{L^1(B^d)} + \eps /2.
\]
\paull{By definition of $\rrndl$, we can choose} $\profilefunction_1, \dots, \profilefunction_n : \RR^\ell \to \RR$ and $A_1, \dots, A_n \in \RR^{\ell \times d}$ such that 
\[
\overline{R}(x) = \sum_{j=1}^n \profilefunction_j(A_j x) \quad \text{for } x \in \RR^d,
\]
where $\left(B^d \ni x \mapsto \profilefunction_j(A_j x)\right) \in L^1(B^d)$ for $j \in \{1,\dots,n\}$.
For $N \in \NN$ and $j \in \{1,\dots,n\}$, we define
\[
\profilefunction_j^N(x) \defeq \begin{cases} \profilefunction_j(x),& \text{if } \abs{\profilefunction_j(x)} \leq N, \\
0,& \text{otherwise.}\end{cases}
\]
Then we have $\profilefunction_j^N \in \mb(\RR^\ell)$ (in fact, $\profilefunction_j^N$ is even globally bounded) for every $j \in \{1,\dots,n\}$ and $N \in \NN$.
We define 
\[
R^N: \quad B^d \to \RR, \quad R^N(x) \defeq \sum_{j=1}^n \profilefunction_j^N(A_j x)
\]
and note $R^N \in \rndl$.

Clearly, we have $R^N \to \overline{R}$ pointwise on $B^d$ \paul{as $N \to \infty$}.
Moreover, we observe
\[
\abs{R^N(x)} \leq \sum_{j=1}^n \abs{\profilefunction_j^N(A_jx)} \leq \sum_{j=1}^n \abs{\profilefunction_j(A_j x)}
\]
for every $x \in B^d$, where the right-hand side belongs to $L^1(B^d)$.
Therefore, by dominated convergence we get 
\begin{equation*}
\mnorm{R^N- \overline{R}}_{L^1(B^d)} \to 0 \quad (N \to \infty).
\end{equation*}
Hence, we can pick $N \in \NN$ large enough such that 
\[
\mnorm{R^N - \overline{R}}_{L^1(B^d)} \leq \eps / 2.
\]
We thus get 
\begin{align*}
\mnorm{f - R^N}_{L^1(B^d)} \leq \mnorm{f - \overline{R}}_{L^1(B^d)} + \mnorm{R^N - \overline{R}}_{L^1(B^d)} \leq \underset{R \in \rrndlq}{\inf} \mnorm{f - R}_{L^1(B^d)} + \eps.
\end{align*}
Since $\eps > 0$ was arbitrary, we get 
\begin{equation}\label{eq:second_ineq}
\underset{R \in \rrndlq}{\inf} \mnorm{f - R}_{L^1(B^d)} \geq \underset{R \in \rndl}{\inf} \mnorm{f - R}_{L^1(B^d)}.
\end{equation}

Lastly, note that 
\[
\rndl \subseteq \rrndl,
\]
which implies 
\begin{equation}\label{eq:third_ineq}
\underset{R \in \rndl}{\inf} \mnorm{f - R}_{L^1(B^d)} \geq \underset{R \in \rrndl}{\inf} \mnorm{f - R}_{L^1(B^d)}.
\end{equation}
\Cref{eq:second_ineq,eq:third_ineq} together yield the claim.
\end{proof}

In the following proposition, we note that the matrices in the definition of the set $\rrdl$ (and therefore also in the definition of the set $\rndl$)
can be replaced by matrices with \emph{orthonormal rows}.
Here, we say that a matrix $A \in \RR^{\ell \times d}$ has orthonormal rows if and only if $AA^T = I_{\ell \times \ell}$,
which is equivalent to the rows of $A$ being an orthonormal system in $\RR^d$.
Completing the columns of $A^T$ to an orthonormal basis of $\RR^d$, we obtain in that case
 the existence of an orthogonal matrix $\sigma \in \RR^{d \times d}$
with 
\[
A \sigma = I_{\ell \times d},
\]  
where 
\[
I_{\ell \times d} \defeq \left(\begin{array}{c | c}
I_{\ell \times \ell} & 0
\end{array}\right)
\in \RR^{\ell \times d}.
\]
\begin{proposition} \label{prop:orthrows}
Let $d, \ell \in \NN$ with $\ell < d$. 
Then we have
\[
\rrdl = \left\{ B^d \to \RR, \ x \mapsto \profilefunction(Ax): \ \profilefunction \in \mb(\RR^\ell), \ 
A \in \RR^{\ell \times d} \text{ with } AA^T = I_{\ell \times \ell}\right\}.
\]
\end{proposition}
\begin{proof}
\paul{The inclusion ``$\supseteq$'' is trivial. 
To prove ``$\subseteq$'',} let $\profilefunction \in \mb(\RR^\ell)$ and $A \in \RR^{\ell \times d}$ be arbitrary. 
Let $A = U\Sigma V^T$ be the compact singular value decomposition of $A$ with (semi-)orthogonal matrices 
$U \in \RR^{\ell\times \ell}$, $V \in \RR^{d\times\ell}$ (i.e., $V^T V = UU^T= I_{\ell \times \ell}$) 
and a diagonal matrix $\Sigma=\diag(\sigma_1,\ldots,\sigma_\ell)\in\RR^{\ell \times \ell}$ 
with $\sigma_1 \geq\cdots\geq \sigma_\ell \geq 0$.
We define 
\begin{equation*}
\tilde{\profilefunction} : \quad \RR^\ell \to \RR, \quad \tilde{\profilefunction}(x) = \profilefunction(U\Sigma x).
\end{equation*}
Clearly, it holds that $\tilde{\profilefunction} \in \mb(\RR^\ell)$.
Moreover, for arbitrary $x \in \RR^d$ we get
\begin{equation*}
\profilefunction(Ax) = \profilefunction(U\Sigma V^Tx) = \tilde{\profilefunction}(V^T x)
\end{equation*}
where $V^T \in \RR^{\ell \times d}$ is a matrix with orthonormal rows.  
This proves the claim. 
\end{proof}

\subsection{Sobolev functions on \texorpdfstring{$B^d$}{the unit ball}}\label{sec:sobolev}

In this work, we study the approximation of Sobolev functions by sums of multivariate ridge functions. 
Therefore, we \paul{present our notation regarding Sobolev functions} in this paragraph. 

We identify $W^{r,p}(B^d)$ with $W^{r,p}\big(\left(B^d\right)^\circ\big)$, where $\left(B^d\right)^\circ$ is the open unit ball in $\RR^d$.
Thus, for arbitrary $ p \in [1,\infty]$, we call a function $f \in L^p(B^d)$ an $L^p$\emph{-Sobolev function} of regularity $r \in \NN$ if 
for every multiindex $ \kk \in \NN_0^d$ with $\abs{\kk}\leq r$ the derivative $\partial^{\kk} f$ exists in the weak sense on $\left(B^d\right)^\circ$ and is itself contained in $L^p(B^d)$.
For such a function $f : B^d \to \RR$, we define
\begin{equation}\label{eq:sob_def}
\mnorm{f}_{W^{r,p}_d} \defeq \begin{cases} 
\left(\sum_{\kk \in \NN_0^d, \abs{\kk} \leq r} \mnorm{\partial^{\kk}f}^p_{L^p(B^d)}\right)^{1/p},& p < \infty, \\
\underset{\kk \in \NN_0^d,\abs{\kk} \leq r}{\max}\ \mnorm{\partial^{\kk}f}_{L^\infty(B^d)},& p = \infty.\end{cases}
\end{equation}
We then write $\sob$ for the unit ball in the $L^p$-Sobolev space of regularity $r$, i.e., for the set of functions $f \in W^{r,p}(B^d)$ for which
$\mnorm{f}_{W^{r,p}_d} \leq 1$.

The following result (a Jackson-type bound for Sobolev functions) is folklore and is essential for deriving the approximation bounds in this paper. 
\begin{proposition}[{cf. \cite[Equation~(2.10)]{mhaskar1996neural}}]\label{prop:jack}
Let $d,r \in \NN$ and $1 \leq q \leq p \leq \infty$.
Then there exists a constant $C = C(d,p,q, r)>0$ with the following property:
For any $s \in \NN$ and any $f \in \sob$ there exists a polynomial $P \in \polyy_s(B^d)$ such that
\[
\mnorm{f-P}_{L^q(B^d)} \leq C \cdot s^{-r}.
\]
\end{proposition}

\subsection{Quasi-projection operator onto \texorpdfstring{$\polyy_s(B^d)$}{Pₛ}}
\label{sec:ortho}
\paull{The main goal of this subsection is to set up certain ``quasi-projection'' operators $\prs: L^1(B^d) \to \polyy_{2s-1}(B^d)$ for $s \in \NN$
such that $\prs(P)= P$ for all $P \in \polyy_s(B^d)$ and such that $\sup_{s \in \NN} \mnorm{\prs}_{L^1(B^d) \to L^1(B^d)} < \infty$.}

The space $L^2(B^d)$ together with the inner product defined in \eqref{eq:inprod} forms a Hilbert space. 
Since the set $\polyy(B^d)$ is dense in $L^2(B^d)$ and the set 
\[
\mathcal{B} \defeq \left\{ x^{\kk}: \ \kk \in \NN_0^d\right\}
\]
forms a basis of $\polyy(B^d)$, we conclude that the linear span of $\mathcal{B}$ is dense in $L^2(B^d)$.
Let the map $\varphi: \NN \to \NN_0^d$ be a bijection with the property that $n \mapsto \abs{\varphi(n)}$ is non-decreasing. 
By \paul{applying the Gram-Schmidt algorithm to} $\mathcal{B}$ with respect to the inner product $\langle \cdot, \cdot \rangle$ on $L^2(B^d)$
in the order that is given by the enumeration $\varphi$, we obtain a countable set 
$\Pi^d = \{P_i\}_{i \in \NN}\subseteq \polyy(B^d)$ with the following
properties:
\begin{enumerate}
\item $\Pi^d$ is an orthonormal basis of $L^2(B^d)$.
\item For each $s \in \NN$ the set $\Pi^d_s \defeq \{P_i\}_{i \in I_s}$ forms an orthonormal basis of $\polyy_s(B^d)$, where we define 
$I_s \defeq \{i \in \NN: \ \mathrm{deg}(P_i) \leq s\}$.
\end{enumerate}

For a given $s \in \NN_0$, we let $J_s \defeq \{i \in \NN: \ \mathrm{deg}(P_i) = s\}$. 
For $x,y \in B^d$, we then set 
\[
Q_s(x,y) \defeq \underset{i \in J_s}{\sum} P_i(x)P_i(y), \qquad L_s(x,y) \defeq \sum_{k=0}^\infty \eta \left( \frac{k}{s}\right) Q_k(x,y) \quad \text{for } s  \in \NN,
\]
where $\eta : \RR \to \RR$ is a smooth function with $\eta(x) = 1$ for $x \in [-1,1]$ and $\eta(x) = 0$ for $\abs{x} \geq 2$.
Lastly, following \cite[Def.~11.5.1, Def.~11.1.1, bottom~of~p.~268]{dai2013approximation}, for $ f \in L^1(B^d)$, we define $\prs f : B^d \to \RR$ via 
\[
(\prs f)(x) \defeq \langle f, L_s(x, \bullet)\rangle.
\]
For $f \in L^1(B^d)$, a computation shows 
\begin{align}
(\prs f)(x) &= \langle f, L_s(x, \bullet)\rangle = \sum_{k= 0}^{2s-1} \eta(k/s) \cdot \langle f, Q_k(x, \bullet)\rangle  
 = \sum_{k= 0}^{2s-1} \underset{i \in J_k}{\sum}\eta(k/s) \cdot P_i(x) \cdot \langle f, P_i\rangle   \nonumber\\
 \label{eq:projiden}
 &= \underset{i \in I_{2s-1}}{\sum} \eta(\mathrm{deg}(P_i)/ s) \cdot \langle f , P_i \rangle \cdot P_i(x)
 = \underset{i \in I_{2s-1}}{\sum} a_{i,s}\cdot \langle f , P_i \rangle \cdot P_i(x),
\end{align}
where we denote $a_{i,s} \defeq \eta(\mathrm{deg}(P_i)/ s)$.

Alternatively, motivated by \cite[Equation~(11.1.12)]{dai2013approximation}, if we let $\mathrm{proj}_k (f) \defeq \sum_{j \in J_k} \langle f, P_j \rangle \cdot P_j$ denote the orthogonal projection onto the subspace spanned by $\{P_j\}_{j \in J_k}$,
we may write 
\[
\prs f = \sum_{k= 0}^\infty \eta(k/s) \cdot \mathrm{proj}_k(f).
\]
\paull{In order to show that $\prs$ satisfies the properties stated at the beginning of this subsection it will be helpful to define, for $k,\sigma \in \NN_0$,} 
\[
S_k^\sigma(f) \defeq \frac{1}{\binom{k+\sigma}{\sigma}} \cdot \sum_{j=0}^k \binom{k -j + \sigma}{\sigma} \cdot \mathrm{proj}_j(f),
\]
see \cite[Equation~(11.2.8)~and~(A.4.2)]{dai2013approximation}.
The sequence $(S_k^\sigma(f))_{k \in \NN_0}$ is called the \paul{sequence of }$\sigma$-Cesaro means of the sequence $(\mathrm{proj}_k(f))_{k \in \NN_0}$.

Moreover, for a function $g: \RR \to \RR$, we let $(\Delta g)(x) \defeq g(x) - g(x+1)$ and recursively 
define $\Delta^{\sigma + 1}g \defeq \Delta(\Delta^\sigma(g))$ for any $\sigma \in \NN_0$.
The following technical lemma provides a useful identity.
\begin{lemma}\label{lem:deltarewrite}
For $s \in \NN$ and $\sigma \in \NN_0$, let everything be defined as above. 
We let $\eta^\ast(x) \defeq \eta(x/s)$.
Then  
\begin{equation*}
\prs f = \sum_{k= 0}^\infty \ \left(\Delta^{\sigma + 1}\eta^\ast\right)(k) \cdot \binom{k + \sigma}{ \sigma} \cdot S_k^\sigma(f).
\end{equation*}
\end{lemma}
The proof of \Cref{lem:deltarewrite} can be found in \Cref{app:postponed1}.

We can now deduce three properties of the operator $\prs$ that are central for the present work. 
\begin{proposition}[{cf. \cite[Theorem~11.5.2]{dai2013approximation}}]\label{prop:prs}
\begin{enumerate}
\item $\prs: L^1(B^d) \to \polyy_{2s-1}(B^d)$ is a well-defined linear operator.
\item For every $P \in \polyy_s(B^d)$ we have $\prs P = P$.
\item There exists a constant $C = C(d)>0$ such that $\mnorm{\prs f}_{L^1(B^d)} \leq C \cdot \mnorm{f}_{L^1(B^d)}$ for every $f \in L^1(B^d)$ and $s \in \NN$.
\end{enumerate}
\end{proposition}
\Cref{prop:prs} is stated in \cite{dai2013approximation}, but the proof is omitted, since it is similar to the proof of a different result in \cite{dai2013approximation}.
Since \Cref{prop:prs} is essential for our argument and to make our paper more self-contained, we give a proof.
\begin{proof}[Proof of \Cref{prop:prs}]
\begin{enumerate}
\item Follows from \Cref{eq:projiden}.
\item Let $P \in \polyy_s(B^d)$.
By orthogonality, we have $\langle P, P_i \rangle = 0$ if $\mathrm{deg}(P_i) > s$. 
By \eqref{eq:projiden}, we hence get 
\begin{align*}
\prs P &= \underset{i \in I_{2s-1}}{\sum} \eta(\mathrm{deg}(P_i)/ s) \cdot \langle P , P_i \rangle \cdot P_i \\
 &= \underset{i \in I_{s}}{\sum} \eta(\underbrace{\mathrm{deg}(P_i)/ s}_{\in [0,1]}) \cdot \langle P, P_i \rangle \cdot P_i
 = \underset{i \in I_{s}}{\sum}  \langle P, P_i \rangle \cdot P_i = P,
\end{align*}
where the last step uses that $(P_i)_{i \in I_s}$ is an orthonormal basis of $\polyy_s(B^d)$.

\item 
By \Cref{lem:deltarewrite}, we have
\[
\prs f = \sum_{k= 0}^\infty \ \left(\left(\Delta^{\sigma + 1}\eta^\ast\right)(k) \cdot \binom{k + \sigma}{ \sigma} \cdot S_k^\sigma(f)\right) \quad \text{for every } \sigma \in \NN_0
\]
with $\eta^\ast (x) \defeq \eta(x/s)$.
Note that the choice of  $\eta$ and the definition of $\Delta^{\sigma + 1}$ imply that $\Delta^{\sigma + 1}\eta^\ast(k) = 0$ for $k \geq 2s$. 
We hence get
\[
\mnorm{\prs f}_{L^1(B^d)} \leq \sum_{k= 0}^{2s-1} \abs{\left(\Delta^{\sigma + 1} \eta^\ast\right)(k)} \cdot \binom{k+\sigma}{\sigma} \cdot \mnorm{S_k^\sigma(f)}_{L^1(B^d)}.
\]

We observe, \paull{with $(\eta^\ast)^{(\sigma + 1)}$ denoting the $(\sigma + 1)$-th derivative of $\eta^\ast$},
\begin{align*}
\abs{\left(\Delta^{\sigma + 1} \eta^\ast\right)(k)} &\leq \int_{[0,1]^{\sigma + 1}} \abs{(\eta^\ast)^{(\sigma + 1)}\left(k + u_1 + \dots + u_{\sigma + 1}\right)} \ \dd u_1 \dots \dd u_{\sigma + 1} \\
&\leq \underset{x \in \RR}{\sup} \abs{(\eta^\ast)^{(\sigma + 1)}(x)},
\end{align*}
as follows from \cite[Proposition~A.3.1~(ii)]{dai2013approximation}.
Since $(\eta^\ast)^{(\sigma + 1)}(x) = s^{-\sigma - 1} \cdot \eta^{(\sigma + 1)}(x/s)$ for every $x \in \RR$ by the chain rule, we get 
\[
\abs{\left(\Delta^{\sigma + 1} \eta^\ast\right)(k)} \leq s^{-\sigma - 1} \cdot C_1
\]
with $C_1 = C_1(\sigma) \defeq \underset{x \in \RR}{\sup} \abs{\eta^{(\sigma + 1)}(x)} < \infty$.

Secondly, using \cite[Exercise~0.0.5]{vershynin_high-dimensional_2018}, we have 
\begin{align*}
\binom{k+\sigma}{\sigma} \leq \left(\frac{\ee (k+\sigma)}{\sigma}\right)^{\sigma} 
= \left(\frac{\ee k}{ \sigma} + \ee\right)^\sigma \leq \ee^\sigma \cdot (k + 1)^\sigma = C_2 \cdot (k+1)^\sigma
\end{align*}
for $\sigma \geq 1$, with $C_2 = C_2(\sigma) = \ee^\sigma$.

Lastly, by picking $\sigma > \frac{d}{2}$, we can ensure that 
\[
\mnorm{S_k^\sigma(f)}_{L^1(B^d)} \leq C_3 \cdot \mnorm{f}_{L^1(B^d)} \quad \text{for all } f \in L^1(B^d)
\]
with a constant $C_3 = C_3(d, \sigma)$ which does not depend on $k$.
This follows from \cite[Theorem~11.4.1]{dai2013approximation}, where we choose $\kappa = (0,\dots, 0, 1/2) \in \RR^{d+1}$ in order to obtain the Lebesgue measure. 

Hence, by fixing $\sigma > d/2$, we obtain
\begin{align*}
\mnorm{\prs f}_{L^1(B^d)} &\leq C_1 \cdot C_2 \cdot C_3 \cdot s^{-\sigma - 1} \cdot \sum_{k= 0}^{2s-1} (k+1)^\sigma \cdot \mnorm{f}_{L^1(B^d)} \\
&\leq C_1 \cdot C_2 \cdot C_3 \cdot s^{-\sigma - 1} \cdot \sum_{k= 0}^{2s-1} (2s)^\sigma  \cdot \mnorm{f}_{L^1(B^d)}\\
&\leq C_1 \cdot C_2 \cdot C_3 \cdot s^{-\sigma - 1}\cdot 2^{\sigma + 1} \cdot s^{\sigma + 1} \cdot \mnorm{f}_{L^1(B^d)}= C \cdot \mnorm{f}_{L^1(B^d)}, 
\end{align*}
where $C = C(d) \defeq C_1C_2C_3 \cdot 2^{\sigma + 1}$.
Note that the choice of $\sigma$ only depends on $d$.
This proves the claim. \qedhere
\end{enumerate}
\end{proof}
The properties from \Cref{prop:prs} justify that we call $\prs$ a \emph{quasi-projection} onto $\polyy_s(B^d)$ with range in $\polyy_{2s-1}(B^d)$.
Crucially, this operator is bounded with respect to the $L^1$-norm and the operator norm can be upper bounded independently of $s$, 
whereas the same does not hold for an \emph{orthogonal} projection onto $\polyy_s(B^d)$.

\section{Proof of the lower bound}\label{sec:low_bound}
\paul{
In this section, we prove the \paull{asymptotic} lower bound of $n^{-r/(d-\ell)}$ for the approximation error 
of Sobolev functions using linear combinations of \paull{$n$} multivariate ridge functions, see \Cref{thm:main_1}.
We split this section into three parts:
Firstly, we provide an overview of the proof strategy in \Cref{sec:overview}. 
In \Cref{sec:signset}, we prove that the cardinality of a certain set, defined in \eqref{eq:signset}, can be bounded from above in a suitable way.
This upper bound is then used in \Cref{sec:lower_final} to prove \Cref{thm:main_1}.
\subsection{Proof overview}\label{sec:overview}
Since the proof of the lower bound is in large parts quite technical, we provide a proof overview to explain the underlying idea of the proof and its main steps. 

First, it is not difficult to see that it suffices to prove \Cref{thm:main_1} for the set $\rndl$ and for the case $p=\infty, q= 1$, i.e., it suffices to show
\begin{equation}\label{eq:to_show}
\sup_{f \in \mathcal{B}(W^{r,\infty}_d)} \inf_{R \in \rndl} \mnorm{f - R}_{L^1(B^d)} \gtrsim_{d,\ell,r} n^{-r/(d-\ell)},
\end{equation}
where the notation ``$\gtrsim_{d,\ell,r}$'' indicates an inequality up to a multiplicative factor depending only on $d,\ell$ and $r$.

As the next core idea of the proof, we note that the infimum over $\rndl$ in \eqref{eq:to_show} can effectively be replaced by an infimum over $\prs(\rndl)$, where the degree $s \in \NN$ needs to be carefully balanced with the number of summands $n$.
Here, we recall that $\prs$ denotes the quasi projection onto $\polyy_s(B^d)$ with range in $\polyy_{2s-1}(B^d)$ as discussed in \Cref{sec:ortho}.
To see that it is enough to consider $\prs(\rndl)$, we pick a (large) constant $C = C(d, \ell,r) > 0$ and pick $s \in \NN$ such that $C \cdot n \leq s^{d-\ell} \leq 2C \cdot n$.
Using Jackson's inequality (see \Cref{prop:jack}) and the properties of $\prs$ noted in \Cref{prop:prs} one can deduce the existence of a constant $C_{\mathrm{proj}} = C_{\mathrm{proj}}(d,r)>0$ such that 
\begin{equation}\label{eq:jackproj}
\sup_{f \in \mathcal{B}(W^{r,\infty}_d)} \mnorm{f - \prs(f)}_{L^1(B^d)} \leq C_{\mathrm{proj}} \cdot s^{-r} \leq C_{\mathrm{proj}} \cdot C^{-r/(d-\ell)} \cdot n^{-r/(d-\ell)}.
\end{equation}
The goal is then to show that 
\begin{equation}\label{eq:to_show_proj}
\sup_{f \in \mathcal{B}(W^{r,\infty}_d)} \inf_{P \in \prs(\rndl)} \mnorm{f - P}_{L^1(B^d)} \gtrsim_{d,\ell,r} C^{(r/d) - (r/(d-\ell))} \cdot n^{-r/(d-\ell)}.
\end{equation}
Indeed, since 
\[
\mnorm{f-R}_{L^1(B^d)} \gtrsim_{d} \mnorm{\prs(f)-\prs(R)}_{L^1(B^d)} \geq  \mnorm{f-\prs(R)}_{L^1(B^d)} - \mnorm{f-\prs(f)}_{L^1(B^d)}
\]
for every $R \in \rndl$, we see that \eqref{eq:jackproj} and \eqref{eq:to_show_proj} together imply \eqref{eq:to_show} if $C$ is chosen sufficiently large. 

Next, we construct a function $f \in \mathcal{B}(W^{r,\infty}_d)$ realizing the lower bound in \eqref{eq:to_show_proj} by considering a set of \paull{certain sums of} smooth bump functions. 
More precisely, we pick $m \asymp_{d,\ell,r} C^{-1} \cdot s^d$ (where $C = C(d, \ell,r)>0$ is the constant from above which balances the degree $s$ with $n$).
We then choose $\xi_1, \dots, \xi_m \in B^d$ such that the cubes $Q_i \defeq \xi_i + Q$ are pairwise disjoint subsets of $B^d$ (in fact, the $Q_i$ have distance of order $m^{-1/d}$ from each other), where $Q \subseteq \RR^d$ is a cube with $\metalambda^d(Q) \asymp_{d} m^{-1}$, meaning that the side lengths are of order $m^{-1/d}$.
For given $\eps \in \{ \pm 1\}^m$, we then let $f_\eps  \in \mathcal{B}(W^{r,\infty}_d)$ be a smooth function satisfying $f_\eps \equiv \eps_i \cdot \frac{\kappa}{m^{r/d}}$ on $Q_i$ for every $i \in \{1, \dots, m\}$.  
The constant $\kappa = \kappa(d,r)>0$ is needed to ensure $\mnorm{f_\eps}_{W^{r,\infty}_d} \leq 1$.
For arbitrary $P \in \prs(\rndl)$ we then compute
\begin{align}
\mnorm{f_\eps - P}_{L^1(B^d)} &\geq \sum_{i=1}^m \int_{Q_i} \abs{f_\eps(x) - P(x)} \ \dd x = \sum_{i=1}^m \int_{Q} \abs{f_\eps(\xi_i + t) - P(\xi_i + t)} \ \dd t \nonumber\\
&\geq \int_Q \ \inf_{\tilde{P} \in \prs(\rndl), \eta \in \RR^d} \sum_{i=1}^m \abs{\frac{\kappa}{m^{r/d}}\eps_i - \tilde{P}(\xi_i + \eta)} \ \dd t \nonumber\\
\label{eq:l1bound}
&\asymp_{d,r} \frac{1}{m^{1 + r/d}} \cdot \inf_{\tilde{P} \in \prs(\rndl), \eta \in \RR^d} \mnorm{\eps - \left(\tilde{P}(\xi_i + \eta)\right)_{i=1}^m}_{\ell^1};
\end{align}
see \Cref{prop:low_bound_eps} for the details.
By \cite[Lemma~6]{maiorov2010best} it is known that if 
\begin{equation}\label{eq:card_bound}
\abs{\left\{ \left(\sgn(\tilde{P}(\xi_1 + \eta)), \dots, \sgn(\tilde{P}(\xi_m + \eta))\right): \ \tilde{P} \in \prs(\rndl), \eta \in \RR^d\right\}} \leq 2^{m/4}
\end{equation}
then there exists $\eps^\ast \in \{\pm 1\}^m$ with 
\[
\inf_{\tilde{P} \in \prs(\rndl), \eta \in \RR^d} \mnorm{\eps^\ast - \left(\tilde{P}(\xi_i + \eta)\right)_{i=1}^m}_{\ell^1} \geq am,
\]
where $a > 0$ is an absolute constant. 
Plugging this into \eqref{eq:l1bound} yields 
\[
\mnorm{f_{\eps^\ast} - P}_{L^1(B^d)} \gtrsim_{d,r} m^{-r/d} \asymp_{d,\ell,r} C^{r/d} \cdot s^{-r} 
\gtrsim_{d,\ell,r} C^{(r/d)- (r/(d-\ell))} \cdot n^{-r/(d-\ell)}, 
\]
as desired. 

It remains to show that \eqref{eq:card_bound} is indeed true. 
The proof of this fact is contained in \Cref{sec:signset} and is based on an application of \cite[Lemma 3]{maiorov_best_1999} together with a careful evaluation of the inner product
of a \paull{(multivariate)} ridge function and a given polynomial using multivariate polar coordinates; see \Cref{lem:inner_product_expansion}.

On a high level, the proof structure follows that of the result for univariate ridge functions in \cite{maiorov2010best}, 
properly modified and adapted to the multivariate setting. 
In particular, we highlight the following two key differences: 
\begin{itemize}
\item Central to the proof in \cite{maiorov2010best} is Lemma 2, where it is shown that the inner product (as defined in \eqref{eq:inprod}) between a (univariate) 
	ridge function and a polynomial
        from a special system of orthogonal polynomials admits a certain separation of variables. 
        We, in contrast, express the inner product between a (multivariate) ridge function and a fixed polynomial in a different way, proving a generalization of 
        \cite[Theorem~3]{maiorov_best_1999}; see \Cref{lem:inner_product_expansion}. 
        Moreover, while the specific choice of the system of orthogonal polynomials is of central importance in \cite{maiorov2010best}, we impose (almost) no restrictions
        on the set of orthogonal polynomials that we consider but can simply use the ``naive'' system defined in \Cref{sec:ortho}.
\item In contrast to the proof in \cite{maiorov2010best}, 
we do not use the \emph{orthogonal} projections onto the space $\polyy_s(B^d)$ but the quasi-projection operators onto $\polyy_s(B^d)$ with range in 
$\polyy_{2s-1}(B^d)$ defined in \Cref{sec:ortho}. 
The fact that these operators are uniformly bounded with respect to the $L^1$-norm enables us to show the lower bound with respect to the 
$L^1$-norm and therefore close the gap in the proof in \cite{maiorov2010best} which we describe in more detail in \Cref{sec:disc}.
\end{itemize}
}
\subsection{Sign set cardinality}\label{sec:signset}
\paul{In this section, we show that \eqref{eq:card_bound} is satisfied under suitable assumptions.
To this end,} let $d,\ell,m,n,s \in \NN$ with $\ell < d$ be arbitrary and fix\footnote{
Such a $\vartheta$ always exists: If $m^{1/d} \geq 2$ this follows from $m^{1/d} - \frac{m^{1/d}}{2} \geq 1$ and if $m^{1/d} \in [1,2]$ we can simply pick $\vartheta = 1$.
}
$\vartheta = \vartheta(d,m) \in \NN$ satisfying 
\[
\frac{m^{1/d}}{2} \leq \vartheta \leq m^{1/d}.
\]
Consider the lattice
\[
\Xi \defeq \left\{ \left(\frac{i_1 + 1/2}{\sqrt{d}\vartheta},\dots,\frac{i_d + 1/2}{\sqrt{d}\vartheta} \right): \ i_1,\dots, i_d \in \ZZ \cap [-\vartheta, \vartheta - 1]\right\} \subseteq B^d.
\]
Then it holds that $\abs{\Xi} = (2\vartheta)^d \geq m$.
Let $\{\xi_1, \dots, \xi_m\} \subseteq \Xi$ with $\xi_i \neq \xi_j$ for $i \neq j$.
We then set
\begin{align}\label{eq:def_pi}
\Pi_{m,s,n,\ell,d}  \defeq \Bigl\{(P(\xi_1+t),\ldots,P(\xi_m+t)): \ P\in\Pr\nolimits_s (\rndl),\ t\in\RR^d\Bigr\}\subseteq\RR^m.
\end{align}
Here, $\prs$ is as introduced in \Cref{sec:ortho}.
Moreover, we define 
\[
E^m \defeq \{\pm 1\}^m = \{ \eps = (\eps_1 ,\dots, \eps_m): \ \eps_i = \pm 1, \ i = 1,\dots,m\}.
\]
\paull{We want to show that the set defined in \eqref{eq:card_bound}, i.e.,}
\begin{equation}\label{eq:signset}
\sgn(\Pi_{m,s,n,\ell,d})\defeq\{(\sgn (h_1),\ldots,\sgn (h_m)):h\in\Pi_{m,s,n,\ell,d}\} \subseteq E^m
\end{equation}
 (where $\sgn(a)=1$ if $a\geq0$ and $\sgn(a) = -1$ if $a<0)$ is ``small'' relative to $E^m$.
 Precisely, in this section, we will show the following statement, which may be seen as a generalization of \cite[Lemma~4]{maiorov2010best} to the case of multivariate ridge functions. 
\begin{lemma}\label{lem:pi_small}
Let $d, \ell \in \NN$ with $\ell < d$ and $c_1> 0$ be arbitrary and let $c_0 \defeq 4^{\ell + 3} /c_1$.
Then there exists a constant $C = C(d, \ell, c_1)$ such that for any choice of 
$m,s,n\in \NN$ with  
\begin{equation}\label{eq:cond}
c_1 s^d \leq m \leq 2c_1 s^d, \quad c_0 n \leq s^{d-\ell} \leq 2c_0 n \quad \mathrm{and} \quad s \geq C 
\end{equation}
we have
\begin{align*}
\lvert\sgn(\Pi_{m,s,n,\ell,d})\rvert\leq 2^{m/4}.
\end{align*}
Here, $\Pi_{m,s,n,\ell,d}$ and $\sgn(\Pi_{m,s,n,\ell,d})$ are as defined in \eqref{eq:def_pi} and \eqref{eq:signset}.
\end{lemma}
 In the case that $\lvert\sgn(\Pi_{m,s,n,\ell,d})\rvert\leq 2^{m/4}$, there exists a vector $\varepsilon^\ast \in E^m$ that is ``far away'' from $\Pi_{m,s,n,\ell,d}$ and we will use
 this vector in \Cref{sec:lower_final} to construct a function that realizes the lower bound for ridge function approximation. 
 This intuition is backed up by the following lemma. 
 \begin{lemma}[{cf. \cite[Lemma~6]{maiorov2010best}}] \label{lem:epsbound}
Let $m \in \NN$ and $\Gamma \subseteq \RR^m$ with 
\[
\abs{\sgn(\Gamma)} \leq 2^{m/4}.
\]
Then there exists a vector $\eps^\ast \in E^m$ such that
\[
\underset{x \in \Gamma}{\inf} \onenorm{\eps^\ast - x} \geq am
\]
for an absolute constant $a>0$.
\end{lemma}
The proof of the lemma is given in \cite{maiorov2010best}. 
However, in order to clarify the proof in \cite{maiorov2010best} and to keep the paper more self-contained, we include the proof in \Cref{sec:signproof}.

The remaining part of this subsection is dedicated to proving \Cref{lem:pi_small}. 
 The following lemma provides the central tool in order to obtain a bound on the cardinality of $\sgn(\Pi_{m,s,n,\ell,d})$.
 
\begin{lemma}[{cf. \cite[Lemma 3]{maiorov_best_1999}}]\label{lem:polynomial_manifold}
Let $m,s,N,K \in \NN$ be natural numbers such that $N+K\leq m/2.$ Moreover, let 
$\pi_{\alpha,\beta}(z)$ for $\alpha\in\{1,\ldots,m\}$ and $\beta\in\{1,\ldots,K\}$ be polynomials of degree at most
$s$ in $z\in\RR^N$, and set
\[
\pi_\alpha: \quad \RR^{K} \times \RR^N \to \RR, \quad \pi_\alpha(b,z)\defeq\sum_{\beta=1}^K b_\beta\pi_{\alpha,\beta}(z)
\qquad 
\text{for } \alpha \in \{1,\dots,m\}.
\]
Then, for 
\begin{align*}
\Pi^*_{m,s,N,K}\defeq\Bigl\{(\pi_1(b,z),\ldots,\pi_m(b,z)):(b,z)\in\RR^K\times\RR^N\Bigr\}\subseteq\RR^m,
\end{align*}
we have
\begin{align*}
\lvert\sgn(\Pi^*_{m,s,N,K})\rvert\leq(4s)^N(N+K+1)^{N+2}\biggl(\frac{2\ee m}{N+K}\biggr)^{N+K}.
\end{align*}
\end{lemma}

In order to show that \paul{the set }$\sgn(\Pi_{m,s,n,\ell,d})$ is ``small'' \paull{compared to $E^m$} we need to express the set $\Pi_{m,s,n,\ell,d}$ in a way that fits the setting of \Cref{lem:polynomial_manifold}. 
We start with a technical auxiliary statement.
It is well-known, but for the sake of completeness we include a short proof in \Cref{sec:proof_trig_power_reduction}.
\begin{lemma}\label{lem:trig_power_reduction}
Let $a,b \in \NN_0$. 
Then there exist coefficients $\alpha_0 = \alpha_0(a,b), \dots, \alpha_{a+b}=\alpha_{a+b}(a,b) \in \RR$ as well as $\beta_0 = \beta_0(a,b), \dots, \beta_{a+b}= \beta_{a+b}(a,b) \in \RR$
with
\begin{equation*}
\cos(\varphi)^a \sin(\varphi)^b = \sum_{h= 0}^{a+b} \left(\alpha_h \cos(h\varphi) + \beta_h \sin(h\varphi)\right), \quad \varphi \in \RR.
\end{equation*}
\end{lemma}
In the following, we show that the inner product of a fixed polynomial of degree at most $s$ and an arbitrary \paull{multivariate} ridge function can be expressed in a convenient way.
This is in fact a generalization of \cite[Theorem~3]{maiorov_best_1999} to the case of multivariate ridge functions. 
\begin{lemma}\label{lem:inner_product_expansion}
Let $s,d,\ell \in \NN$ with $\ell < d$ and set $\mu \defeq 2^\ell (s+d+1)^\ell$.
Given $P \in \polyy_s(\RR^d)$, there exist polynomials 
$Q_1(\bullet; P),\dots, Q_\mu(\bullet; P) \in \polyy_s(\RR^{d^2})$ 
(where we identify $\RR^{d \times d}$ with $\RR^{d^2}$) such that, 
given $A \in \RR^{\ell \times d}$ and $\sigma \in O(d)$ with $A\sigma = I_{\ell \times d}$ and $\profilefunction \in \mb(\RR^\ell)$, 
writing $\profilefunction_A(x) = \profilefunction(Ax)$, we have 
\begin{align*}
\langle \profilefunction_A, P \rangle = \sum_{h=1}^{\mu} b_h(\profilefunction) Q_h(\sigma; P),
\end{align*}
with coefficients $b_h(\profilefunction) \in \RR$ that only depend on $h$ and $\profilefunction$.
\end{lemma}
\begin{proof}
Let $P \in \polyy_s(\RR^d)$ be an arbitrary polynomial of degree at most $s$ and let $\sigma \in \RR^{d \times d}$.
Reordering in terms of $y$, we can write
\begin{align}\label{eq:pk}
P(\sigma y) &= \underset{\abs{\kk} \leq s}{\sum}P_{\kk}(\sigma; P) y^{\kk} \quad \text{for every } y \in \RR^d,
\end{align}
where each function $\sigma \mapsto P_{\kk}(\sigma; P)$ is a polynomial of degree at most $s$ in the $d^2$ variables
$(\sigma_{i,j})_{i,j=1,\dots,d}$.
For a multiindex $\kk \in \NN_0^d$ with $\abs{\kk} \leq s$, we set
\begin{equation}\label{eq:qkdl}
q_{\kk,d,\ell} \defeq \frac{1}{\kk_{\ell+1} + \dots+ \kk_d + d - \ell} \cdot \int_{\SS^{d-\ell-1}} \xi_{\ell+1}^{\kk_{\ell+1}} \dots \xi_d^{\kk_d}\ \dd \mathcal{H}^{d-\ell-1}(\xi),
\end{equation}
where $\mathcal{H}^{d-\ell-1}$ denotes the $(d-\ell-1)$-dimensional Hausdorff measure. 
Let further $k < \ell$.
Using \Cref{lem:trig_power_reduction}, we can pick coefficients $\alpha_{h,\kk}^k, \alpha_{h,\kk}^\ell, \beta_{h,\kk}^k, \beta_{h,\kk}^\ell \in \RR$ that only depend on $h,d,\ell,k$ and $\kk$ such that 
\begin{align*}
\cos(\varphi)^{\kk_{k+1}}\sin(\varphi)^{k-1+\kk_1+\dots+\kk_k}&=\sum_{h=0}^{k-1+\kk_1+\dots+\kk_{k+1}}(\alpha_{h,\kk}^k\cos(h\varphi)+\beta_{h,\kk}^k\sin(h\varphi)) \quad \text{and}\\
\cos(\varphi)^{\kk_{\ell+1}+\dots+\kk_d+d-\ell+1}\sin(\varphi)^{\ell-1+\kk_1+\dots+\kk_\ell}
&=\sum_{h=0}^{d+\kk_1+\dots+\kk_d}(\alpha_{h,\kk}^\ell\cos(h\varphi)+\beta_{h,\kk}^\ell\sin(h\varphi))
\end{align*}
hold for every $\varphi \in \RR$.
Since $k<\ell < d$ and $\abs{\kk}\leq s$, we can add zeros to obtain 
\begin{align}
\cos(\varphi)^{\kk_{k+1}}\sin(\varphi)^{k-1+\kk_1+\dots+\kk_k}&=\sum_{h=0}^{d+s}(\alpha_{h,\kk}^k\cos(h\varphi)+\beta_{h,\kk}^k\sin(h\varphi)) \quad \text{and} \nonumber\\
\label{eq:addzeros}
\cos(\varphi)^{\kk_{\ell+1}+\dots+\kk_d+d-\ell+1}\sin(\varphi)^{\ell-1+\kk_1+\dots+\kk_\ell}
&=\sum_{h=0}^{d+s}(\alpha_{h,\kk}^\ell\cos(h\varphi)+\beta_{h,\kk}^\ell\sin(h\varphi))
\end{align}
for every $\varphi \in \RR$. 
Note that each sum $\sum_{h=0}^{d+s}(\alpha_{h,\kk}^k\cos(h\varphi)+\beta_{h,\kk}^k\sin(h\varphi))$ for $k \in \{1,\dots,\ell\}$ consists in total of $2(d+s+1)$ summands. 
Therefore, after rearranging, we can write 
\begin{equation}\label{eq:zetatilde}
\prod_{k=1}^\ell\sum_{h=0}^{d+s}(\alpha_{h,\kk}^k\cos(h\varphi_k)+\beta_{h,\kk}^k\sin(h\varphi_k)) 
= \sum_{h=1}^{2^\ell(d+s+1)^\ell}\tilde{\zeta}_{h,\kk}f_h(\varphi_1,\ldots,\varphi_\ell) \quad \text{for all } (\varphi_1, \dots, \varphi_\ell) \in \RR^\ell,
\end{equation}
where the $\tilde{\zeta}_{h, \kk} \in \RR$ are coefficients that depend only on $h,\kk,d,\ell$ and
each function $f_h$ is of the form 
\[
f_h(\varphi_1,\dots,\varphi_\ell) = \prod_{k=1}^\ell \left\{ \begin{matrix} \cos \\ \sin\end{matrix}\right\}_{h,k}(\tau_{h,k}\varphi_k)
\]
with $\tau_{h,k} \in \{0,1,\dots,d+s\}$ and $\left\{\begin{matrix} \cos \\ \sin\end{matrix}\right\}_{h,k} \in \{\cos, \sin\}$ for $h \in \{1,\dots, 2^\ell(d+s+1)^\ell\}$ and $k \in \{1,\dots,\ell\}$.
Note that $f_h$ only depends on $h$ (and $d,\ell$)
but \emph{not} on $\kk$.
We then set 
\begin{equation}\label{eq:zeta}
\zeta_{h,\kk} \defeq q_{\kk,d,\ell}\tilde{\zeta}_{h,\kk},
\end{equation}
 where $q_{\kk,d,\ell}$ is as defined in \eqref{eq:qkdl}.
Lastly, for $h \in \{1,\dots, 2^\ell(s+d+1)^\ell\}$ we define 
\begin{equation}\label{eq:qh}
Q_h(\sigma;P) \defeq \underset{\abs{\kk} \leq s}{\sum} \zeta_{h,\kk} \cdot P_{\kk}(\sigma;P),
\end{equation}
where the $P_\kk$ are as defined in \eqref{eq:pk}.
Since $P_\kk(\bullet; P) \in \polyy_s(\RR^{d^2})$ for every $\kk \in \NN_0^d$ with $\abs{\kk} \leq s$, we infer that the same holds for the functions $Q_h(\bullet;P)$.

We claim that the polynomials $Q_h(\bullet;P)$ have the property stated in the formulation of the \paull{lemma}.
To see this, let $\profilefunction \in \mb(\RR^\ell)$.
Fix a multiindex $\kk \in \NN_0^{d}$. 
By Fubini's theorem,
\begin{align*}
&\norel \int_{B^d} \profilefunction(y_1,\ldots, y_\ell) y^{\kk} \ \dd y \\
&= \int_{B^\ell} \profilefunction(y_1,\ldots, y_\ell) y_1^{\kk_1}\dots y_\ell^{\kk_\ell} 
\left(\int_{B_{y_1,\ldots,y_\ell}} y_{\ell+1}^{\kk_{\ell+1}} \dots y_d^{\kk_d} \ \dd(y_{\ell+1},\dots,y_d)\right) \ \dd(y_1,\ldots,y_\ell)
\end{align*}
with 
\begin{equation*}
B_{y_1,\ldots,y_\ell} \defeq \left\{ (y_{\ell+1},\dots,y_d) \in \RR^{d-\ell}: \ y_{\ell+1}^2 + \dots + y_d^{2} \leq 1 - y_1^2 -\dots- y_\ell^2\right\}.
\end{equation*}
We first study the inner integral: Transforming to polar coordinates (see, e.g., \cite[p.~118]{evans_measure_1992}), we get 
\begin{align*}
&\norel \int_{B_{y_1,\ldots,y_\ell}} y_{\ell+1}^{\kk_{\ell+1}} \dots y_d^{\kk_d} \ \dd(y_{\ell+1},\dots,y_d) \\
&= \int_0^{\sqrt{1-y_1^2 -\dots- y_\ell^2}} r^{d-\ell-1}\cdot \left(\int_{\SS^{d-\ell-1}} (r\xi_{\ell+1})^{\kk_{\ell+1}} \dots (r\xi_d)^{\kk_d}\ \dd\mathcal{H}^{d-\ell-1}( \xi)\right) \ \dd r  \\
&= \int_0^{\sqrt{1-y_1^2 -\dots- y_\ell^2}} r^{\kk_{\ell+1} + \dots + \kk_d + d-\ell-1}\cdot \left(\int_{\SS^{d-\ell-1}} \xi_{\ell+1}^{\kk_{\ell+1}} \cdots \xi_d^{\kk_d}\ \dd\mathcal{H}^{d-\ell-1}( \xi)\right) \ \dd r  \\ 
&= q_{\kk, d,\ell} \cdot (1- y_1^2 -\dots- y_\ell^2)^{(\kk_{\ell+1} + \dots + \kk_d + d - \ell)/2},
\end{align*}
where $q_{\kk,d,\ell}$ is as defined in \eqref{eq:qkdl}.
This shows
\begin{align}
&\norel \int_{B^d} \profilefunction(y_1,\ldots, y_\ell) y^{\kk} \ \dd y \nonumber\\
\label{eq:pal_1}
&= q_{\kk, d,\ell} \cdot \int_{B^\ell}\profilefunction(y_1,\ldots, y_\ell)y_1^{\kk_1}\dots y_\ell^{\kk_\ell} (1- y_1^2 -\dots- y_\ell^2)^{(\kk_{\ell+1} + \dots + \kk_d + d - \ell)/2} \ \dd(y_1,\ldots, y_\ell).
\end{align}
Let us first assume $\ell \geq 2$.
We transform to hyperspherical coordinates (see for instance \cite{hyperspherical_coordinates}) in the last integrand. 
This transformation is given by (see also \cite[Definition~2]{hyperspherical_coordinates})
\begin{align*}
y_1 &= r \prod_{k=1}^{\ell - 1} \sin(\varphi_k), \\
y_2 &= r \cos(\varphi_1) \prod_{k=2}^{\ell - 1} \sin(\varphi_k), \\
y_3 &= r \cos(\varphi_2) \prod_{k=3}^{\ell - 1} \sin(\varphi_k), \\
&\vdots \\
y_{\ell - 1}&= r \cos(\varphi_{\ell - 2})\sin(\varphi_{\ell - 1})=r \cos(\varphi_{\ell - 2})\prod_{k = \ell - 1}^{\ell - 1} \sin(\varphi_k) \\
y_\ell &= r \cos(\varphi_{\ell - 1}) = r \cos(\varphi_{\ell - 1})\prod_{k = \ell }^{\ell - 1} \sin(\varphi_k),
\end{align*}
where $r \in [0,1], \varphi_1 \in [-\pi, \pi]$ and $\varphi_k \in [0,\pi]$ for $k \in \{2,\dots, \ell-1\}$.
\paull{One can show that $y_1^2 + \dots + y_\ell^2 = r^2$.}
The absolute value of the determinant of the Jacobian of that transformation is given by 
\[
r^{\ell - 1} \prod_{k=2}^{\ell - 1} \sin^{k-1}(\varphi_k); \quad \text{see \cite{hyperspherical_coordinates} for a proof.}
\]

Performing this transformation in the integral in \eqref{eq:pal_1} gives
\begin{align*}
&\norel 
\int_{B^\ell}\profilefunction(y_1,\ldots, y_\ell)y_1^{\kk_1}\dots y_\ell^{\kk_\ell} (1- y_1^2 -\dots- y_\ell^2)^{(\kk_{\ell+1} + \dots + \kk_d + d - \ell)/2} \ \dd(y_1,\ldots, y_\ell) \\
&=  \int_0^1 \!\!\! \int_{-\pi}^\pi \! \int_{[0,\pi]^{\ell-2}}\! \! \!  r^{\ell-1}\prod_{k=2}^{\ell-1}\sin^{k-1}(\varphi_k)\cdot 
\profilefunction\!\left(\!r\prod_{k=1}^{\ell-1}\!\sin(\varphi_k), r\cos(\varphi_1)\prod_{k=2}^{\ell-1}\!\sin(\varphi_k),\ldots,r\cos(\varphi_{\ell-1})\prod_{k=\ell}^{\ell-1}\!\sin(\varphi_k)\!\!\right) \\
&\norel \left(r\prod_{k=1}^{\ell-1}\sin(\varphi_k)\right)^{\kk_1}\left(r\cos(\varphi_1)\prod_{k=2}^{\ell-1}\sin(\varphi_k)\right)^{\kk_2}\dots\left(r\cos(\varphi_{\ell-1})\prod_{k=\ell}^{\ell-1}\sin(\varphi_k)\right)^{\kk_\ell}\\
&\norel(1-r^2)^{(\kk_{\ell+1} + \dots + \kk_d + d - \ell)/2} 
 \dd(\varphi_2,\ldots,\varphi_{\ell-1}) \dd \varphi_1  \dd r.
\end{align*}
Note that in the case $\ell =2$, the innermost integral (over $[0,\pi]^{\ell-2}$) has to be omitted. 
We now additionally apply the substitution $r = \sin(\varphi_\ell)$ \paull{(with $\varphi_\ell \in \left[0, \frac{\pi}{2}\right]$)} and get 
\begin{align*}
&\norel \int_{B^\ell}\profilefunction(y_1,\ldots, y_\ell)y_1^{\kk_1}\dots y_\ell^{\kk_\ell} (1- y_1^2 -\dots- y_\ell^2)^{(\kk_{\ell+1} + \dots + \kk_d + d - \ell)/2} \ \dd(y_1,\ldots, y_\ell) \\
&= \int_0^{\pi/2} \! \int_{-\pi}^\pi \int_{[0,\pi]^{\ell-2}} \prod_{k=1}^\ell\sin^{k-1}(\varphi_k)\cdot 
\profilefunction\left(\prod_{k=1}^\ell\sin(\varphi_k), \cos(\varphi_1)\prod_{k=2}^\ell\sin(\varphi_k),\ldots,\cos(\varphi_{\ell-1})\prod_{k=\ell}^\ell\sin(\varphi_k)\right) \\
&\norel \left(\prod_{k=1}^\ell\sin(\varphi_k)\right)^{\kk_1}\left(\cos(\varphi_1)\prod_{k=2}^\ell\sin(\varphi_k)\right)^{\kk_2}\dots\left(\cos(\varphi_{\ell-1})\prod_{k=\ell}^\ell\sin(\varphi_k)\right)^{\kk_\ell}\cos(\varphi_\ell)^{\kk_{\ell+1} + \dots + \kk_d + d - \ell+1} \\
&\norel \dd(\varphi_2,\ldots,\varphi_{\ell-1}) \dd \varphi_1  \dd \varphi_\ell \\
&=  \int_0^{\pi/2} \! \int_{-\pi}^\pi \int_{[0,\pi]^{\ell-2}} \left( \prod_{k = 1}^{\ell - 1} \cos(\varphi_k)^{\kk_{k+1}}\sin(\varphi_k)^{k- 1 + \kk_1 + \dots + \kk_k}\right) 
\cdot \cos(\varphi_\ell)^{\kk_{\ell + 1} + \dots +\kk_d + d - \ell + 1} \\
&\norel \sin(\varphi_\ell)^{\ell - 1 + \kk_1 + \dots + \kk_\ell} \cdot 
\profilefunction\left(\prod_{k=1}^\ell\sin(\varphi_k), \cos(\varphi_1)\prod_{k=2}^\ell\sin(\varphi_k),\ldots,\cos(\varphi_{\ell-1})\sin(\varphi_\ell)\right)  \\
&\norel\dd(\varphi_2,\ldots,\varphi_{\ell-1})\dd \varphi_1  \dd \varphi_\ell.
\end{align*}
Using \eqref{eq:pal_1} and recalling our choice of $\alpha_{h,\kk}^k, \alpha_{h,\kk}^\ell, \beta_{h,\kk}^k$ and $\beta_{h,\kk}^\ell$ in \eqref{eq:addzeros} as 
well as our choice of $\tilde{\zeta}_{h,\kk}$ and $\zeta_{h,\kk}$ in \eqref{eq:zetatilde} and \eqref{eq:zeta}, we obtain 
\begin{align}
&\norel \int_{B^d} \profilefunction(y_1,\ldots, y_\ell) y^{\kk} \ \dd y \nonumber\\
&= q_{\kk, d,\ell}\int_0^{\pi/2} \! \int_{-\pi}^\pi \int_{[0,\pi]^{\ell-2}}\prod_{k=1}^\ell\sum_{h=0}^{d+s}
(\alpha_{h,\kk}^k\cos(h\varphi_k)+\beta_{h,\kk}^k\sin(h\varphi_k)) \nonumber\\
&\norel \profilefunction\left(\prod_{k=1}^\ell\sin(\varphi_k), \cos(\varphi_1)\prod_{k=2}^\ell\sin(\varphi_k),\ldots,\cos(\varphi_{\ell-1})
\prod_{k=\ell}^\ell\sin(\varphi_k)\right)\ \dd(\varphi_2,\ldots,\varphi_{\ell-1}) \dd \varphi_1  \dd \varphi_\ell \nonumber\\
&= q_{\kk, d,\ell}\int_0^{\pi/2} \! \int_{-\pi}^\pi \int_{[0,\pi]^{\ell-2}}\sum_{h=1}^{2^\ell(d+s+1)^\ell}\tilde{\zeta}_{h,\kk}f_h(\varphi_1,\ldots,\varphi_\ell) \nonumber\\
&\norel \profilefunction\left(\prod_{k=1}^\ell\sin(\varphi_k), \cos(\varphi_1)\prod_{k=2}^\ell\sin(\varphi_k),\ldots,\cos(\varphi_{\ell-1})\prod_{k=\ell}^\ell\sin(\varphi_k)\right)\ \dd(\varphi_2,\ldots,\varphi_{\ell-1}) \dd \varphi_1  \dd \varphi_\ell \nonumber\\
&=\sum_{h=1}^{2^\ell(d+s+1)^\ell}\zeta_{h,\kk}\int_0^{\pi/2} \! \int_{-\pi}^\pi \int_{[0,\pi]^{\ell-2}}
f_h(\varphi_1, \dots, \varphi_\ell) \nonumber\\
&\norel \hspace{1.5cm}\profilefunction\left(\prod_{k=1}^\ell\sin(\varphi_k), \cos(\varphi_1)\prod_{k=2}^\ell\sin(\varphi_k),\ldots,\cos(\varphi_{\ell-1})\prod_{k=\ell}^\ell\sin(\varphi_k)\right) \dd(\varphi_2,\ldots,\varphi_{\ell-1}) \dd \varphi_1  \dd \varphi_\ell \nonumber\\
\label{eq:lgreater1}
&= \sum_{h=1}^{2^\ell(d+s+1)^\ell}\zeta_{h,\kk} \cdot b_h(\profilefunction),
\end{align}
where 
\begin{align*}
&\norel b_{h}(\profilefunction)  \\
&\defeq  \int_0^{\pi/2} \! \int_{-\pi}^\pi \int_{[0,\pi]^{\ell-2}}
f_h(\varphi_1, \dots, \varphi_\ell) 
\profilefunction\left(\prod_{k=1}^\ell\sin(\varphi_k), \cos(\varphi_1)\prod_{k=2}^\ell\sin(\varphi_k),\ldots,\cos(\varphi_{\ell-1})\prod_{k=\ell}^\ell\sin(\varphi_k)\right)\\
&\norel \dd(\varphi_2,\dots,\varphi_{\ell-1}) \dd \varphi_1 \dd \varphi_\ell,
\end{align*}
which only depends on $h$ and $\profilefunction$.

In the case $\ell = 1$, the integral in \eqref{eq:pal_1} evaluates to
\begin{align}
\int_{B^d} g(y_1) y^{\kk} \ \dd y &= q_{\kk,d,1} \cdot \int_{-1}^{1} \profilefunction(y_1)y_1^{\kk_1}(1-y_1^2)^{(\kk_2+\dots+\kk_d + d - 1)/2} \ \dd y_1  \nonumber\\
\overset{y_1 = \sin(\varphi)}&{=} q_{\kk,d,1} \cdot\int_{-\pi/2}^{\pi/2} \profilefunction(\sin(\varphi)) \sin^{\kk_1}(\varphi) \cos(\varphi)^{\kk_2 + \dots + \kk_d + d} \ \dd \varphi \nonumber\\
\overset{\eqref{eq:addzeros}}&{=} q_{\kk,d,1} \cdot\int_{-\pi/2}^{\pi/2} \profilefunction(\sin(\varphi)) \sum_{h=0}^{d+s} \left(\alpha_{h,\kk}^1\cos(h \varphi) + \beta_{h,\kk}^1 \sin(h\varphi)\right) \ \dd \varphi \nonumber\\
\overset{\eqref{eq:zetatilde}}&{=} q_{\kk,d,1} \cdot\sum_{h=1}^{2(d+s+1)} \tilde{\zeta}_{h, \kk} \cdot \int_{-\pi/2}^{\pi/2} \profilefunction(\sin(\varphi)) \cdot f_h(\varphi) \ \dd \varphi \nonumber\\
\label{eq:l1}
\overset{\eqref{eq:zeta}}&{=} \sum_{h=1}^{2(d+s+1)} \zeta_{h,\kk} \cdot b_h(\profilefunction),
\end{align}
with 
\[
b_h(\profilefunction) \defeq \int_{-\pi/2}^{\pi/2} \profilefunction(\sin(\varphi)) \cdot f_h(\varphi) \ \dd \varphi,
\]
which only depends on $h$ and $\profilefunction$.

Finally, let $A \in \RR^{\ell \times d}$ and $\sigma \in O(d)$ with $A \sigma = I_{\ell \times d}$.
Putting everything together, we get
\begin{align*}
\langle \profilefunction_A, P \rangle &= \int_{B^d} \profilefunction(Ax)P(x)\  \dd x \\
\overset{x=\sigma y}&{ = }\int_{B^d} \profilefunction(y_1,\ldots, y_\ell) P(\sigma y) \ \dd y \\
\overset{\eqref{eq:pk}}&{=}\underset{\abs{\kk}\leq s}{\sum}P_{\kk}(\sigma; P)\int_{B^d} \profilefunction(y_1, \dots, y_\ell) y^{\kk} \ \dd y \\
\overset{\eqref{eq:lgreater1},\eqref{eq:l1}}&{=} \underset{\abs{\kk}\leq s}{\sum}P_{\kk}(\sigma; P) \left(\sum_{h=1}^{2^\ell(s + d + 1)^\ell} \zeta_{h,\kk} \cdot b_h(\profilefunction)\right) \\
&= \sum_{h=1}^{2^\ell(s + d + 1)^\ell} \Bigg( b_h(\profilefunction) \cdot 
\underbrace{\underset{\abs{\kk} \leq s}{\sum} \zeta_{h,\kk}\cdot P_{\kk}(\sigma; P)}_{= Q_h(\sigma;P); \text{ see } \eqref{eq:qh}}\Bigg),
\end{align*}
which concludes the proof.
\end{proof}
We are now ready to apply Lemma \ref{lem:polynomial_manifold} and prove \Cref{lem:pi_small}.
\begin{proof}[Proof of \Cref{lem:pi_small}]
Let $\prs$ be the quasi-projection onto $\polyy_s(B^d)$ with range in $\polyy_{2s-1}(B^d)$ as defined in \Cref{sec:ortho}.
Recall that, according to \Cref{eq:projiden}, we have 
\begin{equation}\label{eq:exp}
\prs (f) = \underset{i \in I_{2s-1}}{\sum} a_{i,s} \cdot \langle f, P_i \rangle  P_i \quad \text{for all } f \in L^1(B^d)
\end{equation}
with real coefficients $a_{i,s} \in \RR$ and polynomials $P_i \in \polyy_{2s-1}(B^d)$ \paull{($i \in I_{2s-1}$)} that form an orthonormal basis of $\polyy_{2s-1}(B^d)$.

We now claim that there exist $\{\gamma_{i}^{j}\}_{(i, j) \in I_{2s-1} \times I_{2s-1}}\subset \polyy_{2s-1}(\RR^d)$ satisfying
\begin{align}\label{eq:gammaij}
a_{i,s} \cdot P_{i}(\xi+t)=\sum_{j\in I_{2s-1}}\gamma_{i}^{j}(t)P_{j}(\xi)
\end{align}
for all $i \in I_{2s-1}$ and $\xi\in B^d, \ t \in \RR^d$. 
Indeed, for fixed $t \in \RR^d$ and $i \in I_{2s-1}$, the function defined by $g_t(\xi) \defeq a_{i,s} \cdot P_i(\xi + t)$
is a polynomial of degree at most $2s-1$ in the variable $\xi$.
Since $\{P_j\}_{j \in I_{2s-1}}$ is an orthonormal basis for $\polyy_{2s-1}(B^d)$, we thus get 
\[
g_t(\xi)= \underset{j \in I_{2s-1}}{\sum} \underbrace{\langle g_t, P_j \rangle}_{=: \gamma_i^j(t)} P_j(\xi) \quad \text{for all } \xi \in B^d.
\]
Note that
\[
\gamma_i^j(t) = a_{i,s} \cdot \int_{B^d} P_i(\xi + t)P_j(\xi) \ \dd \xi
\]
are polynomials of degree at most $2s-1$ in $t$.

Let $\xi_1, \dots, \xi_m \in B^d$ be taken as in the beginning of \Cref{sec:signset}.
Set $K \defeq n2^\ell(2s+d)^\ell$ and let 
\begin{equation}\label{eq:bij}
\psi = (\psi_1, \psi_2): \quad \{1,\dots, K\} \to \{1,\dots,n\} \times \{1,\dots,2^{\ell}(2s+d)^\ell\}
\end{equation}
be a bijection.
Further, for $i \in I_{2s-1}$ and $h \in \{1,\dots, 2^{\ell}(2s+d)^\ell\}$, let $Q_h(\bullet; P_i) \in \polyy_{2s-1}(\RR^{d^2})$ be given according to \Cref{lem:inner_product_expansion}.
For $\alpha \in \{1,\dots,m\}$ and $\beta \in \{1,\dots,K\}$, we then set 
\begin{equation}\label{eq:pialphabeta}
\pi_{\alpha,\beta}(\sigma_1,\ldots,\sigma_n,t)\defeq \sum_{i\in I_{2s-1}}\sum_{j\in I_{2s-1}}Q_{\psi_2(\beta)}(\sigma_{\psi_1(\beta)};P_i)\gamma_{i}^{j}(t)P_{j}(\xi_\alpha),
\end{equation}
for $\sigma_1,\dots,\sigma_n \in \RR^{d \times d} \cong \RR^{d^2}$ and $t \in \RR^d$.
By letting $N \defeq nd^2 +d$, we note that $\pi_{\alpha,\beta} \in \polyy_{4s-2}(\RR^N)$.
We claim that 
\begin{equation}\label{eq:sub}
\Pi_{m,s,n,\ell,d} \subseteq \Bigl\{(\pi_1(b,z),\ldots,\pi_m(b,z)):(b,z)\in\RR^K\times\RR^N\Bigr\}
\end{equation}
with $\pi_{\alpha}(b,z) \defeq \sum_{\beta = 1}^K b_\beta \cdot \pi_{\alpha, \beta}(z)$.

To this end, recall \paull{from \eqref{eq:def_pi}} that 
\[
\Pi_{m,s,n,\ell,d} = \Bigl\{(P(\xi_1+t),\ldots,P(\xi_m+t)): \ P\in\Pr\nolimits_s (\rndl),\ t\in\RR^d\Bigr\}.
\]
Let $P \in \prs(\rndl)$.
From \eqref{eq:exp}, we have
\[
P = \sum_{i \in I_{2s-1}} a_{i,s} \cdot \langle R, P_i \rangle  P_i,
\] 
where $R \in \rndl$.
According to \Cref{prop:orthrows}, we can write $R = \sum_{k=1}^n \profilefunction_k(A_k \cdot \bullet)$ with $\profilefunction_k \in \mb(\RR^\ell)$ and $A_kA_k^T = I_{\ell \times \ell}$
for every $k \in \{1,\dots,n\}$, which implies for each $k$ the existence of a matrix $\sigma_k \in O(d)$ with $A_k \sigma_k = I_{\ell \times d}$. 
Thus, Lemma \ref{lem:inner_product_expansion} yields that
\begin{align*}
P=\sum_{i\in I_{2s-1}}\sum_{k=1}^n a_{i,s} \cdot \langle \profilefunction_k(A_k\bullet),P_{i}\rangle P_{i}
=\sum_{i\in I_{2s-1}}\sum_{k=1}^n\sum_{h=1}^{2^\ell(2s+d)^\ell}a_{i,s} \cdot b_h(\profilefunction_k)Q_h(\sigma_k;P_i)P_i
\end{align*}
with coefficients $b_h(\profilefunction_k) \in \RR$ that only depend on $\profilefunction_k$ and $h$.

Recalling the choice of $\gamma_i^j(t)$ in \eqref{eq:gammaij}, the choice of $\psi$ in \eqref{eq:bij}, and the definition of $\pi_{\alpha, \beta}$
in \eqref{eq:pialphabeta} then yields for every $\alpha \in \{1,\dots,m\}$ and $t \in \RR^d$ that 
\begin{align*}
P(\xi_\alpha+t)&=\sum_{i\in I_{2s-1}}\sum_{k=1}^n\sum_{h=1}^{2^\ell(2s+d)^\ell}a_{i,s} \cdot b_h(\profilefunction_k)Q_h(\sigma_k;P_i)P_i(\xi_\alpha + t)\\
&=\sum_{i\in I_{2s-1}}\sum_{k=1}^n\sum_{h=1}^{2^\ell(2s+d)^\ell}\sum_{j\in I_{2s-1}}b_h(\profilefunction_k)Q_h(\sigma_k;P_{i})\gamma_{i}^{j}(t)P_{j}(\xi_\alpha)\\
&=\sum_{k=1}^n\sum_{h=1}^{2^\ell(2s+d)^\ell} b_h(\profilefunction_k)\sum_{i\in I_{2s-1}}\sum_{j\in I_{2s-1}}Q_h(\sigma_k;P_{i})\gamma_{i}^{j}(t)P_{j}(\xi_\alpha) \\
&= \sum_{\beta = 1}^{K} b_{\psi_2(\beta)}(\profilefunction_{\psi_1(\beta)}) \pi_{\alpha, \beta} (\sigma_1, \dots, \sigma_n,t).
\end{align*}
This proves \eqref{eq:sub}.

To apply \Cref{lem:polynomial_manifold}, we need $N+K\leq m/2$ to be satisfied. 
Once we show this, by substituting the values of $N$ and $K$ and applying \Cref{lem:polynomial_manifold} with $4s - 2$ in place of $s$, we would obtain
\begin{align*}
\lvert\sgn(\Pi_{m,s,n,\ell,d})\rvert&\leq(16s-8)^N(N+K+1)^{N+2}\biggl(\frac{2\ee m}{N+K}\biggr)^{N+K} \\
&\leq(16s)^N(N+K+1)^{N+2}\biggl(\frac{2\ee m}{N+K}\biggr)^{N+K}\leq 2^{m/4}
\end{align*}
if $N\lg(16s)+(N+2)\lg(N+K+1)+(N+K)\lg(2 \ee m / (N + K))\leq m/4.$ 
Here, we write $\lg$ for the logarithm with respect to base $2$.
Therefore, if we can prove that 
\begin{equation}\label{eq:main_goal}
N+K\leq m/2 \quad
\text{and} \quad
N\lg(16s)+(N+2)\lg(N+K+1)+(N+K)\lg(2\ee m/(N + K))\leq m/4
 \end{equation}
\paull{for sufficiently large $s \in \NN$ (depending on $d,\ell,c_1$),} then we are done.

Note that $c_0 > 2 \cdot 4^{\ell}/c_1$, which directly implies $c_1/2 > 4^{\ell}/c_0$.
Since it suffices to compare leading coefficients \paul{(with respect to $s$)}, we hence get from \eqref{eq:cond} that 
\begin{equation}\label{eq:m/2bound}
m/2 \geq c_1 s^d/2>(s^{d-\ell}/c_0)d^2+d+(s^{d-\ell}/c_0)2^\ell(2s+d)^\ell \paull{\geq nd^2 + d + n  2^\ell(2s+d)^\ell} = N + K
\end{equation}
for sufficiently large $s \in \NN$ (depending on $d, \ell, c_1$).

To show that $N\lg(16s)+(N+2)\lg(N+K+1)+(N+K)\lg(2\ee m/(N + K))\leq m/4$
for sufficiently large $s$ (depending on $d, \ell, c_1$), we start by showing
\begin{equation}\label{eq:goal}
K \lg(2 \ee m / (N + K)) \leq \tilde{c} \cdot m \quad \text{with } \tilde{c}< \frac{1}{4}
\end{equation}
for sufficiently large $s$ (depending on $d,\ell,c_1$).
\paull{Indeed, by \eqref{eq:cond}, we get}
\begin{align*}
K\lg\biggl(\frac{2 \ee m}{N+K}\biggr)\leq(s^{d-\ell}/c_0)2^\ell(2s+d)^\ell\lg\biggl(\frac{2\ee (2c_1 s^d)}
{\displaystyle(s^{d-\ell}/(2c_0))d^2+d+(s^{d-\ell}/(2c_0))2^\ell(2s+d)^\ell}\biggr).
\end{align*}
We first study the argument of the logarithm. 
By comparing leading coefficients, we note 
\[
\frac{2e(2c_1 s^d)}{\displaystyle(s^{d-\ell}/(2c_0))d^2+d+(s^{d-\ell}/(2c_0))2^\ell(2s+d)^\ell} \quad \to \quad \frac{4\ee c_1 }{4^{\ell}/(2c_0)}
= \frac{2\ee c_1 c_0}{4^{\ell - 1}} \quad (s \to \infty),
\]
which shows that 
\[
\frac{2\ee(2c_1 s^d)}{\displaystyle(s^{d-\ell}/(2c_0))d^2+d+(s^{d-\ell}/(2c_0))2^\ell(2s+d)^\ell}
\leq \frac{\ee c_1 c_0}{4^{\ell - 2}}
\]
for sufficiently large $s$ (depending on $d,\ell,c_1$).
Recalling that $m \geq c_1s^d$ and by comparing leading coefficients, it suffices to ensure
\[
c_1 / 4 > \frac{4^\ell}{c_0} \lg (\ee c_1 c_0 / 4^{\ell - 2})
\]
in order for \eqref{eq:goal} to be satisfied. 
This is equivalent to 
\[
c_0  > \frac{4^{\ell + 1}}{c_1} \lg (\ee c_1 c_0 / 4^{\ell - 2}).
\]
Plugging in $c_0 =  4^{\ell + 3}/ c_1$, we obtain that this is equivalent to 
\[
4^2 > \lg (\ee \cdot 4^5),
\]
which is satisfied since $\lg(\ee \cdot 4^5) = \lg(\ee) + 10 \leq 2 + 10 = 12 < 16 = 4^2$.
Hence, \eqref{eq:goal} is satisfied \paull{for $s$ sufficiently large (depending on $d,\ell,c_1$)}.

Further,  
we note by \eqref{eq:cond} that  
\begin{equation}\label{eq:order1}
N \lg(16s) \leq \left(\frac{d^2}{c_0} s^{d-\ell}+d \right) \lg(16s) \ll s^d,
\end{equation}
i.e., this term is of lower order than $s^d$.
Additionally, we get by \eqref{eq:cond} that 
\begin{equation}\label{eq:order2}
(N+2)\lg(N+K+1) 
\leq \left(\frac{d^2}{c_0} s^{d-\ell} +d+ 2\right) \lg \left(\frac{d^2}{c_0} s^{d-\ell} + d + \frac{2^\ell}{c_0}s^{d-\ell}(2s+d)^\ell \ + 1\right) \ll s^d.
\end{equation}
Lastly, we observe again by \eqref{eq:cond} that 
\begin{equation}\label{eq:order3}
N \lg(2 \ee m /(N + K)) \leq N \lg(2 \ee m) \leq \left(\frac{d^2}{c_0} + d \right) \cdot s^{d-\ell} \cdot \lg(4 \ee c_1 s^d ) \ll s^d.
\end{equation}
\paull{Because of \eqref{eq:cond}}, Equations \eqref{eq:m/2bound}, \eqref{eq:goal}, \eqref{eq:order1}, \eqref{eq:order2} and \eqref{eq:order3} together imply \eqref{eq:main_goal}, which proves the lemma. 
\end{proof}

\subsection{Concluding the proof of the lower bound}\label{sec:lower_final}
In this subsection, we prove \Cref{thm:main_1}. 
To this end, we start by constructing a subset of $\sobb$ consisting of \paull{sums of} smooth bump functions, which we will later show to realize the desired lower bound. 

 Let $d,r \in \NN$ be fixed and set 
\begin{equation}\label{eq:defomega}
\Omega^d \defeq \left[-\frac{1}{\sqrt{d}}, \frac{1}{\sqrt{d}}\right]^d \subseteq B^d.
\end{equation}
Let $\tilde{\omega}: \ \RR^d \to [0,1]$ be a smooth function with
\[
\tilde{\omega}(x) = \begin{cases} 1,& x \in \Omega^d/2, \\ 0,& x \in \RR^d \setminus \Omega^d.\end{cases}
\]
We scale $\tilde{\omega}$ such that its restriction to $B^d$ belongs to $\sobb$ and call this normalization $\omega$.
Let $\vartheta$ and $\xi_1, \dots, \xi_m$ be chosen as at the beginning of \Cref{sec:signset}.
Recall that
\[
E^m = \{ \eps = (\eps_1 ,\dots, \eps_m): \ \eps_i = \pm 1, \ i = 1,\dots,m\}.
\]
Let 
\[
f_\eps: \quad B^d \to \RR, \quad  f_\eps(x) \defeq (2\vartheta)^{-r}\sum_{i=1}^m \eps_i \cdot \omega(2\vartheta(x- \xi_i))
\]
for $\eps \in E^m$ and define
\begin{equation}\label{eq:fmrd}
\mathcal{F}^{m,r,d} \defeq \left\{f_\eps: \ \eps \in E^m\right\}.
\end{equation}
\begin{proposition}\label{prop:fmsubset}
For any choice of $m,r,d \in \NN$ it holds that $\mathcal{F}^{m,r,d} \subseteq \sobb$.
\end{proposition}
\begin{proof}
Since $\omega$ is smooth, we conclude $\mathcal{F}^{m,r,d} \subseteq C^\infty(B^d)$. 
Let $\eps \in E^m$.
We want to show that $\abs{\partial^{\kk} f_\eps(x)} \leq 1$ for all $x \in B^d$ and every multiindex $\kk \in \NN_0^d$ with $\abs{\kk} \leq r$.
For $i \in \{1,\dots,m\}$ we let $\Omega_i \defeq \xi_i + \Omega^d/(2\vartheta)$ denote the closed cube with center $\xi_i$ and side length $1/(\vartheta\sqrt{d})$.
 Since $\partial \Omega_i$ is a null set for every $i$, and $f_\eps$ and all of its derivatives are continuous, 
 we may assume $x \in \left(B^d\right)^\circ \setminus \bigcup_{i=1}^m \partial \Omega_i$.
 Moreover, since $\infnorm{\xi_i - \xi_j} \geq \frac{1}{\sqrt{d}\vartheta}$ 
 for every $i,j \in \{1,\dots,m\}$ with $i \neq j$, we note that there exists at most one $i \in \{1,\dots,m\}$
 with $\infnorm{x- \xi_i} < \frac{1}{2\sqrt{d}\vartheta}$,
 which is equivalent to $x \in \Omega_i^\circ$.
 
 In the case such $i$ exists, we get $2\vartheta(x - \xi_j) \notin \Omega^d$ for every $j \neq i$, whence, by definition of $\omega$, we have
 \[
 f_\eps(x) = (2\vartheta)^{-r} \cdot \eps_i \cdot \omega(2\vartheta(x-\xi_i)). 
 \]
 Since $x \in \Omega_i^\circ$, this identity even holds in an open neighborhood of $x$.
 Therefore, for any multiindex $\kk \in \NN_0^d$ with $\abs{\kk} \leq r$, we have
 \[
 \partial^\kk f_\eps(x) = (2\vartheta)^{\abs{\kk} - r} \cdot   \eps_i \cdot \left(\partial^{\kk}\omega\right)(2\vartheta(x- \xi_i)).
 \]
Using that $2\vartheta \geq 1$ and $\omega \in \sobb$, we get
 \[
 \abs{ \partial^\kk f_\eps(x)} =  (2\vartheta)^{\abs{\kk} - r} \cdot \abs{\left(\partial^{\kk}\omega\right)(2\vartheta(x- \xi_i))} \leq 1.
 \]
 
In the case when $x \notin \bigcup_{i=1}^m \Omega_i$, we have $f_\eps \equiv 0$ in an open neighborhood of $x$.
Combining the two cases concludes the proof. 
\end{proof}

We can now establish a lower bound for the approximation of functions from $\mathcal{F}^{m,r,d}$ using functions from $\prs(\rndl)$ in the $\Vert \cdot \Vert_{L^1(B^d)}$-norm. 
This is in fact a generalization of \cite[Lemma~7]{maiorov2010best} to the case of multivariate ridge functions. 
\begin{proposition}\label{prop:low_bound_eps}
Let $d, \ell, r \in \NN$ with $\ell < d$. 
Then there exists a constant $c = c(d,\ell,r) > 0$ with the following property: For any $c_1> 0$ and $c_0 \defeq 4^{\ell + 3} /c_1$,
there exists a constant $C = C(d, \ell, c_1)$ such that for any choice of 
$m,s,n\in \NN$ with  
\[
c_1 s^d \leq m \leq 2c_1 s^d, \quad c_0 n \leq s^{d-\ell} \leq 2c_0 n, \quad \mathrm{and} \quad s \geq C,
\]
there exists a function $f_{\eps^\ast} \in \mathcal{F}^{m,r,d}$ satisfying 
\[
\underset{P \in \prs(\rndl)}{\inf} \mnorm{f_{\eps^\ast} - P}_{L^1(B^d)} \geq c \cdot  c_1^{(r/(d-\ell))-(r/d)} \cdot n^{-r/(d-\ell)}.
\]
Here, $\mathcal{F}^{m,r,d}$ is as defined in \eqref{eq:fmrd}.
\end{proposition}
\begin{proof}
For a given $c_1 > 0$, we let $C = C(d,\ell, c_1)>0$ be chosen according to \Cref{lem:pi_small}.
Let $f_\eps \in \mathcal{F}^{m,r,d}$ and $P \in \prs(\rndl)$ be arbitrary. 
We recall that then 
\[
f_\eps(x) = (2\vartheta)^{-r} \sum_{i=1}^m \eps_i \cdot \omega(2\vartheta(x- \xi_i))
\]
 for given grid points $\xi_1,\dots,\xi_m \in \Omega^d$ with $\Omega^d$ as in \eqref{eq:defomega}.
 Moreover, recall from the start of \Cref{sec:signset} that $\vartheta \in \NN$ is chosen such that $\frac{m^{1/d}}{2} \leq \vartheta \leq m^{1/d}$.
With the notation $\Omega_i \defeq \xi_i + \Omega^d/(2\vartheta)$, we get $B^d \supseteq \bigcup_{i=1}^m \Omega_i^\circ$, where the union is disjoint. 
This yields
\[
\mnorm{f_\eps - P}_{L^1(B^d)} \geq \sum_{i=1}^m \int_{\Omega_i^\circ} \abs{f_\eps(x) - P(x)} \ \dd x =  \sum_{i=1}^m \int_{\Omega_i} \abs{f_\eps(x) - P(x)} \ \dd x,
\]
where in the last equality we used that $\partial \Omega_i$ is a measure zero set. 
Applying the change of variables $x \defeq t + \xi_i$, we conclude
\begin{align*}
\sum_{i=1}^m \int_{\Omega_i} \abs{f_\eps(x) - P(x)} \ \dd x 
&= \int_{\Omega^d/(2\vartheta)}\sum_{i=1}^m  \abs{f_\eps(t + \xi_i) - P(t+\xi_i)} \ \dd t \\
&\geq \int_{\Omega^d/(4\vartheta)}\sum_{i=1}^m  \abs{f_\eps(t + \xi_i) - P(t+\xi_i)} \ \dd t.
\end{align*}
For any $t\in \Omega^d/(4\vartheta)$ and fixed $i \in \{1,\dots,m\}$, since $\xi_i + t \in \Omega_i^\circ$ and $2\vartheta t \in \Omega^d/2$, it follows from the 
choice of $\omega$ that
\begin{align*}
f_\eps(\xi_i + t) = (2\vartheta)^{-r}\cdot \eps_i \cdot \omega(2\vartheta t) = (2\vartheta)^{-r} \cdot \eps_i \cdot c_2
\end{align*}
with an absolute constant $c_2 = c_2(d,r)>0$.
Hence, for any $t \in \Omega^d/(4\vartheta)$ we get
\begin{align}
\sum_{i=1}^m  \abs{f_\eps(t + \xi_i) - P(t+\xi_i)}
 &\geq \underset{\tau \in \RR^d}{\underset{P' \in \prs(\rndl)}{\inf}} \sum_{i=1}^m  \abs{(2\vartheta)^{-r} \cdot \eps_i \cdot c_2 - P'(\xi_i + \tau)} \nonumber\\
 \label{eq:eq_low}
 & = c_2 \cdot (2\vartheta)^{-r} \cdot \underset{\tau \in \RR^d}{\underset{P' \in \prs(\rndl)}{\inf}} \sum_{i=1}^m \abs{\eps_i - P'(\xi_i + \tau)}.
\end{align}
At the last step, we used that the set $\rndl$ is invariant under multiplication with nonzero factors and that the map $\prs$ is linear, whence the
scaling invariance translates to $\prs(\rndl)$.
Using the bound \eqref{eq:eq_low}, we get
\begin{align*}
\mnorm{f_\eps - P}_{L^1(B^d)} &\geq {\metalambda}^d\left(\Omega^d/(4\vartheta)\right) \cdot c_2 \cdot (2\vartheta)^{-r}
 \cdot \underset{\tau \in \RR^d}{\underset{P' \in \prs(\rndl)}{\inf}} \sum_{i=1}^m \abs{\eps_i - P'(\xi_i + \tau)} \\
 &= \left(\frac{1}{2\vartheta \cdot \sqrt{d}}\right)^d
 \cdot c_2 \cdot (2\vartheta)^{-r}
 \cdot \underset{\tau \in \RR^d}{\underset{P' \in \prs(\rndl)}{\inf}} \sum_{i=1}^m \abs{\eps_i - P'(\xi_i + \tau)} \\
 \overset{\vartheta \leq m^{1/d}}&{\geq} c_2 \cdot c_3 \cdot \frac{1}{m \cdot (2\vartheta)^r} \cdot \underset{x \in \Pi_{m,s,n,\ell,d}}{\inf} \onenorm{x-\eps},
\end{align*}
where we set $c_3 = c_3(d) \defeq \left(\frac{1}{2\sqrt{d}}\right)^d$ \paull{and recall that $\Pi_{m,s,n,\ell,d}$ was defined in \eqref{eq:def_pi}}.
From \Cref{lem:epsbound,lem:pi_small}, we infer, under the constraint that $s \geq C$, the existence of $\eps^\ast \in E^m$ with 
\[
\underset{x \in \Pi_{m,s,n,\ell,d}}{\inf} \onenorm{x-\eps^\ast} \geq am
\]
with an absolute constant $a>0$.
Hence, we get 
\[
\mnorm{f_{\eps^\ast} - P}_{L^1(B^d)} \geq a \cdot c_2 \cdot c_3 \cdot (2\vartheta)^{-r} \overset{\vartheta \leq m^{1/d}}{\geq} a \cdot c_2 \cdot c_3 \cdot 2^{-r} \cdot m^{-r/d}.
\]
By assumption, we have $m \leq 2c_1 \cdot s^d$, which yields
\[
\mnorm{f_{\eps^\ast} - P}_{L^1(B^d)} \geq a \cdot c_2 \cdot c_3 \cdot 2^{-r} \cdot (2c_1)^{-r/d} \cdot s^{-r}.
\]
Furthermore, since $s \leq (2c_0)^{1/(d-\ell)}\cdot n^{1/(d-\ell)}$, we get 
\[
\mnorm{f_{\eps^\ast} - P}_{L^1(B^d)} \geq a \cdot c_2 \cdot c_3 \cdot 2^{-r} \cdot (2c_1)^{-r/d} \cdot (2c_0)^{-r/(d-\ell)} \cdot n^{-r/(d-\ell)}.
\]
Since $c_0 = \frac{4^{\ell + 3}}{c_1}$, this translates to 
\[
\mnorm{f_{\eps^\ast} - P}_{L^1(B^d)} \geq a \cdot c_2 \cdot c_3 \cdot 2^{-r} \cdot 2^{-r/d} \cdot 2^{-r(2\ell + 7)/(d-\ell)} \cdot c_1^{(r/(d-\ell))-(r/d)}  \cdot n^{-r/(d-\ell)}.
\]

Therefore, the claim follows letting $c = c(d,\ell,r) \defeq a \cdot c_2 \cdot c_3 \cdot 2^{-r} \cdot 2^{-r/d} \cdot 2^{-r(2\ell + 7)/(d-\ell)}$.
\end{proof}
The following lemma is a central ingredient in order to prove the final approximation bound and follows immediately from the fact that the quasi-projection 
$\prs$ is bounded with respect to the
$L^1$-norm, with the bound being independent of $s$.
\begin{lemma}[{corrected version of \cite[Lemma~5]{maiorov2010best}}] \label{lem:prob_wrong}
Let $d \in \NN$. 
Then there exists a positive constant $c = c(d)>0$ with the following property:
For any choice of $\ell, s, n \in \NN$ with $\ell< d$ and $P \in \polyy_s(B^d)$  
it holds
\[
\underset{R \in \rndl}{\inf} \mnorm{P-R}_{L^1(B^d)} \geq c \cdot \underset{P' \in \prs(\rndl)}{\inf} \mnorm{P-P'}_{L^1(B^d)}.
\] 
\end{lemma}
\begin{proof}
Let $P \in \polyy_s(B^d)$ and $R \in \rndl$ be arbitrary. 
Since $P-R$ is bounded on $B^d$, we clearly have $P-R \in L^1(B^d)$.
Then we have 
\[
\mnorm{P-R}_{L^1(B^d)} \geq c \cdot \mnorm{\prs (P-R) }_{L^1(B^d)} = c \cdot \mnorm{P - \prs(R) }_{L^1(B^d)}
\]
by the properties of $\prs$ from \Cref{prop:prs}. 
By taking the infimum on the right-hand side, we get 
\[
\mnorm{P-R}_{L^1(B^d)} \geq c \cdot \underset{P' \in \prs(\rndl)}{\inf} \mnorm{P-P'}_{L^1(B^d)}.
\]
Now we can take the infimum over $R \in \rndl$ on the left-hand side and obtain the claim. 
\end{proof}
We now have everything we need to complete our proof of the lower bound. 
\begin{theorem}\label{thm:main}
Let $d,\ell, r \in \NN$ with $\ell < d$. Then there exists a constant $c = c(d,\ell,r) >0$ with the property that for any $n \in \NN$ we have
\[
\underset{f \in \sobb}{\sup} \ \underset{R \in \rndl}{\inf} \mnorm{f - R}_{L^1(B^d)} \geq c \cdot n^{-r/(d-\ell)}.
\]
\end{theorem}
\begin{proof}
Let $c_1 \paul{= c_1(d,\ell,r)}> 0$ be arbitrary (to be determined later) and set $c_0 \defeq 4^{\ell + 3} / c_1$.
Let the two constants $\tilde{c} = \tilde{c}(d, \ell, r)>0$ and $\tilde{C} = \tilde{C}(d, \ell, c_1)>0$ be given by \Cref{prop:low_bound_eps}. 

For the moment, we assume that $n \geq N$ with a large number $N = N(d,\ell,c_1) \in \NN$ such that the conditions 
\[
n^{1/(d-\ell)} \left((2c_0)^{1/(d-\ell)} - c_0^{1/(d-\ell)}\right) \geq 1, \quad c_0^{1/(d-\ell)} \cdot n^{1/(d-\ell)} \geq \tilde{C} \quad \text{and} 
\quad c_0^{1/(d-\ell)} \cdot n^{1/(d-\ell)} \geq c_1^{-1/d}
\]
are satisfied. 
The first condition ensures that we can pick a natural number $s \in \NN$ that satisfies $c_0 n \leq s^{d-\ell} \leq 2c_0 n$.
From the second condition, we infer that this implies $s  \geq \tilde{C}$. 
And lastly, the third condition guarantees that 
\[
s \geq c_0^{1/(d-\ell)} \cdot n^{1/(d-\ell)} \geq c_1^{-1/d} \quad \text{and thus} \quad (2c_1) s^d -c_1s^d \geq 1,
\]
whence we can pick a natural number $m \in \NN$ with $c_1s^d \leq m \leq (2c_1)s^d$.
To summarize, we have
\[
c_1 s^d \leq m \leq 2c_1 s^d, \quad c_0 n \leq s^{d-\ell} \leq 2c_0 n \quad \text{and} \quad s \geq \tilde{C}.
\]
According to \Cref{prop:low_bound_eps}, we can thus pick a function $f_\eps \in \mathcal{F}^{m,r,d}$ that satisfies 
\begin{equation}\label{eq:star}
\underset{P \in \prs(\rndl)}{\inf} \mnorm{P - f_\eps}_{L^1(B^d)} \geq \tilde{c} \cdot c_1^{(r/(d-\ell))- (r/d)} \cdot n^{-r/(d-\ell)}.
\end{equation}
Moreover, using \Cref{prop:jack,prop:fmsubset}, we pick a polynomial $g_\eps \in \polyy_s(B^d)$ with 
\[
\mnorm{g_\eps - f_\eps}_{L^1(B^d)} \leq C_2 \cdot s^{-r},
\]
with an absolute constant $C_2 = C_2(d,r)> 0$.
We thus get
\begin{align*}
 \underset{R \in \rndl}{\inf} \mnorm{R - f_\eps}_{L^1(B^d)} 
 &\geq  \underset{R \in \rndl}{\inf} \mnorm{g_\eps - R}_{L^1(B^d)} - \mnorm{g_\eps - f_\eps}_{L^1(B^d)} \\
 &\geq c' \cdot \underset{P \in \prs(\rndl)}{\inf} \mnorm{g_\eps - P}_{L^1(B^d)} - C_2s^{-r},
\end{align*}
where $c' = c'(d) \paull{\in (0,1]}$ is provided by \Cref{lem:prob_wrong} \paull{(we can assume that $c' \leq 1$ by shrinking it, if necessary)}. 
This can be further bounded via
\paull{
\begin{align*}
 &\norel c' \cdot \underset{P \in \prs(\rndl)}{\inf} \mnorm{g_\eps - P}_{L^1(B^d)} - C_2s^{-r} \\
 &\geq c' \cdot \left(\underset{P \in \prs(\rndl)}{\inf} \mnorm{f_\eps - P}_{L^1(B^d)} - \mnorm{f_\eps - g_\eps}_{L^1(B^d)}\!\right)- C_2s^{-r} \\
 &\geq c' \cdot\underset{P \in \prs(\rndl)}{\inf} \mnorm{f_\eps - P}_{L^1(B^d)} - c'C_2 s^{-r}-C_2 s^{-r} \\
 \overset{c' \leq 1}&{\geq} c' \cdot\underset{P \in \prs(\rndl)}{\inf} \mnorm{f_\eps - P}_{L^1(B^d)} - 2C_2 s^{-r} \\
 &\geq  c' \tilde{c} \cdot c_1^{(r/(d-\ell))- (r/d)}\cdot n^{-r/(d-\ell)}- 2C_2 s^{-r},
\end{align*}
}
where we used \eqref{eq:star}.
Since $s \geq c_0^{1/(d-\ell)} \cdot n^{1/(d-\ell)}$ and $c_0 = 4^{\ell +3}/c_1$, we get
\begin{align*}
 \underset{R \in \rndl}{\inf} \mnorm{R - f_\eps}_{L^1(B^d)} &\geq 
c'\tilde{c} \cdot c_1^{(r/(d-\ell))- (r/d)}\cdot n^{-r/(d-\ell)}- 2C_2 s^{-r} \\
&\geq c'\tilde{c} \cdot c_1^{(r/(d-\ell))- (r/d)}\cdot n^{-r/(d-\ell)}- 2C_2 c_0^{-r/(d-\ell)} n^{-r/(d-\ell)} \\
&= \left( c'\tilde{c} \cdot c_1^{(r/(d-\ell))- (r/d)} - 2C_2 \cdot 4^{-r(\ell + 3) /(d-\ell)} \cdot c_1^{r/(d-\ell)}\right) \cdot n^{-r/(d-\ell)}.
\end{align*}
We set $C_3 = C_3(d,\ell, r) \defeq 2C_2 \cdot 4^{-r(\ell + 3) /(d-\ell)}$. We then see
\begin{align*}
c'\tilde{c} \cdot c_1^{(r/(d-\ell))- (r/d)} - C_3 \cdot c_1^{r/(d-\ell)} > 0  \quad \Leftrightarrow \quad c'\tilde{c} \cdot c_1^{-r/d} > C_3 \quad \Leftrightarrow \quad 
c_1 < \left(\frac{c'\tilde{c}}{C_3}\right)^{d/r}.
\end{align*}
Therefore, by picking $c_1 = c_1(d, \ell, r)>0$ small enough, we can achieve the desideratum by letting 
\[
c = c(d, \ell, r) \defeq c'\tilde{c} \cdot c_1^{(r/(d-\ell))- (r/d)} - C_3 \cdot c_1^{r/(d-\ell)}
\]
and noting that $f_\eps \in \sobb$ according to \Cref{prop:fmsubset}.

Recall that we assumed that $n \geq N$. 
In order to obtain the desired claim for any $n \in \NN$, we first note that $\rndl \subseteq \mathcal{R}^\ast_{N,d,\ell}$ for any $n \leq N$. 
The above therefore tells us 
\[
\underset{f \in \sobb}{\sup} \ \underset{R \in \rndl}{\inf} \mnorm{R - f}_{L^1(B^d)}
\geq \underset{f \in \sobb}{\sup} \ \underset{R \in \mathcal{R}^\ast_{N,d,\ell}}{\inf} \mnorm{R - f}_{L^1(B^d)}> 0 \quad \text{for any } n \leq N.
\]
Hence, we obtain the claim by possibly shrinking the size of the constant $c$. 
\end{proof}

It is now straightforward to generalize the claim to the case of approximation of Sobolev functions from $\sob$ with respect to $\mnorm{\cdot}_{L^q}$ for arbitrary $p,q \in [1, \infty]$.
\newcommand{\rrrndl}{R_{n,d,\ell}}
\begin{corollary}\label{corr:low_final}
Let $d,\ell, r \in \NN$ with $\ell < d$ and $p,q \in [1, \infty]$.
Then there exists a positive constant $c = c(d, \ell, p, q, r)>0$ with the property that for any $n \in \NN$ we have 
\[
\underset{f \in \sob}{\sup} \ \underset{R \in \rrndl}{\inf} \mnorm{f - R}_{L^q(B^d)} \geq c \cdot n^{-r/(d-\ell)}.
\]
Here, $\rrndl$ is as defined in \Cref{eq:rrndl}.
\end{corollary}
\begin{proof}
Let $c' = c'(d,\ell,r)>0$ be chosen according to \Cref{thm:main}.
A direct application of \Cref{prop:infequal} then yields 
\[
\underset{f \in \sobb}{\sup} \ \underset{R \in \rrndl}{\inf} \mnorm{f - R}_{L^1(B^d)} \geq c' \cdot n^{-r/(d-\ell)}.
\]
Moreover, according to Hölder's inequality, we may pick a constant $C_1 = C_1(d,p)>0$ which satisfies 
\[
\mnorm{f}_{L^p(B^d)} \leq C_1 \cdot \mnorm{f}_{L^\infty (B^d)} \quad \text{for all measurable } f: B^d \to \RR.
\]
This implies $c_2 \cdot \sobb \subseteq \sob$ for a constant $c_2 = c_2(d,r,p)>0$, which yields 
\begin{align*}
\underset{f \in \sob}{\sup} \ \underset{R \in \rrndl}{\inf} \mnorm{f - R}_{L^1(B^d)}
&\geq \underset{f \in c_2 \cdot \sobb}{\sup} \ \underset{R \in \rrndl}{\inf} \mnorm{f-R}_{L^1(B^d)} \\
&= \underset{f \in \sobb}{\sup} \ \underset{R \in \rrndl}{\inf} \mnorm{c_2 \cdot f - R}_{L^1(B^d)} \\
&= c_2 \cdot \underset{f \in \sobb}{\sup} \ \underset{R \in \rrndl}{\inf} \mnorm{f - R}_{L^1(B^d)} ,
\end{align*}
where the last equality uses the fact that the set $\rrndl$ is invariant with respect to scaling. 
Moreover, there exists another constant $C_3 = C_3(d,q)>0$ with 
\[
\mnorm{f}_{L^1(B^d)} \leq C_3 \cdot \mnorm{f}_{L^q (B^d)} \quad \text{for all measurable } f: B^d \to \RR,
\]
again according to Hölder's inequality. 
Therefore,
\begin{align*}
\underset{f \in \sob}{\sup}\  \underset{R \in \rrndl}{\inf} \mnorm{R - f}_{L^q(B^d)} 
&\geq C_3^{-1} \cdot \underset{f \in \sob}{\sup} \ \underset{R \in \rrndl}{\inf} \mnorm{R - f}_{L^1(B^d)} \\
&\geq c_2 \cdot C_3^{-1} \cdot \underset{f \in \sobb}{\sup} \ \underset{R \in \rrndl}{\inf} \mnorm{R - f}_{L^1(B^d)} \\
&\geq c_2 \cdot C_3^{-1} \cdot c' \cdot n^{-r/(d-\ell)}.
\end{align*}
This yields the claim by picking $c = c(d,\ell,p,q,r) \defeq c_2 \cdot C_3^{-1} \cdot c'$.
\end{proof}

\newcommand{\poly}{\paull{\mathcal{R}}_{n,d,\ell}^{\mathrm{poly}}}
\section{Proof of the upper bound}\label{sec:up_bound}
In this section, we prove the upper bound $\paull{\mathcal{O}}\left(n^{-r/(d-\ell)}\right)$ for the rate of approximating Sobolev functions using sums of \paull{$n$} multivariate ridge functions, which is the optimal
rate according to \Cref{corr:low_final}.
Moreover, we show that this optimal rate can even be attained using sums of \emph{polynomial} ridge functions
with \emph{fixed} matrices $A_1,\dots,A_n \in \RR^{\ell \times d}$.
To this end, for fixed $A_1, \dots, A_n \in \RR^{\ell \times d}$, we introduce the notation
\[
\poly(A_1, \dots, A_n) \defeq \left\{ B^d \ni x \mapsto \sum_{j=1}^n P_j(A_j x): \ P_j \in \polyy(\RR^\ell)\right\}.
\]
Furthermore, we denote by $\polyy_s^h(\RR^{d-\ell + 1})$ the space of \emph{homogeneous} polynomials in $d-\ell + 1$ variables of degree $s$, i.e., 
\[
\polyy_s^h(\RR^{d-\ell + 1}) \defeq \left\{\RR^{d-\ell + 1} \ni x \mapsto \underset{\abs{\kk} = s}{\sum_{\kk \in \NN_0^{d-\ell + 1}}} a_\kk x^\kk: \ a_\kk \in \RR\right\}.
\]
We start by showing that each (not necessarily homogeneous) polynomial of degree at most $s$ can be written as the sum of $n$ $\ell$-variate polynomial ridge functions, 
if we choose the number $n$ to exceed the dimension of the space $\polyy_s^h(\RR^{d-\ell + 1})$.
This result is mainly based on results from \cite[Section~5]{pinkus_ridge_2016}.
\begin{proposition}\label{prop:up_1}
Let $d,n,\ell,s \in \NN$ with $\ell < d$ and  
\begin{equation*}
\dim \left(\polyy_s^h (\RR^{d-\ell + 1})\right) \leq n.
\end{equation*}
Then there exist matrices $A_1, \dots, A_n \in \RR^{\ell \times d}$ with 
\begin{equation*}
\polyy_s(B^d) \subseteq \poly(A_1,\dots, A_n).
\end{equation*}
\end{proposition}
\begin{proof}
Using \cite[Proposition~5.9]{pinkus_ridge_2016} and $\dim \left(\polyy_s^h (\RR^{d-\ell + 1})\right) \leq n$, we can pick $a_1,\dots,a_n \in \RR^{d-\ell + 1}$ satisfying
\begin{equation*}
\polyy_s^h(\RR^{d- \ell + 1}) = \spann \left\{\RR^{d-\ell + 1} \ni x \mapsto (a_i^T x)^s: \ 1 \leq i \leq n\right\}.
\end{equation*}
From \cite[Corollary~5.12]{pinkus_ridge_2016}, we infer that 
\begin{equation*}
\polyy_s(\RR^{d-\ell + 1}) = \spann \left\{ \RR^{d-\ell + 1} \ni x \mapsto (a_i^Tx)^k: \ 1 \leq i \leq n, \ 0 \leq k \leq s\right\}.
\end{equation*}
Let $P \in \polyy_s(\RR^d)$ be arbitrary. With $x \in \RR^{d-\ell + 1}$ and $y \in \RR^{\ell - 1}$, we can then write
\begin{equation*}
P(x,y) = \underset{\abs{\kk} \leq s}{\sum_{\kk \in \NN_0^{\ell - 1}}} y^{\kk} P_{\kk}(x),
\end{equation*}
with suitably chosen polynomials $P_{\kk} \in \polyy_s(\RR^{d-\ell + 1})$.
Moreover, each $P_{\kk}$ can be written as 
\begin{equation*}
P_{\kk}(x) = \underset{0 \leq k \leq s}{\sum_{1 \leq i \leq n}} \alpha_{i,k,\kk} (a_i^Tx)^k
= \sum_{i=1}^n \underbrace{\sum_{k= 0}^s \alpha_{i,k,\kk} (a_i^Tx)^k}_{=: Q_{i,\kk} ( a_i^Tx)}
\end{equation*}
with suitable real coefficients $\alpha_{i,k,\kk} \in \RR$ and
with $Q_{i,\kk} \in \polyy_s(\RR)$ for each $i \in \{1,\dots,n\}$ and $\kk \in \NN_0^{\ell - 1}$ with $\abs{\kk} \leq s$.

This gives us
\begin{align*}
P(x,y) &= \underset{\abs{\kk} \leq s}{\sum_{\kk \in \NN_0^{\ell - 1}}} y^{\kk} P_{\kk}(x) 
= \underset{\abs{\kk} \leq s}{\sum_{\kk \in \NN_0^{\ell - 1}}} y^{\kk} \sum_{i=1}^n Q_{i, \kk}(a_i^Tx)
= \sum_{i=1}^n \underset{\abs{\kk} \leq s}{\sum_{\kk \in \NN_0^{\ell - 1}}} y^{\kk} Q_{i, \kk}(a_i^Tx). 
\end{align*}
For $i \in \{1,\dots,n\}$ we define $H_i \in \polyy(\RR^\ell)$ as
\begin{equation*}
H_i: \quad \RR^\ell = \RR \times \RR^{\ell - 1}\to \RR, \quad (t,y) \mapsto \underset{\abs{\kk} \leq s}{\sum_{\kk \in \NN_0^{\ell - 1}}} y^{\kk} Q_{i, \kk}(t)
\end{equation*}
and furthermore $A_i \in \RR^{\ell \times d}$ as 
\paull{
\renewcommand{\arraystretch}{1.5}
\[
A_i \defeq \left(\begin{array}{c|c}
   a_i^T &  0_{1 \times (\ell -1)} \\
   \hline
   0_{(\ell - 1) \times (d-\ell + 1)}& I_{(\ell - 1) \times (\ell - 1)}
\end{array}\right),
\]
}\renewcommand{\arraystretch}{1}
whence it holds that
\begin{equation*}
H_i \left(A_i \begin{pmatrix} x \\ y \end{pmatrix}\right) = \underset{\abs{\kk} \leq s}{\sum_{\kk \in \NN_0^{\ell - 1}}} y^{\kk} Q_{i, \kk}(a_i^Tx)
\end{equation*}
for $x \in \RR^{d-\ell + 1}$ and $y \in \RR^{\ell - 1}$. 
We thus get
\begin{equation*}
P(x,y) = \sum_{i=1}^n \ H_i \left(A_i \begin{pmatrix} x \\ y \end{pmatrix}\right).
\end{equation*}
Since $P \in \polyy_s(\RR^d)$ was arbitrary, the claim is shown. 
\end{proof}
We can now prove the upper bound. 
\begin{theorem}\label{thm:pqstatement}
Let $d, \ell,r \in \NN$ with $d > \ell$ and $1 \leq q \leq p \leq \infty$. Then there exists a positive constant $C = C(d,\ell, p,q,r) > 0$ with the following property:
For any $n \in \NN$ there exist matrices $A_1, \dots, A_n \in \RR^{\ell \times d}$ such that for any $f \in \sob$ there exists a function $R \in \poly(A_1, \dots, A_n)$ with 
\begin{equation*}
\Vert f - R \Vert_{L^q(B^d)} \leq C \cdot n^{-r/(d-\ell)}.
\end{equation*}
\end{theorem}
\begin{proof}
Let $C_1 = C_1(d,\ell)>0$ be a constant satisfying 
\begin{equation*}
\dim \left(\polyy_s^h (\RR^{d-\ell + 1})\right) \leq C_1 \cdot s^{d - \ell} \quad \text{for every } s \in \NN;
\end{equation*}
we refer to \cite[Lemma~F.1]{geuchen2024optimal} for a proof of the existence of such a constant. 
We may assume $n \geq C_1$ at the cost of possibly increasing $C$ in the end, similar to the end of the proof of \Cref{thm:main}.
We then pick $s \in \NN$ as the maximal number with 
\begin{equation*}
C_1 \cdot s^{d-\ell} \leq n.
\end{equation*}
Note that this implies
\begin{equation*}
C_1 \cdot (2s)^{d-\ell} > n \quad \Leftrightarrow \quad C_1 \cdot 2^{d-\ell} \cdot s^{d-\ell} > n \quad \Leftrightarrow \quad s > c_2 \cdot n^{1/(d-\ell)}
\end{equation*}
with \paull{$c_2 = c_2(d,\ell)\defeq C_1^{1/(\ell-d)} \cdot \frac{1}{2}$}.
Moreover, note that according to \Cref{prop:up_1} the inclusion $\polyy_s(B^d) \subseteq \poly(A_1, \dots, A_n)$ holds for a specific choice of $A_1, \dots, A_n \in \RR^{\ell \times d}$.
Let $f \in \sob$ be arbitrary. 
According to \Cref{prop:jack}, we can pick $P \in \polyy_s(B^d) \subseteq \poly(A_1, \dots, A_n)$ with 
\begin{equation*}
\Vert f - P \Vert_{L^q(B^d)} \leq C_3 \cdot s^{-r} \leq C_3 \cdot c_2^{-r} \cdot n^{-r/(d-\ell)}
\end{equation*}
with an absolute constant $C_3 = C_3(d,p,q,r)>0$. 
Hence, in the end the claim follows defining the constant $C \defeq C_3 \cdot c_2^{-r}$.
\end{proof}

\begin{remark}
The proof of \Cref{thm:pqstatement} shows that the only property of Sobolev functions that is actually needed in order to obtain the approximation rate is the fact
that these functions have the property stated in \Cref{prop:jack}.
Therefore, one can establish the same approximation rate for the set consisting of all $L^p$-functions $f: B^d \to \RR$ for which 
\[
\underset{s \in \NN}{\sup}\  \underset{P \in \polyy_s(B^d)}{\inf} \left(\mnorm{f - P}_{L^q(B^d)} \cdot s^r \right) \leq C,
\]
where $C$ is a fixed constant.
Note that the final bound depends on the constant $C$.
\end{remark}
\newcommand{\nn}{\mathcal{CVNN}^{\phi}_{d,n}}
\section{Application to neural networks}\label{sec:nn}
In this section, we apply the results obtained in the previous sections to the case of shallow neural networks. 
More specifically, we show upper and lower bounds for the approximation of Sobolev functions using 
\emph{shallow generalized translation networks} and \emph{shallow complex-valued neural networks}. 
Here, for $d,\ell,n \in \NN$ 
and any (activation) function $\paul{\tau}: \RR^\ell \to \RR$ we define the set of \emph{shallow generalized translation networks} with $d$ input neurons, activation dimension $\ell$, $n$ 
hidden-layer neurons and activation function $\paul{\tau}$ as 
\newcommand{\nng}{\mathcal{NN}^{\paul{\tau}}_{d,\ell,n}}
\newcommand{\nngg}{\mathcal{NN}^{\paul{\tau}, \ast}_{d,\ell,n}}
\begin{align*}
&\nng \defeq \\
&\left\{ \RR^d \ni x \mapsto \sum_{k=1}^n c_k \paul{\tau}(A_k x + b_k) \in \RR: \ A_1, \dots, A_n \in \RR^{\ell \times d}, \ b_1, \dots, b_n \in \RR^\ell, \ c_1, \dots, c_n \in \RR\right\}.
\end{align*}
Note that we obtain a classical shallow neural network in the case $\ell = 1$. 
For technical reasons, we further introduce the set 
\begin{align*}
&\nngg \defeq \\
&\left\{ \RR^d \ni x \mapsto \sum_{k=1}^n c_k \paul{\tau}(A_k x + b_k) \in \RR: A_k \in \RR^{\ell \times d},  \rk(A_k) \in \{0,\ell\},
 b_k \in \RR^\ell, c_k \in \RR \text{ for } k \!\in \!\{1,\dots,n\}\right\}
\end{align*}
of generalized translation networks where the matrices $A_k$ are limited to full-rank matrices and the zero matrix. 
Lastly, for fixed matrices $A_1,\dots,A_n \in \RR^{\ell \times d}$, we let 
\[
\nng(A_1, \dots, A_n) \defeq \left\{ \RR^d \ni x \mapsto \sum_{k=1}^n c_k \paul{\tau}(A_k x + b_k) \in \RR:  \ b_1, \dots, b_n \in \RR^\ell, \ c_1, \dots, c_n \in \RR\right\}
\]
denote the set of all generalized translation networks with \paull{fixed} weight matrices $A_1, \dots, A_n$.

Let us now turn to complex-valued neural networks (CVNNs):
For any fixed (activation) function $\phi: \CC \to \CC$, we write 
\[
\nn \defeq \left\{ \CC^d \ni z \mapsto \sum_{k=1}^n \gamma_k \phi(\alpha_k^T z + \beta_k) \in \CC : \ 
\alpha_1, \dots, \alpha_n \in \CC^d, \beta_1, \dots, \beta_n, \gamma_1, \dots, \gamma_n \in \CC\right\}
\]
for the set of shallow CVNNs with $d$ input neurons, $n$ hidden-layer neurons and activation function $\phi$. 
Moreover, for fixed weight vectors $\alpha_1, \dots, \alpha_n \in \CC^d$ we let 
\[
\nn(\alpha_1, \dots, \alpha_n) \defeq \left\{ \CC^d \ni z \mapsto \sum_{k=1}^n \gamma_k \phi(\alpha_k^T z + \beta_k) \in \CC : \ 
 \beta_1, \dots, \beta_n, \gamma_1, \dots, \gamma_n \in \CC\right\}.
\]

At first glance, one might be tempted to think that one obtains the set $\nn$ as a special case of generalized translation networks by replacing $d$ by $2d$, 
putting $\ell =2$ and using $\CC \cong \RR^2$. 
However, this is not the case, for the following reasons:
\begin{itemize}
\item The functions in $\mathcal{NN}^{\paul{\tau}}_{2d,2,n}$ map to $\RR$ whereas CVNNs map to $\CC$, which is equivalent to \emph{two} real output neurons. 
\item The activation function in $\mathcal{NN}^{\paul{\tau}}_{2d,2,n}$ is a function $\RR^2 \to \RR$, whereas the activation
 function of a CVNN is a function $\CC \to \CC$, i.e., $\RR^2 \to \RR^2$.
\item The matrices $A_k$ in the definition of $\mathcal{NN}^{\paul{\tau}}_{2d,2,n}$ are arbitrary elements of $\RR^{2 \times (2d)}$.
       The complex vectors $\alpha_k$ in the definition of $\nn$ may be regarded as elements of $\RR^{2 \times (2d)}$ but only elements of $\RR^{2 \times (2d)}$ with a specific 
       structure arise in this way. 
       That is, when viewed as elements of $\RR^{2 \times (2d)}$ they have the block form 
       \begin{equation}\label{eq:struc}
       \left(\begin{array}{cc|cc|c|cc} 
       a_1 & b _1 & a_2 & b_2 & \dots & a_d & b_d \\
       -b_1 & a_1 & -b_2 & a_2 & \dots & -b_d & a_d\end{array}
       \right) \in \RR^{2 \times (2d)}
       \end{equation}
       with real numbers $a_1, \dots, a_d, b_1, \dots, b_d \in \RR$ that represent the real and imaginary parts of the entries of the weight vector $\alpha_k$.
\end{itemize}

\subsection{Generalized translation networks}\label{subsec:nng}
In this section, we provide sharp bounds on the rate of approximation of Sobolev functions by generalized translation networks as defined above. 
A lower bound can be obtained immediately from \Cref{corr:low_final} by noting that every function from $\nng$ is in fact the sum of $n$ $\ell$-variate ridge functions. 
\begin{corollary}\label{corr:gtn_low}
Let $d,\ell, r \in \NN$ with $\ell < d$ and $p,q \in [1,\infty]$ be arbitrary. 
Then there exists a positive constant $c = c(d,\ell,p,q,r)>0$ with the following property:
For every $n \in \NN$ there exists a function $f \in \sob$ such that for any activation function $\paul{\tau} \in L^{1}_{\mathrm{loc}}(\RR^\ell)$ we have 
\[
\underset{\paul{\mathcal{T}} \in \nngg}{\inf} \ \mnorm{f - \paul{\mathcal{T}}}_{L^q(B^d)} \geq c \cdot n^{-r/(d-\ell)}.
\]
If $\paul{\tau} \in \mb(\RR^\ell)$, we furthermore have
\[
\underset{\paul{\mathcal{T}} \in \nng}{\inf} \ \mnorm{f - \paul{\mathcal{T}}}_{L^q(B^d)} \geq c \cdot n^{-r/(d-\ell)}.
\]
\end{corollary}
\begin{proof}
Let $\paul{\mathcal{T}} \in \nng$ be arbitrary, i.e., 
\[
\paul{\mathcal{T}}(x) = \sum_{k=1}^n c_k \paul{\tau}(A_k x + b_k).
\]
For $k \in \{1,\dots,n\}$, we define $\profilefunction_k (x) \defeq c_k \cdot \paul{\tau}(x + b_k)$.
In the case of $\paul{\tau} \in \mb(\RR^\ell)$ we clearly get $\paull{\profilefunction_k \in \mb(\RR^\ell)}$ for any $k \in \{1,\dots,n\}$,
which implies $\paul{\mathcal{T}} \in \rndl$ in that case. 
On the other hand, if $\paul{\tau} \in  L^{1}_{\mathrm{loc}}(\RR^\ell)$ and $A_k \in \RR^{\ell \times d}$ is full-rank, we get that $\profilefunction_k \in  L^{1}(B^d)$ by \Cref{prop:rank}.
If $A_k = 0$, the function $\profilefunction_k$ is constant and therefore trivially contained in $L^{1}(B^d)$.
The claim then follows from \Cref{corr:low_final}.
\end{proof}
Note the subtle difference that we have to restrict to full-rank matrices $A_k$ (or the zero matrix) in the case of an arbitrary $L^{1}_{\mathrm{loc}}$-activation function.
This restriction is not necessary in the case of a locally bounded activation function. 
Whether this restriction in the case $\paul{\tau} \in L^{1}_{\mathrm{loc}}(\RR^\ell)$ is necessary or can be removed remains a question for further investigation.
 
In the following theorem, we show that for every $\ell \in \NN$ there exists a smooth activation function $\paul{\tau} : \RR^\ell\to \RR$ such that the optimal rate of $n^{-r/(d-\ell)}$
when approximating Sobolev functions with networks from $\nng$ can indeed be achieved. 
This activation function is constructed in a ``piecewise'' manner, such that the shifts of this function, restricted to $B^\ell$, form a dense subset of $C(B^\ell)$. 
\begin{theorem} \label{thm:gtn_upper}
Let $\ell \in \NN$. 
Then there exists a smooth activation function $\paul{\tau} : \RR^\ell \to \RR$ with the following property:
For every $d,r \in \NN$ with $\ell < d$ and $1 \leq q \leq p \leq \infty$, there exists a positive constant $C = C(d,\ell,p,q,r)>0$ 
such that for every $n \in \NN$ there exist matrices $A_1, \dots, A_n \in \RR^{\ell \times d}$ with the property that 
for every $f \in \sob$ there exists a shallow generalized translation network $\paul{\mathcal{T}} \in \nng(A_1, \dots, A_n)$ with 
\[
\mnorm{f - \paul{\mathcal{T}}}_{L^q(B^d)} \leq C \cdot n^{-r/(d - \ell)}.
\]
\end{theorem}
\begin{proof}
Let $\left\{ u_m : \ B^\ell \to \RR\right\}_{m \in \NN}$ be a countable set of smooth functions which is dense in $C(B^\ell)$ with respect to $\mnorm{\cdot}_{L^\infty}$.
For instance, we can take the set of polynomials in $\ell$
variables with rational coefficients. 
\paull{We write $e_1$ for the first standard basis vector in $\RR^\ell$.}
Let then $\paul{\tau} : \RR^\ell \to \RR$ be a smooth function with 
\[
\paul{\tau}(x + 3m \paull{\cdot e_1}) = u_m(x) \quad \text{for every } x \in B^\ell, \ m \in \NN.
\]
We show that $\paul{\tau}$ has the desired property. 
To this end, let $d,n,r \in \NN$ with $d > \ell$, $1 \leq q \leq p \leq \infty$ and $f \in \sob$ be arbitrary. 
We use \Cref{thm:pqstatement} and obtain the existence of a positive constant $C_1 = C_1(d,\ell,p,q,r)>0$, polynomials $Q_1, \dots, Q_n \in \polyy(\RR^\ell)$ and matrices $B_1, \dots, B_n \in \RR^{\ell \times d}$
with 
\[
\mnorm{f(x) - \sum_{k=1}^n Q_k(B_k x)}_{L^q(B^d)} \leq C_1 \cdot n^{-r/(d-\ell)}.
\]
Note that the choice of the matrices $B_k$ does \emph{not} depend on the choice of $f$ according to \Cref{thm:pqstatement}.
Let $k \in \{1,\dots,n\}$.
If $B_k \neq 0$, we set\footnote{
For a matrix $M \in \RR^{\ell \times d}$, we let $\mnorm{M}_{\ell^2 \to \ell^2} \defeq \underset{x \in B^d}{\sup} \twonorm{Mx}$.}
$A_k \defeq B_k / \mnorm{B_k}_{\ell^2 \to \ell^2}$
and $P_k(x) \defeq Q_k(\mnorm{B_k}_{\ell^2 \to \ell^2} \cdot x)$.
If $B_k = 0$, set $A_k \defeq B_k$ and $P_k \defeq Q_k$.
By construction, we have 
\[
\sum_{k=1}^n Q_k(B_k x) = \sum_{k=1}^n P_k(A_k x).
\]
Note that we then have $A_k x  \in B^\ell$ for every $x \in B^d$.
Due to the density of the $u_m$, we can for every $k \in \{1,\dots,n\}$ pick a number $m_k \in \NN$ such that 
\paull{
\[
\mnorm{P_k - u_{m_k}}_{L^\infty(B^\ell)}\leq n^{-r/(d-\ell)} \cdot n^{-1} \cdot \left({\metalambda}^d(B^d)\right)^{-1/q}
\]
and hence
}
\[
\mnorm{P_k(A_k x) - u_{m_k}(A_k x)}_{L^q(B^d)} \leq n^{-r/(d-\ell)} \cdot n^{-1}. 
\]
Then, because of $\paul{\tau}(A_k x + 3m_k \paull{\cdot e_1}) = u_{m_k}(A_k x)$ for $x \in B^d$, we get that 
\begin{align*}
\mnorm{f(x) - \sum_{k=1}^n \!\paul{\tau}(A_k x + 3m_k\paull{\cdot e_1})}_{L^q(B^d)} \!\!&\leq \mnorm{f(x) - \sum_{k=1}^n \! P_k(A_k x)}_{L^q(B^d)}
\!\!\!\!+ \sum_{k=1}^n \! \mnorm{P_k(A_k x) - u_{ m_k}(A_kx)}_{L^q(B^d)} \\
&\leq (C_1 + 1) \cdot n^{-r/(d- \ell)}.
\end{align*}
Hence, the claim follows by letting $C \defeq C_1 + 1$.
\end{proof}
\subsection{Complex-valued neural networks}\label{subsec:cvnn}
The goal of this subsection is to prove the sharp approximation rate of $n^{-r/(2d-2)}$ for the approximation of (complex-valued) Sobolev functions on the unit ball
\paull{in $\CC^d$}
using shallow CVNNs with $d$ input neurons, $n$ hidden-layer neurons and locally \paull{integrable} activation function; see below for precise definitions. 
Establishing a \emph{lower} bound of $n^{-r/(2d-2)}$ is a direct consequence of \Cref{corr:low_final} and can be found in \Cref{corr:lowbound}.  
Showing that there exists an activation function for which the \emph{upper} bound of $n^{-r/(2d-2)}$ can be attained 
is more difficult and in particular requires the translation of several results from \cite[Section~5]{pinkus_ridge_2016}
to the complex-valued setting; see \Cref{thm:cvnnup}.

In order to establish our approximation bounds for shallow CVNNs, we first introduce some new notation.
For $d \in \NN$, we let 
\[
B^d(\CC) \defeq \left\{ z \in \CC^d : \ \twonorm{z} \leq 1\right\}, \quad \text{where } \twonorm{z} \defeq \sqrt{\sum_{j=1}^d \RE(z_j)^2 + \IM(z_j)^2}.
\]
Moreover, for $p \in [1,\infty]$, we let 
\[
\sobbb \defeq \left\{ f: B^d(\CC) \to \CC: \ \RE(f), \IM(f) \in \mathcal{B}(W^{r,p}_{2d})\right\},
\]
where we canonically identify $\CC^d \cong \RR^{2d}$ (and thus $B^d(\CC) \cong B^{2d}$).

We first derive a lower bound for the approximation of Sobolev functions using shallow CVNNs as introduced above. 
This is a \paull{relatively straightforward} consequence of \Cref{corr:low_final}
and may be viewed as a special case of \Cref{corr:gtn_low}.
\begin{corollary}\label{corr:lowbound}
Let $d,r \in \NN$ with $d \geq 2$ and $p,q \in [1,\infty]$ be arbitrary. 
Then there exists a positive constant $c = c(d,p,q,r)>0$ with the following property:
For any $n \in \NN$ there exists a function $f \in \sobbb$ such that for any (complex) activation function $\phi \in L^{1}_{\mathrm{loc}}(\CC; \CC)$ we have 
\[
\inf_{\Phi \in \nn} \ \mnorm{f - \Phi}_{L^q(B^d(\CC))} \geq c \cdot n^{-r/(2d-2)}.
\]
\end{corollary}
\begin{proof}
Let $\phi \in L^{1}_{\mathrm{loc}}(\CC; \CC)$ and $\Phi \in \nn$ be given. 
Then there exist $\alpha_1, \dots, \alpha_n \in \CC^d$ and $\beta_1, \dots, \beta_n, \gamma_1, \dots, \gamma_n \in \CC$ such that 
\[
\Phi(z) =\sum_{k=1}^n \gamma_k \phi(\alpha_k^T z + \beta_k) \quad \text{for every } z\in \CC^d.
\]
We note that
\begin{align*}
\RE(\Phi(z)) &= \sum_{k=1}^n \RE \left(\gamma_k \phi(\alpha_k^T z + \beta_k) \right) 
= \sum_{k=1}^n \left[\RE(\gamma_k)\cdot  (\RE \phi)(\alpha_k^T z + \beta_k) - \IM(\gamma_k) \cdot (\IM \phi)(\alpha_k^T z + \beta_k)\right].
\end{align*}
The goal is to show $\RE(\Phi) \in \mathcal{R}_{2n,2d,2}$ by identifying $\CC^d$ with $\RR^{2d}$.
According to the computation above, this is satisfied if 
\[
(\mathrm{Part} \ \phi) (\alpha_k^T \cdot \bullet + \beta_k) \in L^{1}(B^d(\CC)),
\]
for any $\mathrm{Part} \in \{\RE, \IM\}$ and $k \in \{1,\dots,n\}$. 
Since $\phi \in L^{1}_{\mathrm{loc}}(\CC)$, we get 
\[
(\mathrm{Part} \ \phi)(\bullet + \beta_k) \in  L^{1}_{\mathrm{loc}}(\RR^2),
\]
by identifying $\CC$ with $\RR^2$. 
Moreover, if $\alpha_k \neq 0$, it follows that the associated real-valued matrix $A \in \RR^{2 \times 2d}$ (see \eqref{eq:struc}) has full rank 2.
Hence, in this case we get 
\[
(\mathrm{Part} \ \phi) (\alpha_k^T \cdot \bullet + \beta_k) \in L^{1}(B^d(\CC))
\]
by \Cref{prop:rank}.
Conversely, if $\alpha_k = 0$, the map $(\mathrm{Part} \ \phi) (\alpha_k^T \cdot \bullet + \beta_k)$ is constant and therefore trivially contained in $L^{1}(B^d(\CC))$.
This proves $\RE(\Phi) \in \mathcal{R}_{2n,2d,2}$.

Hence, according to \Cref{corr:low_final}, we conclude the existence of a real-valued function $f \in \mathcal{B}(W^{r,p}_{2d})$ and a constant $\tilde{c} = \tilde{c}(r,d,p,q)> 0 $ such that 
\[
\underset{\phi \in L^{1}_{\mathrm{loc}}(\CC)}{\inf} \ \underset{\Phi \in \nn}{\inf} \ \mnorm{f - \RE(\Phi)}_{L^q(B^d(\CC))} \geq \tilde{c} \cdot (2n)^{-r/(2d-2)} 
= c \cdot n^{-r/(2d-2)},
\]
by letting $c \defeq 2^{-r/(2d-2)} \cdot \tilde{c}$.
The claim is then obtained by noting that 
\[
\mnorm{f - \RE(\Phi)}_{L^q(B^d(\CC))} =  \mnorm{\RE(f - \Phi)}_{L^q(B^d(\CC))}\leq \mnorm{f - \Phi}_{L^q(B^d(\CC))} \quad \text{for every } \Phi \in \nn. \qedhere
\]
\end{proof}

\newcommand{\psc}{\paull{\mathcal{P}}_s(\CC^d)}
\newcommand{\psth}{\paull{\mathcal{P}}^h_{s,t}(\CC^d)}
\newcommand{\pssh}{\paull{\mathcal{P}}^h_{s,s}(\CC^d)}
\newcommand{\ppsth}{\paull{\mathcal{P}}^h_{s',t'}(\CC^d)}
\newcommand{\wirt}{\partial_{\mathrm{wirt}}}
\newcommand{\wirtq}{\overline{\partial}_{\mathrm{wirt}}}

It remains to show that there exists a complex activation function $\phi:\CC \to \CC$ for which the rate of $n^{-r/(2d-2)}$ (which is proven to be optimal according
to \Cref{corr:lowbound}) can indeed be attained. 
The activation function for which we will show that it achieves the desired approximation rate is the same ``piecewise'' activation
 function that was already constructed in \cite[Lemma~F.4]{geuchen2024optimal},
where an approximation rate of $n^{-r/(2d-1)}$ has been proven (see \cite[Theorem~4.2]{geuchen2024optimal}).
This function is constructed following the same idea as in \Cref{thm:gtn_upper}. 
The main reason why the approach from \Cref{thm:gtn_upper} cannot be used \paull{directly} to obtain the desired result was already discussed at the beginning of the section:
while in the definition of the set \paull{$\mathcal{NN}^\tau_{2d,2,n}$} the matrices $A_k$ are arbitrary, they are restricted to a specific structure when considering CVNNs. 
Specifically, it is not straightforward to show that in the case $\ell = 2$ one can pick the matrices appearing in the proof of \Cref{prop:up_1} 
to have the structure considered in \eqref{eq:struc}.

Therefore, we translate several results from \cite[Section~5]{pinkus_ridge_2016} to the case of complex polynomials in $z$ and $\overline{z}$.
Here, we make use of the \emph{Wirtinger Calculus} which we briefly discuss here. 
For a function $f \in C^1(U;\CC)$ with an open set $U \subseteq \CC$ and where $C^1$ refers to differentiability with respect to real variables, we define for $w = x + \ii y \in U$ \paull{(with $x,y \in \RR$)}
the Wirtinger derivatives at $w$ as 
\begin{align*}
\wirt f(w) &\defeq \frac{\partial f}{\partial z}(w) \defeq \frac{1}{2} \left(\frac{\partial f}{\partial x}(w) - \ii \cdot \frac{\partial f}{\partial y}(w)\right) \\
\wirtq f(w) &\defeq \frac{\partial f}{\partial \overline{z}}(w) \defeq \frac{1}{2} \left(\frac{\partial f}{\partial x}(w) + \ii \cdot \frac{\partial f}{\partial y}(w)\right).
\end{align*}
The intuition behind the Wirtinger derivatives is to formally treat $z$ and $\overline{z}$ as \emph{independent} variables and to then take derivatives only with respect to $z$ or $\overline{z}$.
For multiindices $\kk, \elll \in \NN_0^d$ we write $\wirt^{\kk}\wirtq^{\elll}$ for iterated multivariate Wirtinger derivatives according to the multiindices $\kk$ and $\elll$.
This is well-defined when applied to functions of sufficient regularity, since Wirtinger derivatives commute \paull{because they are} linear combinations of partial derivatives.  

For $s,t \in \NN_0$ and $d \in \NN$ we define
\[
\psc \defeq \left\{ \CC^d \ni z \mapsto \underset{\abs{\kk}, \abs{\elll}\leq s}{\sum_{\kk, \elll \in \NN_0^d}} 
a_{\kk, \elll} z^{\kk} \overline{z}^{\elll}: \ a_{\kk, \elll} \in \CC \right\}
\]

and 

\[
\psth \defeq \left\{\CC^d \ni z \mapsto \underset{\abs{\kk} = s, \abs{\elll} = t}{\sum_{\kk, \elll \in \NN_0^d}} 
a_{\kk, \elll} z^{\kk} \overline{z}^{\elll}: \ a_{\kk, \elll} \in \CC \right\}.
\]
We remark that the condition $\abs{\kk}, \abs{\elll} \leq s$ appearing in the definition of $\psc$ is different from $\abs{\kk} + \abs{\elll} \leq s$, which would be the direct generalization 
of the definition in the real case to the complex case. 
Clearly, it holds that
\begin{equation}
\psc =  \bigoplus_{s', t' \leq s} \ppsth.
\end{equation}

For a complex polynomial $Q \in \psc$ with 
\[
Q(z) = \underset{\abs{\kk}, \abs{\elll}\leq s}{\sum_{\kk, \elll \in \NN_0^d}}
a_{\kk, \elll} z^{\kk} \overline{z}^{\elll},
\]
we define the associated differential operator as 
\[
Q(D) \defeq  \underset{\abs{\kk}, \abs{\elll}\leq s}{\sum_{\kk, \elll \in \NN_0^d}} 
a_{\kk, \elll} \wirt^{\kk} \wirtq^{\elll},
\]
where the notations $\wirt$ and $\wirtq$ refer to the (multivariate) Wirtinger derivatives mentioned above.
A computation shows that for multiindices $\kk, \elll, \kk', \elll' \in \NN_0^d$ with 
$\abs{\kk} = \abs{\kk'}$ and $\abs{\elll} = \abs{\elll'}$ we get
\begin{equation}\label{eq:ts1}
\wirt^{\kk}\wirtq^{\elll}(z^{\kk'}\overline{z}^{\elll'}) = \mathbbm{1}_{(\kk, \elll) = (\kk', \elll') } \cdot \kk ! \cdot \elll !.
\end{equation}
We refer to \Cref{lem:wirtiden1} \paull{in \Cref{sec:wirt_proofs}} for a rigorous proof of that identity. 
Let $\mathcal{L}(\psth; \CC)$ denote the space of $\CC$-linear maps from $\psth$ to $\CC$.
According to \eqref{eq:ts1}, the set 
\[
\left\{\wirt^{\kk}\wirtq^{\elll}: \ \kk,  \elll \in \NN_0^d \text{ with } \abs{\kk} = s \text{ and } \abs{\elll} = t\right\}
\]
forms a basis of $\mathcal{L}(\psth; \CC)$, so we may conclude that 
\begin{equation}\label{eq:linear_pol}
\mathcal{L}(\psth; \CC) = \{Q(D): \ Q \in \psth\}.
\end{equation}
Moreover, for a fixed vector $a \in \CC^d$ and multiindices $\kk, \elll \in \NN_0^d$ with $\abs{\kk}=s$ and $\abs{\elll} = t$
for $s,t \in \NN_0$, we get 
\begin{equation}\label{eq:ts2}
\wirt^{\kk} \wirtq^{\elll}\left((a^T z)^s(\overline{a^T z})^{t}\right) = s! \cdot t! \cdot a^{\kk} \overline{a}^{\elll}.
\end{equation}
Here, we refer to \Cref{lem:wirtiden2} for a proof of this fact. 
Hence, for $Q \in \psth$, we observe 
\begin{equation}\label{eq:vanish}
Q(D) \left((a^T z)^s(\overline{a^T z})^{t}\right) = s! \cdot t! \cdot Q(a) \quad \text{for all } a\in \CC^d.
\end{equation}
\paull{
Moreover, we note that 
\begin{align}
(a^T z)^s(\overline{a^T z})^{t} &= \left(\sum_{j=1}^n a_j z_j\right)\left(\sum_{j=1}^n \overline{a_j} \overline{z_j}\right) \nonumber \\
\label{eq:hom_pol}
&= \left(\sum_{j_1, \dots, j_s=1}^n a_{j_1} \cdots a_{j_s} \cdot z_{j_1} \cdots z_{j_s}\right)\left(\sum_{j_1, \dots, j_s=1}^n \overline{a_{j_1}} \cdots \overline{a_{j_s}} \cdot \overline{z_{j_1}} \cdots \overline{z_{j_s}}\right) \in \psth
\end{align}
for every $a \in \CC^d$.
}
This leads to the following proposition, which is a generalization of \cite[Proposition~5.1]{pinkus_ridge_2016} to the complex-valued setting.
\begin{proposition}\label{prop:firstcharac}
Let $\Omega \subseteq \CC^d$, $s,t \in \NN_0$ and $P \in \psth$.
Then we have 
\[
P \in \spann_\CC \left\{ z \mapsto (a^T z)^s (\overline{a^T z})^t : \ a \in \Omega\right\}
\]
if and only if $Q(D) P = 0$ for every $Q \in \psth$ that vanishes on $\Omega$.
\end{proposition}
\begin{proof}
Let 
\[
V \defeq \spann_\CC\left\{z \mapsto (a^T z)^s (\overline{a^T z})^t : \ a \in \Omega\right\} \overset{\eqref{eq:hom_pol}}{\subseteq} \psth.
\]
Then, from elementary linear algebra, for $P \in \psth$ we have $P \in V$ if and only if $L(P) = 0$ for every $L \in  \mathcal{L}(\psth; \CC)$
with $\fres{L}{V} \equiv 0$.
From \eqref{eq:linear_pol}, we infer that this is equivalent to $Q(D)P = 0$ for every $Q \in \psth$ with 
$\fres{Q(D)}{V} \equiv 0$.
But from \eqref{eq:vanish} we get that $\fres{Q(D)}{V} \equiv 0$ if and only if $\fres{Q}{\Omega} \equiv 0$.
This proves the claim. 
\end{proof}

This immediately gives us the following characterization, which is the generalization of \cite[Corollary~5.11]{pinkus_ridge_2016} to the complex-valued case. 
\begin{proposition}\label{prop:spancharac}
Let $\Omega \subseteq \CC^d$ and $s,t \in \NN_0$.
Then we have
\[
\psth = \spann_\CC \left\{ z \mapsto (a^T z)^s (\overline{a^T z})^t : \ a \in \Omega\right\}
\]
if and only if for every $Q \in \psth$ we have 
\[
\fres{Q}{\Omega} \equiv 0 \quad \Rightarrow \quad Q \equiv 0.
\]
\end{proposition}
\begin{proof}
We again let 
\[
V \defeq \spann_\CC\left\{ z  \mapsto (a^T z)^s (\overline{a^T z})^t : \ a \in \Omega\right\} \overset{\eqref{eq:hom_pol}}{\subseteq} \psth.
\]
According to \Cref{prop:firstcharac} we have $\psth = V$ if and only if 
\[
\fall P,Q \in \psth: \quad \fres{Q}{\Omega} \equiv 0 \quad \Rightarrow \quad Q(D)P = 0. 
\]
We reformulate the latter to 
\[
\fall Q \in \psth: \quad \fres{Q}{\Omega} \equiv 0 \quad \Rightarrow \quad \left(Q(D)P = 0 \quad \text{for all } P \in \psth\right). 
\]
But $Q(D)P = 0$ for all $P \in \psth$ holds if and only if $Q \equiv 0$, which follows for instance from \eqref{eq:vanish}.
\end{proof}
Note that the previous proposition in particular shows the following a priori not entirely obvious statement
\[
\psth = \spann_\CC \left\{ z \mapsto (a^T z)^s (\overline{a^T z})^t : \ a \in \CC^d\right\},
\]
since the only polynomial that vanishes on $\CC^d$ is the zero polynomial. 
Since the space $\psth$ is finite-dimensional, we in particular infer the existence of a set $\Omega \subseteq \CC^d$ with $\abs{\Omega} = \dim_\CC(\psth)$ and 
\begin{equation}\label{eq:spanfinite}
\psth = \spann_\CC \left\{z \mapsto  (a^T z)^s (\overline{a^T z})^t : \ a \in \Omega\right\}.
\end{equation}

The following proposition is crucial for the proof of the upper bound. 
For its real-valued analogon, we refer to \cite[Corollary~5.12]{pinkus_ridge_2016}.
\begin{proposition}\label{prop:imp}
Let $\Omega \subseteq \CC^d$ and $s,t \in \NN_0$.
If 
\[
\psth = \spann_\CC \left\{ z \mapsto (a^T z )^s (\overline{a^T z})^t : \ a \in \Omega\right\},
\]
then also 
\[
\ppsth = \spann_\CC \left\{ z \mapsto (a^T z )^{s'} (\overline{a^T z})^{t'} : \ a \in \Omega\right\}
\]
for all $s', t' \in \NN_0$ with $s' \leq s$ and $t' \leq t$.
\end{proposition}
\begin{proof}
\paull{
We  know that $\ppsth \overset{\eqref{eq:hom_pol}}{\supseteq} \spann_\CC \left\{ z \mapsto (a^T z )^{s'} (\overline{a^T z})^{t'} : \ a \in \Omega\right\}$. 
Suppose} that 
\[
\ppsth \supsetneq \spann_\CC \left\{ z \mapsto (a^T z )^{s'} (\overline{a^T z})^{t'} : \ a \in \Omega\right\}
\]
for some $s' \leq s$ and $t' \leq t$.
According to \Cref{prop:spancharac}, we can then pick $Q \in \ppsth$ with $Q \neq 0$
and $\fres{Q}{\Omega} \equiv 0$. Let $\tilde{Q} \in \mathcal{P}^h_{s- s', t- t'}(\CC^d) \setminus \{0\}$ be arbitrary. 
Then we have $Q \tilde{Q} \in \psth \setminus \{0\}$ and $\fres{Q \tilde{Q}}{\Omega} \equiv 0$, which contradicts \Cref{prop:spancharac}, 
\paull{since $\psth = \spann_\CC\left\{ z \mapsto (a^T z )^s (\overline{a^T z})^t : \ a \in \Omega\right\}$ by assumption of the proposition.}
\end{proof}
We can now show that each polynomial from $\psc$ can be written as the sum of (complex) ridge polynomials, where the number of summands 
depends on the dimension of the space $\pssh$.
This statement is the translation of \Cref{prop:up_1} to the complex-valued setting. 
\begin{theorem}\label{thm:ridgecomplexpol}
Let $d \in \NN$ and $s \in \NN_0$ and pick $n \in \NN$ with 
\[
\dim_\CC (\pssh) \leq n.
\]
Then there exist $a_1 ,\dots, a_n \in \CC^d$ with $\twonorm{a_j}  \leq 1$ for $j \in \{1,\dots,n\}$
with the following property:
For every $P \in \psc$ there exist $P_1 ,\dots, P_n \in \polyy_s(\CC)$ with
\[
P(z) = \sum_{j=1}^n  P_j(a_j^T z).
\]
\end{theorem}
\begin{proof}
We pick $a_1,\dots, a_n \in \CC^d$ with $\twonorm{a_j} \leq 1$ for $j \in \{1,\dots,n\}$ such that 
\[
\pssh = \spann_\CC \left\{ z \mapsto (a_j^T z)^s (\overline{a_j^T z})^s : \ j = 1,\dots,n\right\}.
\]
This is possible according to \eqref{eq:spanfinite}.
Note that we can scale the $a_j$ as we want since scaling does not change the span above. 
Let $P \in \psc$ be arbitrary. 
Then we can write 
\[
P(z) = \sum_{s', t' \leq s} Q_{s', t'}(z)
\]
with $Q_{s',t'} \in \ppsth$ for every $s',t' \leq s$.
According to \Cref{prop:imp}, we can write
\[
Q_{s', t'}(z) = \sum_{j=1}^n a_{s',t',j}\cdot  (a_j^T z)^{s'} (\overline{a_j^T z})^{t'}
\]
with suitable coefficients $a_{s',t',j} \in \CC$.
This gives us
\[
P(z) = \sum_{s', t' \leq s} \sum_{j=1}^n a_{s',t',j}\cdot  (a_j^T z)^{s'} (\overline{a_j^T z})^{t'}
=  \sum_{j=1}^n \sum_{s', t' \leq s} a_{s',t',j}\cdot  (a_j^T z)^{s'} (\overline{a_j^T z})^{t'} 
= \sum_{j=1}^n P_j(a_j^T z)
\]
with 
\[
P_j (z) \defeq \sum_{s', t' \leq s} a_{s',t',j}\cdot  z^{s'}\overline{z}^{t'}.
\]
Since clearly $P_j \in \polyy_s(\CC)$ for every $j \in \{1,\dots,n\}$, the claim is shown. 
\end{proof}

The activation function that yields the optimal approximation rate is obtained in the following lemma.
We refer to \cite{geuchen2024optimal} for the proof.
\begin{lemma}[{cf. \cite[Lemma~F.4]{geuchen2024optimal}}] \label{lem:acti}
Let $\left\{ u_\ell\right\}_{\ell = 1}^\infty$ be an enumeration of the set of complex polynomials in $z$ and $\overline{z}$ with coefficients in $\QQ + \ii \QQ$.
    Then there exists a smooth function $\phi: \  \CC \to \CC$ (where by ``smoothness'' we refer to smoothness with respect to real variables) with the property that 
    for every $\ell \in \NN$ and $z \in [-1,1] + \ii \cdot [-1,1] \subseteq \CC$ one has
        \begin{equation*}
            \phi(z+3\ell) = u_\ell(z).
        \end{equation*}
\end{lemma}

Note that since $ B^1(\CC) \subseteq [-1,1] + \ii \cdot [-1,1]$, the function $\phi$ in particular satisfies 
\begin{equation}\label{eq:actiprop}
\phi(z+3\ell) = u_\ell(z) \quad \text{for every } z \in B^1(\CC).
\end{equation}
We can now state and prove the main result of this section.

\begin{theorem}\label{thm:cvnnup}
Let $\phi:\CC \to \CC$ be the activation function from \Cref{lem:acti}. 
Moreover, let $d,r\in \NN$ with $d \geq 2$ and $1 \leq q \leq p \leq \infty$. 
Then there exists a constant $C = C(d,p,q,r)> 0$ with the following property: 
For any $n \in \NN$ there exist complex vectors $\alpha_1, \dots, \alpha_n \in \CC^d$ such that for every $f \in \sobbb$
 there exists a shallow CVNN $\Phi \in \nn(\alpha_1, \dots, \alpha_n)$ with 
\[
\mnorm{f - \Phi}_{L^q(B^d)} \leq C \cdot n^{-r/(2d-2)}.
\]
\end{theorem}
\begin{proof}
Let $s \in \NN$. 
Then it is easy to see that 
\[
\dim_\CC (\pssh) = \# \left\{ z^{\kk}z^{\elll}:\ \abs{\kk} = \abs{\elll} = s\right\} = \left(\# \left\{ \kk \in \NN_0^d: \ \abs{\kk} = s\right\}\right)^2 \leq C_1 \cdot \left(s^{d-1}\right)^2
= C_1 \cdot s^{2d-2},
\]
where $C_1 = C_1(d)>0$; again, see for instance \cite[Lemma~F.1]{geuchen2024optimal}
We may assume $n \geq C_1$ at the cost of possibly enlarging the constant $C$ in the end, similar to the end of the proof of \Cref{thm:main}.
We then pick $s \in \NN$ as the largest number satisfying $C_1 \cdot s^{2d-2} \leq n$.
Note that this implies 
\[
C_1 \cdot (2s)^{2d-2} > n \quad \Leftrightarrow \quad s > c_2 \cdot n^{1/(2d-2)}
\]
for a constant $c_2 = c_2(d)>0$.

Let $f \in \sobbb$ be arbitrary. 
For a multiindex $\kk \in \NN_0^{2d}$ and $z \in \CC^d$, we write 
\[
(\RE(z), \IM(z))^{\kk} \defeq \prod_{j=1}^d \RE(z_j)^{\kk_j} \cdot \IM(z_j)^{\kk_{d+j}}.
\]
By applying \Cref{prop:jack} to $\RE(f) \paull{\in \mathcal{B}(W_{2d}^{r,p})}$ and $\IM(f)\paull{\in \mathcal{B}(W_{2d}^{r,p})}$, we obtain the existence of a polynomial 
\[
P(z) = \underset{\abs{\kk} \leq s}{\sum_{\kk \in \NN_0^{2d}}} a_{\kk} \cdot (\RE(z), \IM(z))^{\kk}
\]
with 
\[
\mnorm{f - P}_{L^q(B^d(\CC))} \leq C_3 \cdot s^{-r} \leq C_4 \cdot n^{-r/(2d-2)},
\]
where $C_3 = C_3(d,p,q,r)>0$ and $C_4 = C_4(d,p,q,r)>0$ are constants and $a_{\kk} \in \CC$ for $\kk \in \NN_0^{2d}$ with $\abs{\kk} \leq s$.
For a fixed $\kk \in \NN_0^{2d}$, we compute 
\begin{align*}
(\RE(z), \IM(z))^{\kk} = \prod_{j=1}^d \RE(z_j)^{\kk_j} \cdot \IM(z_j)^{\kk_{d+j}} =  
\prod_{j=1}^d \left[ \frac{1}{2^{\kk_j} \cdot (2 \ii)^{\kk_{d+j}}} \cdot (z_j + \overline{z_j})^{\kk_j} \cdot (z_j - \overline{z_j})^{\kk_{d+j}} \right].
\end{align*}
For each $j \in \{1,\dots,d\}$ we let $Q_j(z_j) \defeq (z_j + \overline{z_j})^{\kk_j}$ and $\tilde{Q_j}(z_j) \defeq (z_j - \overline{z_j})^{\kk_{d+j}}$.
Note that $Q_j \in \polyy_{\kk_j}(\CC)$ and $\tilde{Q_j} \in \polyy_{\kk_{d+j}}(\CC)$, which yields $Q_j \tilde{Q_j} \in \polyy_{\kk_j + \kk_{d+j}}(\CC)$.
Overall, this yields 
\[
\prod_{j=1}^d \RE(z_j)^{\kk_j} \cdot \IM(z_j)^{\kk_{d+j}} \in \polyy_{\abs{\kk}}(\CC^d) \subseteq \polyy_s(\CC^d),
\]
which then yields $P \in \polyy_s(\CC^d)$.
Since $n \geq \dim_\CC(\pssh)$, we can apply \Cref{thm:ridgecomplexpol} and write 
\[
P(z) = \sum_{j=1}^n P_j(\alpha_j^T z),
\]
where $P_j \in \polyy_s(\CC)$ and $\alpha_j \in \CC^d$ with $\twonorm{\alpha_j} \leq 1$ for $j \in \{1,\dots,n\}$.
Note that the choice of the $\alpha_j$ is \emph{independent} of the choice of $f$ (and $P$) according to \Cref{thm:ridgecomplexpol}.
Recall from \Cref{lem:acti} that $\left\{ u_\ell\right\}_{\ell = 1}^\infty$ is an enumeration of the set of complex polynomials in $z$ and $\overline{z}$
with coefficients in $\QQ + \ii \QQ$. 
Since this set is dense in $C(B^1(\CC))$ with respect to $\mnorm{\cdot}_{L^\infty}$ and $\abs{\alpha_j^T z }\leq \twonorm{\alpha_j} \cdot \twonorm{z} \leq 1$ for every $z \in B^d(\CC)$, we can pick $\ell_1, \dots, \ell_n \in \NN$ with 
\paull{
\[
\mnorm{P_j - u_{\ell_j}}_{L^\infty(B^1(\CC))} \leq n^{-1 - \frac{r}{2d-2}} \cdot \left({\metalambda}^{2d}(B^d(\CC))\right)^{-1/q}
\]
and hence
}
\[
\mnorm{P_j(\alpha_j^T z) - u_{\ell_j}(\alpha_j^T z)}_{L^q(B^d(\CC))} \leq n^{-1-\frac{r}{2d-2}} \quad \text{for every } j \in \{1,\dots,n\}.
\]
We then get 
\begin{align*}
&\norel \mnorm{f(z) - \sum_{j=1}^n \phi(\alpha_j^T z + 3\ell_j)}_{L^q(B^d(\CC))} \\
&= \mnorm{f(z) - \sum_{j=1}^n u_{\ell_j}(\alpha_j^T z)}_{L^q(B^d(\CC))} \\
&\leq \mnorm{f(z) - \sum_{j=1}^n P_j(\alpha_j^T z)}_{L^q(B^d(\CC))} + \sum_{j=1}^n \mnorm{P_j(\alpha_j^T z) - u_{\ell_j}(\alpha_j^T z)}_{L^q(B^d(\CC))} \\
&\leq \mnorm{f - P}_{L^q(B^d(\CC))}  + n^{-r/(2d-2)} \leq (C_4 + 1) \cdot n^{-r/(2d-2)}.
\end{align*}
Since $ z \mapsto \sum_{j=1}^n \phi(\alpha_j^T z + 3\ell_j) \in \nn(\alpha_1, \dots, \alpha_n)$, the claim follows by letting $C \defeq C_4 + 1$.
\end{proof}
We remark that one can \emph{not} expect 
the rate of $n^{-r/(2d-2)}$ for \emph{general} (smooth, non-polyharmonic) activation functions $\phi$, as follows from \cite[Theorem~4.3]{geuchen2024optimal}.
\appendix
\section{Discussion of an issue in \texorpdfstring{\cite{maiorov2010best}}{the paper ``Best approximation by ridge functions in Lᵖ-spaces``}}
\label{sec:disc}

In this appendix, we discuss an issue in the paper \cite{maiorov2010best} in which a lower bound of $n^{-r/(d-1)}$ for the error of approximating 
Sobolev functions by sums of $n$ univariate ridge functions
with respect to the $L^1$-norm is shown. 
One of the central ingredients of the proof in that work is \cite[Lemma~5]{maiorov2010best}, which, in the notation of the present work, can be stated as follows: 
\medskip

Let $d, n,s \in \NN$ with $d>1$ and let $P \in \polyy_s(B^d)$ be arbitrary. 
Let $\pi_s : L^2(B^d) \to \polyy_s(B^d)$ denote the orthogonal projection onto $\polyy_s(B^d)$.
Then it holds that
\[
\underset{R \in \mathcal{R}^\ast_{n,d,1}}{\inf} \mnorm{P-R}_{L^1(B^d)} \geq \underset{P' \in \pi_s(\mathcal{R}^\ast_{n,d,1})}{\inf} \mnorm{P-P'}_{L^1(B^d)}.
\] 

\medskip
This claim is thus essentially identical to a univariate version of \Cref{lem:prob_wrong} (i.e., for $\ell = 1$), 
with the difference that the quasi-projection $\prs$ is replaced by the orthogonal projection $\pi_s$
and that there is no absolute constant appearing in the inequality.

The proof presented in \cite{maiorov2010best} relies on the fact that the set $ \mathcal{R}^\ast_{n,d,1}$ can be written as a union of subspaces $U_i \subseteq L^2(B^d)$.
Then, by showing that 
\begin{equation}\label{eq:wwrong}
\underset{R \in U_i}{\inf} \mnorm{P-R}_{L^1(B^d)} \geq \underset{P' \in \pi_s(U_i)}{\inf} \mnorm{P-P'}_{L^1(B^d)}
\end{equation}
holds for any $i$, one easily gets the bound by forming the infimum over $i$ on both sides. 
No special property of the subspaces $U_i$ is used in \cite{maiorov2010best}.
Yet, for arbitrary subspaces of $L^2(B^d)$ \paul{the inequality} \eqref{eq:wwrong} in fact \emph{fails to hold},
even if an absolute constant is allowed. 
We show this in the following proposition.
\begin{proposition}\label{thm:counterexample}
Let $d \in \NN$.
Then for any constant $\kappa>0$ there exist one-dimensional subspaces $U,V \subseteq L^2(B^d)$ and a function $P \in U$ with the property that 
\[
\underset{ g \in V}{\inf}\ \mnorm{P - g}_{L^1(B^d)} < \kappa \cdot \underset{\tilde{g} \in \pi_U(V)}{\inf}  \ \mnorm{P - \tilde{g}}_{L^1(B^d)}.
\]
Here, $\pi_U : \ L^2(B^d) \to U$ denotes the orthogonal projection onto $U$.
\end{proposition}
\begin{proof}
Let $\kappa > 0$ be an arbitrary constant. 
For $n \in \NN$, let $\psi_n: \RR \to \RR$ be continuous with $0 \leq \psi_n \leq 1$, with $\supp(\psi_n) \subseteq \left[\frac{1}{n},\infty\right)$, and with 
$\psi_n \equiv 1$ on $\left[\frac{2}{n},1\right]$.
Then, \paul{writing $x = (x_1, \dots, x_d)$ for $x \in \RR^d$}, define the functions 
\[
P_n : \quad \RR^d \to \RR, \quad P_n(x) = x_1^{-1/3} \cdot \psi_n(x_1)
\]
\paul{with the understanding that $P_n(x) = 0$ if $x_1 = 0$.}
We get 
\begin{align*}
\mnorm{P_n}^2_{L^2(B^d)} \leq \int_{[-1,1]^d} \abs{P_n(x)}^2 \ \dd x \overset{\text{Tonelli}}{\leq}
\int_{[\frac{1}{n}, 1]} x_1^{-2/3} \ \dd x_1 \cdot 2^{d-1}
= (3 - 3 \cdot n^{-1/3}) \cdot 2^{d-1} \leq 3 \cdot 2^{d-1}.
\end{align*}
Moreover, taking $n \geq 4 \sqrt{d}$, we get 
\begin{align*}
\mnorm{P_n}_{L^1(B^d)} &\geq \int_{\left[ - \frac{1}{\sqrt{d}}, \frac{1}{\sqrt{d}}\right]^d} x_1^{-1/3}\cdot \mathbbm{1}_{[\frac{2}{n} , 1]} (x_1) \ \dd(x_1, \dots, x_d) \\
\overset{\text{Tonelli}}&{=}  \int_{[\frac{2}{n}, \frac{1}{\sqrt{d}}]} x_1^{-1/3} \ \dd x_1 \cdot \int_{[- \frac{1}{\sqrt{d}}, \frac{1}{\sqrt{d}}]^{d-1}} \ \dd(x_2, \dots, x_{d}) \\
&\geq \left(\frac{2}{\sqrt{d}}\right)^{d-1} \cdot \int_{[\frac{1}{2\sqrt{d}}, \frac{1}{\sqrt{d}}]} x_1^{-1/3} \ \dd x_1 =: \theta = \theta(d)> 0.
\end{align*}
Moreover, since $P_n$ is continuous, 
\[
\mnorm{P_n}_{L^\infty(B^d)} \geq P_n(2/n,0,\dots,0)= \left(\frac{n}{2}\right)^{1/3} \cdot \psi_n \left(\frac{2}{n}\right) = \left(\frac{n}{2}\right)^{1/3}.
\]
Therefore, 
\[
\frac{2 \cdot \mnorm{P_n}^2_{L^2(B^d)}}{\mnorm{P_n}_{L^\infty(B^d)} \cdot \mnorm{P_n}_{L^1(B^d)}} 
\leq \frac{6 \cdot 2^{d-1}}{(n/2)^{1/3} \cdot \theta} \to 0 \quad (n \to \infty).
\]
We can therefore pick $N \in \NN$ with 
\[
\frac{2 \cdot \mnorm{P_N}^2_{L^2(B^d)}}{\mnorm{P_N}_{L^\infty(B^d)}} < \kappa \cdot \mnorm{P_N}_{L^1(B^d)}.
\]
Thus, $P \defeq P_N$ is continuous with $P \neq 0$ and 
\begin{equation}\label{eq:kappa}
\frac{2 \cdot \mnorm{P}^2_{L^2(B^d)}}{\mnorm{P}_{L^\infty(B^d)}} < \kappa \cdot \mnorm{P}_{L^1(B^d)}.
\end{equation}
 Let $x^\natural \in B^d$ with $\abs{P(x^\natural)} = \mnorm{P}_{L^\infty}$, which exists since continuous functions on compact sets attain their maximum.
 Without loss of generality, we may assume $P(x^\natural)> 0$, otherwise replace $P$ by $-P$.
Then, by continuity and since $B^d$ is the closure of $\left(B^d\right)^\circ$, 
it is easy to see that there exists a point $x^\ast \in \left(B^d\right)^\circ$ and $\delta > 0$ 
with the property that $Q \defeq x^\ast + (-\delta, \delta)^d \subseteq \left(B^d\right)^\circ$ and 
\begin{equation}\label{eq:x0}
P(x) \geq \frac{P(x^\natural)}{2}> 0 \quad \text{for all } x \in Q.
\end{equation}
We define 
\[
f \defeq P - c \cdot \mathbbm{1}_Q \quad \text{with} \quad c \defeq \frac{\mnorm{P}^2_{L^2(B^d)}}{\int_Q P(x) \ \dd x},
\]
noting that $\int_Q P(x) \ \dd x>  0$ by \eqref{eq:x0}.
We then compute that
\[
\langle P,f \rangle = \mnorm{P}_{L^2}^2 - c \cdot \int_Q P(x) \ \dd x = 0,
\]
meaning $P \perp f$.
Furthermore, since
\[
\frac{\mnorm{P}_{L^\infty(B^d)}}{2} \cdot {\metalambda}^d(Q) 
= \frac{P(x^\natural)}{2} \cdot {\metalambda}^d(Q) \overset{\eqref{eq:x0}}{\leq} \int_Q P(x) \ \dd x,
\]
we get 
\[
 c \leq \frac{2\cdot \mnorm{P}_{L^2(B^d)}^2}{\mnorm{P}_{L^\infty(B^d)}} \cdot \frac{1}{{\metalambda}^d(Q)}.
\]
This yields that
\[
\mnorm{P-f}_{L^1(B^d)} = c \cdot \mnorm{\mathbbm{1}_Q}_{L^1(B^d)} = c \cdot {\metalambda}^d(Q) \leq \frac{2 \cdot \mnorm{P}_{L^2(B^d)}^2}{\mnorm{P}_{L^\infty(B^d)}}. 
\]
For $V \defeq \spann f$ and $U \defeq \spann P$, we have $P \in U$ and get
\[
\underset{ g \in V}{\inf}\ \mnorm{P - g}_{L^1(B^d)} \leq \mnorm{P-f}_{L^1(B^d)} \leq \frac{2 \cdot \mnorm{P}_{L^2(B^d)}^2}{\mnorm{P}_{L^\infty(B^d)}},
\]
but $\pi_U(V) = \{0\}$ (since $f \perp P$) and hence 
\[
\underset{\tilde{g} \in \pi_U(V)}{\inf}  \ \mnorm{P - \tilde{g}}_{L^1(B^d)} = \mnorm{P}_{L^1(B^d)}.
\]
The claim then follows from \eqref{eq:kappa}.
\end{proof}
Specifically, the gap in the proof in \cite{maiorov2010best} lies in its Equation (14), where it is claimed that for a fixed polynomial $P \in \polyy_s(B^d)$ we have 
\begin{equation}\label{eq:wrong}
\underset{v \in \left(\pi_s( U_i)\right)^\perp \cap \polyy_s(B^d), \ \mnorm{v}_{L^\infty(B^d)} \leq 1}{\sup} \langle P,v \rangle = \underset{h \in \pi_s (U_i)}{\inf} \mnorm{P-h}_{L^1(B^d)}.
\end{equation}
Here, by definition we have 
\[
\left(\pi_s( U_i)\right)^\perp \defeq  \left\{ v \in L^\infty(B^d): \ \langle v, w \rangle = 0 \text{ for all } w \in \pi_s( U_i)\right\}.
\]
For an arbitrary normed space $(\mathcal{X}, \mnorm{\cdot})$ and a subspace $M \subseteq \mathcal{X}$, it is well-known that the distance between an arbitrary element $x \in \mathcal{X}$ 
and \paul{the space} $M$ can be expressed using the \emph{dual} space of $\mathcal{X}$. 
More precisely, let $\mathcal{X}^\ast$ denote the normed dual of $\mathcal{X}$ with dual norm $\mnorm{\cdot}_\ast$. 
One can then define the \emph{annihilator} of $M$ as 
\[
M^\perp \defeq \left\{ \varphi \in \mathcal{X}^\ast : \ \fres{\varphi}{M} \equiv 0\right\}.
\]
For an arbitrary element $x \in \mathcal{X}$, we then have 
\[
\underset{m \in M}{\inf} \mnorm{x - m } = \underset{\mnorm{\varphi}_\ast \leq 1}{\underset{\varphi \in M^\perp}{\sup}} \varphi(x),
\]
see for instance \cite[p.119,Thm.1]{luenberger1997optimization}.
Applying this fact to the case $\mathcal{X} = (\polyy_s(B^d), \mnorm{\cdot}_{L^1(B^d)})$ and $M = \pi_s(U_i)$ implies 
\[
\underset{h \in \pi_s (U_i)}{\inf} \mnorm{P-h}_{L^1(B^d)} = \underset{\varphi}{\sup} \ \varphi(P),
\]
where the supremum is taken over all continuous linear functionals $\varphi: \polyy_s(B^d) \to \RR$, for which 
\begin{equation}\label{eq:sup}
\underset{\mnorm{P'}_{L^1(B^d)}\leq 1}{\underset{P' \in \polyy_s(B^d)}{\sup}} \varphi(P') \leq 1 \quad \text{and} \quad \varphi(w) = 0 \quad \text{for every } w \in \pi_s(U_i). 
\end{equation}
Note that in order for \eqref{eq:wrong} (at least with ``$\geq$'' instead of ``$=$'') to hold, it would be sufficient that for each continuous linear functional $\varphi: \polyy_s(B^d) \to \RR$ 
satisfying \eqref{eq:sup}, there exists a function $v \in \polyy_s(B^d)$ with $\mnorm{v}_{L^\infty(B^d)} \leq 1$, $\langle v,w \rangle = 0$ for all $w \in \pi_s(U_i)$ and
\[
\varphi(P) = \langle v, P \rangle.
\]
If one is willing to drop the condition $v \in \polyy_s(B^d)$, this is in fact true: 
According to the Hahn-Banach extension theorem, for each such functional $\varphi: \polyy_s(B^d) \to \RR$ we can pick a (with respect to $\mnorm{\cdot}_{L^1(B^d)}$)
continuous linear extension $\varphi' : L^1(B^d) \to \RR$ with $\fres{\varphi'}{\polyy_s(B^d)} = \varphi$ and 
\[
\underset{\mnorm{f}_{L^1(B^d)} \leq 1}{\underset{f \in L^1(B^d)}{\sup}}\varphi'(f) = \underset{\mnorm{P'}_{L^1(B^d)}\leq 1}{\underset{P' \in \polyy_s(B^d)}{\sup}} \varphi(P') \leq 1.
\] 
Moreover, according to \cite[Theorem~6.15]{folland1984real} there exists $v \in L^\infty(B^d)$ with $\mnorm{v}_{L^\infty(B^d)} \leq 1$ and 
\[
\langle v, f \rangle = \varphi'(f) \quad \text{for all }f\in L^1(B^d).
\]
In particular, since $\fres{\varphi'}{\polyy_s(B^d)} = \varphi$, we have $\langle v,w \rangle = \varphi(w) = 0$ for every $w \in \pi_s(U_i)$.
However, it is \emph{not} clear (to the authors of the present paper), whether one can pick $v$ to be contained in $\polyy_s(B^d)$, which would imply \eqref{eq:wrong} (with ``$\geq$'').

Alternatively, for each continuous linear functional $\varphi: \polyy_s(B^d) \to \RR$ satisfying \eqref{eq:sup}, since $\polyy_s(B^d)$ is finite-dimensional
it is well-known that there exists $v \in \polyy_s(B^d)$ with
\[
\varphi(P') = \langle v,P' \rangle \quad \text{for all } P' \in \polyy_s(B^d).
\] 
By \eqref{eq:sup}, we have $\langle v,w \rangle = 0$ for all $w \in \pi_s(U_i)$ and 
\[
\underset{\mnorm{Q}_{L^1(B^d)} \leq 1}{\underset{Q \in \polyy_s(B^d)}{\sup}} \langle v, Q \rangle 
= \underset{\mnorm{Q}_{L^1(B^d)}\leq 1}{\underset{Q \in \polyy_s(B^d)}{\sup}} \varphi(Q)\leq 1.
\]
However, it is \emph{not} clear at all (to the authors of the present paper) if this implies that 
\[
\mnorm{v}_{L^\infty(B^d)} = \underset{\mnorm{f}_{L^1(B^d)} \leq 1}{\underset{f \in L^1(B^d)}{\sup}} \langle v, f\rangle \leq 1.
\]

We emphasize that the preceding discussion including \Cref{thm:counterexample} does not necessarily imply that the statement \cite[Lemma~5]{maiorov2010best} is false. 
It might well be the case that the stated inequality is true, at least up to a fixed multiplicative constant. 
However, since the statement does not hold for arbitrary subspaces of $L^2(B^d)$ (see \Cref{thm:counterexample}), this then necessarily relies on specific 
properties of ridge functions and polynomials that the authors of the present work are not aware of (and that are not mentioned in \cite{maiorov2010best}).

\section{Postponed proofs} \label{app:postponed}

\subsection{Proof of Lemma \ref{lem:deltarewrite}}
\label{app:postponed1}

\begin{proof}[Proof of \Cref{lem:deltarewrite}]
\paul{
We will use throughout the proof without comment that $\binom{k + \sigma}{\sigma} = \binom{k + \sigma}{k}$, so that we need to show 
\[
\prs f = \sum_{k = 0}^\infty (\Delta^{\sigma + 1}\eta^\ast)(k) \cdot \binom{k + \sigma}{k} \cdot S_k^\sigma(f).
\]
}
The proof is by induction over $\sigma \in \NN_0$.
In the case $\sigma = 0$, we get 
\begin{align*}
\sum_{k=0}^\infty \left(\left(\Delta \eta^\ast \right)(k) \cdot \binom{k}{k} \cdot S_k^0(f) \right)&= 
\sum_{k=0}^\infty \left[\left(\Delta \eta^\ast\right)(k) \cdot \frac{1}{\binom{k}{0}} \cdot \sum_{j= 0}^k \underbrace{\binom{k-j}{0}}_{=1} \cdot \mathrm{proj}_j(f)\right] \\
&= \sum_{k=0}^\infty \left[\big(\eta(k/s) - \eta((k+1)/s)\big) \cdot \sum_{j= 0}^k  \mathrm{proj}_j(f) \right] \\
&= \sum_{k= 0}^\infty \left[ \eta(k/s) \cdot \sum_{j=0}^k \mathrm{proj}_j(f)\right] - \sum_{k= 0}^\infty \left[ \eta((k+1)/s) \cdot \sum_{j=0}^k \mathrm{proj}_j(f)\right] \\
&= \sum_{k= 0}^\infty \left[ \eta(k/s) \cdot \sum_{j=0}^k \mathrm{proj}_j(f)\right] - \sum_{k= 1}^\infty \left[ \eta(k/s) \cdot \sum_{j=0}^{k-1} \mathrm{proj}_j(f)\right] \\
&= \sum_{k= 0}^\infty \eta(k/s) \cdot \mathrm{proj}_k(f) \\
&= \prs f.
\end{align*}
We now assume the claim to be true for an arbitrary but fixed $\sigma \in \NN_0$.
We then get 
\begin{align*}
&\norel\sum_{k= 0}^\infty \left(\left(\Delta^{\sigma + 2} \eta^\ast\right)(k) \cdot \binom{k + \sigma + 1}{k} \cdot S_k^{\sigma + 1}(f)\right) \\
&= \sum_{k= 0}^\infty \left[\left[\left(\Delta^{\sigma + 1}\eta^\ast\right)(k) - \left(\Delta^{\sigma + 1}\eta^\ast\right)(k+1)\right]  \cdot \binom{k + \sigma + 1}{k} \cdot S_k^{\sigma + 1}(f) \right] \\
&= \sum_{k= 0}^\infty \left[\left(\Delta^{\sigma + 1}\eta^\ast\right)(k) \cdot \binom{k + \sigma + 1}{k} \cdot S_k^{\sigma + 1}(f) \right] 
- \sum_{k= 0}^\infty \left[\left(\Delta^{\sigma + 1}\eta^\ast\right)(k+1) \cdot \binom{k + \sigma + 1}{k} \cdot S_k^{\sigma + 1}(f) \right] \\
&=  \sum_{k= 0}^\infty \left[\left(\Delta^{\sigma + 1}\eta^\ast\right)(k) \cdot \binom{k + \sigma + 1}{k} \cdot S_k^{\sigma + 1}(f) \right] 
- \sum_{k= 1}^\infty \left[\left(\Delta^{\sigma + 1}\eta^\ast\right)(k) \cdot \binom{k + \sigma}{k-1} \cdot S_{k-1}^{\sigma + 1}(f) \right]  \\
&= \left(\Delta^{\sigma + 1}\eta^\ast\right)(0) \cdot S_0^{\sigma + 1}(f) 
+ \sum_{k=1}^\infty \left( \left(\Delta^{\sigma + 1}\eta^\ast\right)(k) \left[ \binom{k + \sigma +1}{k} \cdot S_k^{\sigma +1}(f) - \binom{k + \sigma}{k-1} \cdot S_{k-1}^{\sigma +1}(f)\right]\right) \\
&= \left(\Delta^{\sigma + 1}\eta^\ast\right)(0)\cdot \binom{0+\sigma}{\sigma} \cdot S_0^{\sigma}(f) \\
&\hspace{0.5cm}+ \sum_{k=1}^\infty \left( \left(\Delta^{\sigma + 1}\eta^\ast\right)(k) \left[ \binom{k + \sigma +1}{k} \cdot S_k^{\sigma +1}(f) - \binom{k + \sigma}{k-1} \cdot S_{k-1}^{\sigma +1}(f)\right]\right),
\end{align*}
where the last step used that $S_0^{\gamma}(f) = \mathrm{proj}_0(f)$ for every $\gamma \in \NN_0$ as follows from the definition of $S_0^\gamma(f)$.

Keeping the induction hypothesis in mind, it therefore suffices to show 
\[
 \binom{k + \sigma +1}{k} \cdot S_k^{\sigma +1}(f) - \binom{k + \sigma}{k-1} \cdot S_{k-1}^{\sigma +1}(f) = \binom{k + \sigma}{\sigma} \cdot S^\sigma_k(f) \quad \text{for every } k \geq 1.
\]
Plugging in the definition of the Cesaro means and using $\binom{k+\sigma + 1}{\sigma + 1} = \binom{k+\sigma + 1}{k}$ and 
$\binom{k+\sigma}{\sigma + 1} = \binom{k+\sigma}{k-1}$, we get 
\begin{align*}
&\norel \binom{k + \sigma +1}{k} \cdot S_k^{\sigma +1}(f) - \binom{k + \sigma}{k-1} \cdot S_{k-1}^{\sigma +1}(f) \\
 &=  \frac{ \binom{k + \sigma +1}{k} }{ \binom{k + \sigma +1}{\sigma + 1}} \cdot \sum_{j= 0}^k \left[\binom{k-j + \sigma + 1}{\sigma + 1} \mathrm{proj}_j(f)\right]
 - \frac{\binom{k+ \sigma}{k-1}}{\binom{k+ \sigma}{\sigma +1}} \cdot \sum_{j= 0}^{k-1} \left[\binom{k-j+\sigma}{\sigma +1}\mathrm{proj}_j(f) \right]  \\
 &= \mathrm{proj}_k(f)  + \sum_{j= 0}^{k-1} \left[ \binom{k-j+ \sigma +1}{\sigma + 1} - \binom{k-j + \sigma}{\sigma + 1}\right] \cdot \mathrm{proj}_j(f) \\
 \overset{(\ast)}&{=} \binom{k-k + \sigma}{\sigma} \cdot \mathrm{proj}_k(f)  + \sum_{j= 0}^{k-1} \binom{k-j + \sigma}{\sigma} \cdot \mathrm{proj}_j(f) \\
 &= \sum_{j= 0}^{k} \binom{k-j + \sigma}{\sigma} \cdot \mathrm{proj}_j(f)  \\
 &= \binom{k + \sigma}{\sigma} \cdot S^\sigma_k(f),
\end{align*}
using Pascal's rule
\[
\binom{n}{m} + \binom{n}{m-1} = \binom{n+1}{m} \quad \text{for every } n,m \in \NN \text{ with } m\leq n
\]
at the step marked with $(\ast)$.
This proves the claim. 
\end{proof}

\subsection{Proof of Lemma \ref{lem:epsbound}}
\label{sec:signproof}

We mention again that the proof is essentially taken from \cite[Lemma~6]{maiorov2010best}, with a few more details added. 
\begin{proof}[Proof of Lemma \ref{lem:epsbound}]
Let $g(x) \defeq 1 - \frac{1}{2}(1-2x)^2 \cdot \log_2(\ee)$. 
Then $g(1/2) = 1$ and $g(0) = 1 - \frac{1}{2} \log_2(\ee)  \leq \frac{1}{2}$. 
By the intermediate value theorem, we can thus choose $a \in \left(0, \frac{1}{2}\right)$ such that
\[
1 - \frac{1}{2}(1-2a)^2 \cdot \log_2(\ee) = \frac{47}{64}.
\]
Let $\pi \in \sgn(\Gamma)\paul{\subseteq E^m}$ be arbitrary and consider the set
\[
E_\pi \defeq \left\{ \eps \in E^m: \ \sum_{i=1}^m \abs{\eps_i - \pi_i} \leq 2am\right\}.
\]
Consider the bijection
\[
\varphi_\pi: \quad E^m \to E^m, \quad \left(\varphi_\pi(\eps)\right)_i = -\pi_i \cdot \eps_i = \begin{cases}\eps_i, & \text{if }\pi_i = -1, \\ - \eps_i, & \text{if }\pi_i = 1.\end{cases}
\]
Fix $i \in \{1,\dots,m\}$ \paull{for the moment}. 
If $\pi_i = 1$, we observe 
\[
\abs{\left(\varphi_\pi(\eps)\right)_i - \pi_i} = - \left(\varphi_\pi(\eps)\right)_i+ 1 = \eps_i + 1
\]
and if $\pi_i = -1$, we get
\[
\abs{\left(\varphi_\pi(\eps)\right)_i - \pi_i} = \left(\varphi_\pi(\eps)\right)_i - (-1) = \eps_i + 1.
\]
Since $\varphi_\pi$ is a bijection (for instance because $\varphi_\pi \circ \varphi_\pi = \mathrm{id}$), we get
\begin{align*}
\abs{E_\pi}  &= \abs{\varphi_\pi^{-1}(E_\pi)}=\abs{\left\{ \eps \in E^m: \ \varphi_\pi(\eps) \in E_\pi\right\}} = \abs{\left\{ \eps \in E^m: \ \sum_{i=1}^m \abs{\left(\varphi_\pi(\eps)\right)_i - \pi_i} \leq 2am\right\}} \\
&= \abs{\left\{ \eps \in E^m: \ \sum_{i=1}^m (\eps_i + 1) \leq 2am\right\}} 
= \abs{\left\{ \eps \in E^m: \ \underset{\eps_i = 1}{\sum_{i \in \{1,\dots,m\}}} 2 \leq 2am\right\}} \\
&= \abs{\left\{ \eps \in E^m: \ \underset{\eps_i = 1}{\sum_{i \in \{1,\dots,m\}}} 1 \leq \lfloor am \rfloor \right\}} 
= \sum_{j=0}^{\lfloor am \rfloor} \abs{\left\{ \eps \in E^m: \ \underset{\eps_i = 1}{\sum_{i \in \{1,\dots,m\}}} 1 =j  \right\}}.
\end{align*}
Each summand of the final sum can be viewed as the number of subsets of $\{1,\dots,m\}$ with cardinality $j$. 
We hence get
\[
\abs{E_\pi} = \sum_{j=0}^{\lfloor am \rfloor} \binom{m}{j}.
\]
This sum can be bounded using Hoeffding's inequality for bounded random variables, see for instance 
\cite[Theorem~2.2.6]{vershynin_high-dimensional_2018}.
To do so, we denote the binomial distribution with parameters $m$ and $p = 1/2$ by $\bin(m, 1/2)$, i.e., $m$ ``tries'', each with success probability $\frac{1}{2}$.
We then get
\begin{align*}
\sum_{j=0}^{\lfloor am \rfloor} \binom{m}{j} = 2^m \cdot \underset{X \sim \bin(m, 1/2)}{\PP} \left(X \leq \lfloor am\rfloor\right) 
& = 2^m \cdot \underset{\sigma_i \iid \unif(\{0,1\})}{\PP} \left(\sum_{i=1}^m \sigma_i \leq \lfloor am \rfloor\right) \\
&= 2^m \cdot  \underset{\sigma_i \iid \unif(\{0,1\})}{\PP} \left(\sum_{i=1}^m (-\sigma_i + 1/2) \geq m/2 - \lfloor am \rfloor\right), \\
\end{align*}
where we note that $m/2 -  \lfloor am \rfloor > 0$ since $a < 1/2$.
Since  $-1 \leq -\sigma_i \leq 0$ and $\EE[-\sigma_i] = - \frac{1}{2}$, Hoeffding's inequality implies
\begin{align*}
\underset{\sigma_i \iid \unif(\{0,1\})}{\PP} \left(\sum_{i=1}^m (-\sigma_i + 1/2) \geq m/2 - \lfloor am \rfloor\right)
&\leq \exp \left(- \frac{2 (m/2 - \lfloor am \rfloor)^2}{m}\right) \\
&= \exp \left(- 2m \cdot (1/2 - \beta)^2\right) \\
&= \exp\left(- \frac{1}{2}\cdot m \cdot (1 - 2\beta)^2\right) 
\end{align*}
with $\beta \defeq \frac{\lfloor am \rfloor}{m}$.
Applying the inequality gives us 
\begin{align*}
\abs{E_\pi} = \sum_{j=0}^{\lfloor am \rfloor} \binom{m}{j} \leq 2^m \cdot \exp\left(-\frac{1}{2} \cdot m  \cdot (1-2\beta)^2\right) 
&= 2^{m - \log_2(\ee) \cdot  \frac{1}{2}\cdot m \cdot (1-2\beta)^2} \\
&\leq 2^{m(1 -\log_2(\ee) \cdot \frac{1}{2} \cdot (1-2a)^2)}= 2^{m \cdot 47/64},
\end{align*}
where we used that $\beta \leq a \leq \frac{1}{2}$ implies $(1-2\beta)^2 \geq (1-2a)^2$. 
Moreover, note that 
\begin{align*}
\abs{E^m \setminus \bigcup_{\pi \in \sgn(\Gamma)} E_\pi} = \abs{E^m} - \abs{\bigcup_{\pi \in \sgn(\Gamma)} E_\pi} \geq \abs{E^m} - \abs{\sgn(\Gamma)}\cdot 2^{m \cdot 47/64} \geq 2^m - 2^{m \cdot 63/64}>0,
\end{align*}
where we used $\abs{\sgn(\Gamma)} \leq 2^{m/4}$ by assumption. 
We can thus choose $\eps^\ast \in E^m \setminus \bigcup_{\pi \in \sgn(\Gamma)} E_\pi $. 
By definition \paull{of the sets $E_\pi$}, we \paull{then} have 
\[
\underset{\pi \in \sgn(\Gamma)}{\inf} \onenorm{\eps^\ast - \pi} \geq 2am. 
\]
It remains to translate this estimate to the set $\Gamma$.
To this end, let $x \in \RR$ be arbitrary and $\eps \in \{-1,1\}$. 
If $\sgn(x) = \eps$, we get $\abs{x - \eps} \geq 0 = \frac{1}{2} \cdot \abs{\sgn(x)-\eps}$. 
If $\sgn(x) \neq \eps$, we get 
\[
\abs{x- \eps} \geq 1 = \frac{1}{2} \cdot 2 = \frac{1}{2} \cdot \abs{\sgn(x) - \eps }.
\]
In any case, we conclude
\[
\abs{x-\eps} \geq \frac{1}{2} \cdot \abs{\sgn(x) - \eps}.
\]
This yields
\[
\underset{x \in \Gamma}{\inf} \onenorm{\eps^\ast - x}  \geq \frac{1}{2} \cdot \underset{x \in \Gamma}{\inf}
\onenorm{\eps^\ast - \sgn(x)}=\frac{1}{2} \cdot \underset{\pi \in \sgn(\Gamma)}{\inf} \onenorm{\eps^\ast - \pi} \geq am,
\]
as desired. 
\end{proof}

\subsection{Proof of Lemma~\ref{lem:trig_power_reduction}}
\label{sec:proof_trig_power_reduction}

\begin{proof}[Proof of \Cref{lem:trig_power_reduction}]
From \cite[Appendix~I.1.9]{prudnikov1986integrals}, we get the existence of $\gamma_0 = \gamma_0(a), \dots,\gamma_a = \gamma_a(a) \in \RR$ as well as
$\delta_0 = \delta_0(b), \dots,\delta_b = \delta_b(b) \in \RR$ with 
\[
\cos(\varphi)^a = \sum_{h= 0}^a \gamma_h \cos(h\varphi) \quad \text{for all } \varphi \in \RR
\]
and 
\[
\sin(\varphi)^b = \begin{cases}\displaystyle\sum_{h= 0}^b \delta_h \cos(h\varphi), & \text{if }b \text{ even,}\\
\displaystyle\sum_{h= 0}^b \delta_h \sin(h\varphi), & \text{if }b \text{ odd} \end{cases} \quad \text{for all } \varphi \in \RR.
\]
If $b$ is even, we hence get 
\begin{equation}\label{eq:tbrearranged}
\cos(\varphi)^a \sin(\varphi)^b = \sum_{h_1= 0}^a \sum_{h_2 = 0}^b \ \gamma_{h_1} \delta_{h_2} \cdot \cos(h_1 \varphi) \cos(h_2 \varphi) .
\end{equation}
We can use the well-known product-to-sum formula for cosines (see \cite[Appendix~I.1.8]{prudnikov1986integrals}) and get 
\begin{align*}
\cos(h_1 \varphi) \cos(h_2 \varphi)  &= \frac{1}{2} \cdot \big(\cos((h_1 - h_2) \varphi) + \cos((h_1 + h_2) \varphi)\big).
\end{align*}
Since $\abs{h_1 \pm h_2} \leq h_1 + h_2 \leq a + b$ for $h_1 \in \{0,\dots,a\}$ and $h_2 \in \{0,\dots,b\}$ and by 
the symmetry of the cosine, we can thus rearrange \eqref{eq:tbrearranged}
and get 
\[
\cos(\varphi)^a \sin(\varphi)^b =  \sum_{h= 0}^{a+b} \alpha_h \cos(h\varphi) 
\]
for certain coefficients $\alpha_h = \alpha_h(a,b) \in \RR$. 
Similarly, in the case that $b$ is odd, one obtains 
\[
\cos(\varphi)^a \sin(\varphi)^b =  \sum_{h= 0}^{a+b} \beta_h \sin(h\varphi) 
\]
for coefficients $\beta_h = \beta_h(a,b) \in \RR$. 
This proves the claim. 
\end{proof}

\subsection{Postponed proofs for identities involving Wirtinger derivatives}\label{sec:wirt_proofs}
In this subsection, we provide rigorous proofs for the identities in \Cref{eq:ts1,eq:ts2}.
To this end, we need the following well-known properties of Wirtinger derivatives that can be found for example in \cite[E.1a]{kaup_holomorphic_1983}.
Here, we assume that $U \subseteq \CC$ is open and $f \in C^1(U; \CC)$.
\begin{enumerate}
    \item $\wirt$ and $\wirtq$ are both $\CC$-linear operators on the set $C^1(U; \CC)$.
    \item \label{wirtdiff} $f$ is complex-differentiable at $z \in U$ iff $\wirtq f(z) = 0$ and in this case the equality
    \begin{equation*}
        \wirt f(z) = f'(z)
    \end{equation*}
    holds true, with $f'(z)$ denoting the complex derivative of $f$ at $z$. 
    \item \label{wirtconj} We have the conjugation rules
    \begin{align*}
        \overline{\wirt f} = \wirtq \overline{f} \quad \text{and} \quad \overline{\wirtq f} = \wirt \overline{f}.
    \end{align*}
    \item \label{wirtprod} If $g \in C^1(U; \CC)$, the following product rules for Wirtinger derivatives hold for every $z \in U$:
    \begin{align*}
        \wirt (f \cdot g)(z)&= \wirt f (z) \cdot g(z) + f(z) \cdot \wirt g(z), \\
        \wirtq (f \cdot g)(z)&= \wirtq f (z) \cdot g(z) + f(z) \cdot \wirtq g(z).
    \end{align*}
    This product rule is not explicitly stated in \cite{kaup_holomorphic_1983} but follows easily from the product rule for $\frac{\partial}{\partial x}$ and $\frac{\partial}{\partial y}$.
    \item \label{wirtchain} If $V \subseteq \CC$ is an open set and $g \in C^1(V;\CC)$ with $g(V) \subseteq U$, then the following chain rules for Wirtinger derivatives hold true:
    \begin{align*}
        \wirt(f \circ g) &= \left[(\wirt f) \circ g\right]\cdot \wirt g + \left[\left(\wirtq f\right)\circ g\right]\cdot \wirt \overline{g}, \\
        \wirtq(f \circ g) &= \left[(\wirt f) \circ g\right]\cdot \wirtq g + \left[\left(\wirtq f\right)\circ g\right]\cdot \wirtq \overline{g}.
    \end{align*}
\end{enumerate}

Using these properties, we can now prove \eqref{eq:ts1}.

\begin{lemma}\label{lem:wirtiden1}
Let $\kk, \elll, \kk', \elll' \in \NN_0^d$ be multiindices with 
$\abs{\kk} = \abs{\kk'}$ and $\abs{\elll} = \abs{\elll'}$.
Then it holds 
\[
\wirt^{\kk}\wirtq^{\elll}(z^{\kk'}\overline{z}^{\elll'}) = \mathbbm{1}_{(\kk, \elll) = (\kk', \elll') } \cdot \kk ! \cdot \elll !.
\]
\end{lemma}
\begin{proof}
We start by showing that 
\begin{equation}\label{eq:1st}
\wirt^{\kk}\wirtq^{\elll}(z^{\kk}\overline{z}^{\elll}) = \kk ! \cdot \elll !.
\end{equation}
Firstly, assume that $\elll = 0$. 
We show via induction over $\abs{\kk}$ that the identity 
\[
\wirt^{\kk} z^{\kk} = \kk !
\]
holds. 
There is nothing to show in the case $\kk = 0$. 
Therefore, assume that the claim holds for a fixed $\kk \in \NN_0^d$ and let $j \in \{1,\dots,d\}$ be arbitrary. 
Then we get 
\[
\wirt^{\kk + e_j} z^{\kk + e_j} = \wirt^{\kk} \wirt^{e_j} \left[  z^{\kk} \cdot z_j\right]. 
\]
Hence, for fixed variables $z_1, \dots, z_{j-1}, z_{j+1},\dots, z_d$, we consider $\wirt [z_j \mapsto z^{\kk} \cdot z_j]$.
From the linearity of $\wirt$ we deduce 
\[
\wirt [z_j \mapsto z^{\kk} \cdot z_j] = z^{\kk - \kk_j \cdot e_j} \cdot \wirt \left[z_j \mapsto z_j^{\kk_j + 1}\right] = (\kk_j + 1) \cdot z^{\kk - \kk_j \cdot e_j} \cdot z_j^{\kk_j}
= (\kk_j + 1) \cdot z^{\kk}.
\]
Here, we used the fact that $z_j \mapsto z_j^{\kk_j + 1}$ is holomorphic and the fact that $\wirt$ coincides with the regular complex derivative (see \eqref{wirtdiff}) in this case. 
This gives us 
\[
\wirt^{\kk + e_j} z^{\kk + e_j} = \wirt^{\kk} \wirt^{e_j} \left[  z^{\kk} \cdot z_j\right] = (\kk_j + 1) \cdot \wirt^{\kk} z^{\kk} = (\kk + e_j)!  
\]
according to the induction hypothesis. 
By induction, we have thus shown 
\[
\wirt^{\kk} z^{\kk} = \kk ! \quad \text{for every $\kk \in \NN_0^d$.}
\]
In order to get the full claim, we now perform an additional induction over $\abs{\elll}$ and show that for every $\elll \in \NN_0^d$ the statement 
\[
\fall \kk \in \NN_0^d: \quad \wirt^{\kk}\wirtq^{\elll}(z^{\kk}\overline{z}^{\elll}) = \kk ! \cdot \elll !
\]
holds. 
Again for $\elll = 0$, the statement is already proven above.  
Therefore, we pick $j \in \{1,\dots,d\}$ arbitrary and assume that $\elll \in \NN_0^d$ satisfies the claim. 
We then let $\kk \in \NN_0^d$ be arbitrary and note
\[
\wirt^\kk \wirtq^{\elll + e_j}(z^\kk \overline{z}^{\elll + e_j})  
= \wirt^\kk \wirtq^{\elll} \wirtq^{e_j}(z^{\kk - \kk_j \cdot e_j} \overline{z}^{\elll - \elll_j \cdot e_j} \cdot z_j^{\kk_j} \cdot \overline{z_j}^{\elll_j + 1}).
\]
We next note that 
\[
\wirtq^{e_j}(z^{\kk - \kk_j \cdot e_j} \overline{z}^{\elll - \elll_j \cdot e_j} \cdot z_j^{\kk_j} \cdot \overline{z_j}^{\elll_j + 1})
= \wirtq \left[ z_j \mapsto z^{\kk - \kk_j \cdot e_j} \overline{z}^{\elll - \elll_j \cdot e_j} \cdot z_j^{\kk_j} \cdot \overline{z_j}^{\elll_j + 1}\right].
\]
Again, for fixed variables $z_1, \dots, z_{j-1}, z_{j+1},\dots, z_d$, due to linearity we get 
\[
\wirtq \left[ z_j \mapsto z^{\kk - \kk_j \cdot e_j} \overline{z}^{\elll - \elll_j \cdot e_j} \cdot z_j^{\kk_j} \cdot \overline{z_j}^{\elll_j + 1}\right]
= z^{\kk - \kk_j \cdot e_j} \overline{z}^{\elll - \elll_j \cdot e_j} \cdot \wirtq\left[ z_j \mapsto z_j^{\kk_j} \cdot \overline{z_j}^{\elll_j + 1}\right].
\]
Using the product rule \eqref{wirtprod}, we get 
\begin{align*}
\wirtq\left[ z_j \mapsto z_j^{\kk_j} \cdot \overline{z_j}^{\elll_j + 1}\right] &= \wirtq[z_j \mapsto z_j^{\kk_j}] \cdot \overline{z_j}^{\elll_j + 1} + 
z_j^{\kk_j} \cdot \wirtq[z_j \mapsto \overline{z_j}^{\elll_j + 1}].
\end{align*}
Note that the first summand vanishes since $z_j \mapsto z_j^{\kk_j}$ is holomorphic (see \eqref{wirtdiff}).
Moreover, we get 
\[
\wirtq[z_j \mapsto \overline{z_j}^{\elll_j + 1}] \overset{\eqref{wirtconj}}{=} \overline{\wirt[z_j \mapsto z_j^{\elll_j + 1}]}
\overset{\eqref{wirtdiff}}{=} (\elll_j + 1) \cdot \overline{z_j}^{\elll_j}.
\]
All in all, this yields that
\begin{align*}
\wirt^\kk \wirtq^{\elll + e_j}(z^\kk \overline{z}^{\elll + e_j}) = (\elll_j + 1) \cdot \wirt^{\kk} \wirtq^{\elll} [z^\kk \overline{z}^{\elll}] = \kk ! \cdot (\elll + e_j)!,
\end{align*}
where the induction hypothesis was used in the last equality. 
By induction, this proves \eqref{eq:1st}. 

It remains to show that 
\[
\wirt^{\kk}\wirtq^{\elll}(z^{\kk'}\overline{z}^{\elll'}) = 0 \quad \text{if $(\kk, \elll) \neq (\kk', \elll')$.}
\]
Without loss of generality we assume $\kk \neq \kk'$ (the case $\elll \neq \elll'$ follows analogously). 
Since $\abs{\kk} = \abs{\kk'}$, we conclude the existence of an index $j \in \{1,\dots,d\}$ with $\kk_j > \kk'_j$. 
Using the commutativity of partial derivatives, it suffices to show that
\[
\wirt^{\kk_j \cdot e_j}\left[z^{\kk'}\overline{z}^{\elll'} \right]  = 0
.
\]
Since 
\[
\wirt^{\kk_j \cdot e_j}\left[z^{\kk'}\overline{z}^{\elll'} \right] 
= z^{\kk' - \kk'_j \cdot e_j} \cdot \overline{z}^{\elll' - \elll'_j \cdot e_j}\cdot \wirt^{\kk_j}\left[ z_j \mapsto z_j^{\kk'_j} \overline{z_j}^{\elll'_j}  \right],
\]
it suffices to show that 
\[
\wirt^{\kk_j}\left[ z_j \mapsto z_j^{\kk'_j} \overline{z_j}^{\elll'_j}  \right] = 0.
\]
From an iterated application of the product rule \eqref{wirtprod} we get 
\[
\wirt^{\kk'_j}\left[ z_j \mapsto z_j^{\kk'_j} \overline{z_j}^{\elll'_j}  \right] = (\kk'_j)! \cdot \overline{z_j}^{\elll'_j}.
\]
Then we get the claim by noting that 
\[
\wirt \left[ \overline{z_j}^{\elll'_j}\right] \overset{\eqref{wirtconj}}{=} \overline{\wirtq \left[ z_j^{\elll'_j}\right] } \overset{\eqref{wirtdiff}}{=} 0. \qedhere
\]
\end{proof}

We continue by proving \Cref{eq:ts2}.
\begin{lemma}\label{lem:wirtiden2}
For a fixed vector $a \in \CC^d$ and multiindices $\kk, \elll \in \NN_0^d$ with $\abs{\kk}=s$ and $\abs{\elll} = t$
for $s,t \in \NN_0$, it holds 
\[
\wirt^{\kk} \wirtq^{\elll}\left((a^T z)^s(\overline{a^T z})^{t}\right) = s! \cdot t! \cdot a^{\kk} \overline{a}^{\elll}.
\]
\end{lemma}
\begin{proof}
Similar to the proof of \Cref{eq:ts1}, we first perform induction over $s$ and then over $t$ to obtain the claim. 
To begin, let $t = 0$ (i.e., $\elll = 0$) and assume that for all $\kk \in \NN_0^d$ with $\abs{\kk} = s$ we have 
\[
\wirt^{\kk}\left((a^Tz)^s\right) = s! \cdot a^{\kk}.
\]
Let $j \in \{1,\dots,d\}$ be arbitrary and consider
\[
\wirt^{\kk + e_j}\left((a^Tz)^{s+1}\right) = \wirt^{\kk} \wirt \left[ z_j \mapsto (a^Tz)^{s+1}\right].
\]
Using the chain rule \eqref{wirtchain}, we get 
\begin{align*}
\wirt \left[ z_j \mapsto (a^Tz)^{s+1}\right] = (s+1) \cdot (a^Tz)^s \cdot \wirt[z_j \mapsto a^T z].
\end{align*}
Here, we also used \eqref{wirtdiff} and the fact that $w \mapsto w^{s+1}$ is holomorphic, which implies the two properties $\wirt[w \mapsto w^{s+1}] = (s+1) \cdot w^s$ and 
$\wirtq [w \mapsto w^{s+1}] = 0$.
Moreover, since $z_j \mapsto a^T z$ is holomorphic too, we see that $\wirt[z_j \mapsto a^T z] = a_j$.
Overall we get 
\[
\wirt^{\kk + e_j}\left((a^Tz)^{s+1}\right) = (s+1) \cdot a_j \cdot \wirt^{\kk}\left[(a^Tz)^s\right] = (s+1)! \cdot a^{\kk + e_j},
\]
where the induction hypothesis was used for the last equality. 
Since the case $\kk = 0$ (and hence $s= 0$) is trivial, we conclude by induction that
\[
\wirt^{\kk}\left((a^Tz)^s\right) = s! \cdot a^{\kk} \quad \text{for all } \kk \in \NN_0^d \text{ with } \abs{\kk}=s.
\]

We now perform an additional induction over $\abs{\elll} = t$ to show that 
\[
\fall \kk \in \NN_0^d \text{ with } \abs{\kk} = s:\quad \wirt^{\kk} \wirtq^{\elll}\left((a^T z)^s(\overline{a^T z})^{t}\right) = s! \cdot t! \cdot a^{\kk} \overline{a}^{\elll}.
\]
Note that the case $\elll = 0$ (and hence $t= 0$) is shown above.
Therefore, we may assume that the claim holds for all $\elll \in \NN_0^d$ with $\abs{\elll} = t$, where $t \in \NN_0$ is arbitrary but fixed. 
We let $j \in \{1,\dots,d\}$ and $\kk \in \NN_0^d$ with $\abs{\kk} = s$ be arbitrary and consider
\[
\wirt^{\kk} \wirtq^{\elll + e_j}\left((a^T z)^s(\overline{a^T z})^{t + 1}\right) = \wirt^{\kk} \wirtq^{\elll} \wirtq^{e_j}\left((a^T z)^s(\overline{a^T z})^{t + 1}\right).
\]
Note that 
\[
\wirtq^{e_j}\left((a^T z)^s(\overline{a^T z})^{t + 1}\right) = \wirtq\left[ z_j \mapsto (a^T z)^s(\overline{a^T z})^{t + 1}\right].
\]
The product rule \eqref{wirtprod} yields 
\begin{align*}
 \wirtq\left[ z_j \mapsto (a^T z)^s(\overline{a^T z})^{t + 1}\right] &= 
 \wirtq \left[ z_j \mapsto (a^T z)^s\right] \cdot (\overline{a^T z})^{t + 1} + (a^T z)^s \cdot \wirtq\left[ z_j \mapsto (\overline{a^T z})^{t + 1}\right].
\end{align*}
Note that $\wirtq \left[ z_j \mapsto (a^T z)^s\right] = 0$ according to \eqref{wirtdiff}.
Moreover, \eqref{wirtconj} and an application of the chain rule \eqref{wirtchain} similar to above gives us 
\[
\wirtq\left[ z_j \mapsto (\overline{a^T z})^{t + 1}\right] = \overline{\wirt \left[ z_j \mapsto (a^T z)^{t + 1}\right]}
= (t+1) \cdot \overline{a_j} \cdot  (\overline{a^Tz})^t.
\]
Putting everything together and using the induction hypothesis, we get 
\begin{align*}
\wirt^{\kk} \wirtq^{\elll + e_j}\left((a^T z)^s(\overline{a^T z})^{t + 1}\right) = (t+1) \cdot \overline{a_j} \cdot \wirt^{\kk}\wirtq^{\elll}\left((a^T z)^s(\overline{a^T z})^{t}\right)
= (t+1)! \cdot s! \cdot a^{\kk} \overline{a}^{\elll + e_j}.
\end{align*}
The principle of induction thus yields the claim. 
\end{proof}

\medskip

\textbf{Acknowledgements.} 
The authors are grateful to Frank Filbir and Hrushikesh Mhaskar for several helpful and interesting discussions \paull{and for pointing us to the reference \cite{dai2013approximation}}.
FV and PG are grateful for the hospitality of Utrecht University, where this project was initiated during their visit.   
FV and PG acknowledge support by
the German Science Foundation (DFG) in the context of the Emmy Noether junior research
group VO 2594/1-1.
FV acknowledges support by the Hightech Agenda Bavaria. 
PS is supported by NWO Talent programme Veni ENW grant, file number VI.Veni.212.176.
\printbibliography
\end{document}